\documentclass{amsart}
\usepackage{color}
\usepackage{amsmath}
\usepackage{kotex}
\usepackage{amssymb}
\usepackage{amsthm}
\usepackage{amscd}
\usepackage{bm}
\usepackage{eucal} 
\usepackage{enumerate}
\usepackage{tikz}
\usepackage{tikz-cd}

\usetikzlibrary{positioning, fit, calc}
\usepackage{shuffle}
\usepackage[utf8]{inputenc}
\usepackage[all]{xy}
\allowdisplaybreaks

\usepackage{hyperref}
\hypersetup{
  colorlinks   = true,          
  urlcolor     = blue,          
  linkcolor    = blue,          
  citecolor   = red             
}

\theoremstyle{plain}
\newtheorem{theorem}{Theorem}[section]
\newtheorem{conjecture}[theorem]{Conjecture}
\newtheorem{lemma}[theorem]{Lemma}
\newtheorem{proposition}[theorem]{Proposition}

\newtheorem{theoremx}{Theorem}

\theoremstyle{definition}
\newtheorem{definition}[theorem]{Definition}
\newtheorem{remark}[theorem]{Remark}

\numberwithin{equation}{section}

\newcommand{\ppar}{${}$\par}

\newcommand\fantome[1]{}

\def\bL{\mathbb L}

\def\bT{\mathbb T}

\def\Fq{\mathbb F_q}

\DeclareMathOperator{\Init}{Init}

\DeclareMathOperator{\Li}{Li}
\DeclareMathOperator{\Si}{Si}

\newcommand{\F}{\mathbb{F}}
\newcommand{\C}{\mathbb{C}}

\newcommand{\bff}{\mathbf{f}}
\newcommand{\bg}{\mathbf{g}}

\newcommand{\bh}{\mathbf{h}}
\newcommand{\bv}{\mathbf{v}}

\newcommand{\fm}{\bm{\mu}}
\newcommand{\fs}{\mathfrak{s}}
\newcommand{\fe}{\bm{\epsilon}}
\newcommand{\fve}{\bm{\varepsilon}}
\newcommand{\frakL}{\mathfrak{L}}

\newcommand{\N}{\ensuremath \mathbb{N}}

\DeclareMathOperator{\Mat}{Mat}

\DeclareMathOperator{\depth}{depth}

\author[B.-H. Im]{Bo-Hae Im}
\address{
Dept. of Mathematical Sciences, KAIST,
291 Daehak-ro, Yuseong-gu,
Daejeon 34141, South Korea
}
\email{bhim@kaist.ac.kr}

\author[H. Kim]{Hojin Kim}
\address{
Dept. of Mathematical Sciences, KAIST,
291 Daehak-ro, Yuseong-gu,
Daejeon 34141, South Korea
}
\email{hojinkim@kaist.ac.kr}

\author[K. N. Le]{Khac Nhuan Le}
\address{
Normandie Université,
Université de Caen Normandie - CNRS,
Laboratoire de Mathématiques Nicolas Oresme (LMNO), UMR 6139,
14000 Caen, France.
}
\email{khac-nhuan.le@unicaen.fr}

\author[T. Ngo Dac]{Tuan Ngo Dac}
\address{
Normandie Université,
Université de Caen Normandie - CNRS,
Laboratoire de Mathématiques Nicolas Oresme (LMNO), UMR 6139,
14000 Caen, France.
}
\email{tuan.ngodac@unicaen.fr}

\author[L. H. Pham]{Lan Huong Pham}
\address{
Institute of Mathematics, Vietnam Academy of Science and Technology, 18 Hoang Quoc Viet, 10307 Hanoi, Viet Nam
}
\email{plhuong@math.ac.vn}

\title[Cyclotomic multiple zeta values]{Zagier-Hoffman's conjectures in positive characteristic II}



\date{\today}

\begin{document}

\begin{abstract}
Zagier-Hoffman's conjectures predict the dimension and a basis for the $\mathbb Q$-vector spaces spanned by $N$th cyclotomic multiple zeta values (MZV's) of fixed weight where $N$ is a natural number. 

For $N=1$ (MZV's case), half of these conjectures have been solved by the work of Terasoma, Deligne-Goncharov and Brown with the help of Zagier's identity. The other half are completely open. For $N=2$ (alternating MZV's case) and $N=3,4,8$, Deligne-Goncharov and Deligne solved the same half of these conjectures for $N$th-cyclotomic MZV's. For other values of $N$, no sharp upper bound on the dimension is known.

In this paper we completely establish, for all $N$, Zagier-Hoffman's conjectures for $N$th cyclotomic multiple zeta values  in positive characteristic. By working with the tower of all cyclotomic extensions, we present a proof that is uniform on $N$ and give an effective algorithm to express any cyclotomic multiple zeta value in the chosen basis. This generalizes all previous work on these conjectures for MZV's and alternating MZV's in positive characteristic.
\end{abstract}

\maketitle

\tableofcontents


\section*{Introduction}

\subsection{Classical multiple zeta values} \ppar

Let $\mathbb N = \{1, 2, \dots\}$ be the set of positive integers. Multiple zeta values (MZV's for short) are the convergent series given by
\[ \zeta(n_1, \dots, n_r) = \sum_{0 < k_1 < \dots < k_r} \frac{1}{k_1^{n_1} k_2^{n_2} \dots k_r^{n_r}},\]
for $n_i \in \mathbb{N}$ with $n_r \ge 2$. The MZV's of depth $1$ and $2$ were first studied by Euler  \cite{Eul76} in the 18th century and then neglected for more than two centuries. In the nineties of the last century, Zagier~\cite{Zag94} rediscovered the general MZV's and initiated intensive work related to various fields of mathematics including arithmetic geometry, number theory, $K$-theory and knot theory (see for example \cite{Bro12,DG05,Dri90,GKZ06,Hof97,Kon99,Kon03,IKZ06,Iha89,Iha91,LM96,Ter02,Ter06}). At the same time, MZV's were suddenly appeared in various branches of physics, such as in the theory of quantum field theory and deformation quantization (see e.g., \cite{Bro96a,Bro09,BK97,Kon99,Kon03}). 

There exist more general objects called cyclotomic multiple zeta values (cyclotomic MZV's for short). Letting $N \in \mathbb N$ be a positive integer and $\Gamma_N$ be the group of $N$th roots of unity, we define cyclotomic multiple zeta values by
	\[ \zeta \begin{pmatrix}
	\epsilon_1 & \dots & \epsilon_r \\
	n_1 & \dots & n_r
	\end{pmatrix}=\sum_{0<k_1<\dots<k_r} \frac{\epsilon_1^{k_1} \dots \epsilon_r^{k_r}}{k_1^{n_1} \dots k_r^{n_r}} \]
for $\epsilon_i \in \Gamma_N$ and $n_i \in \N$ with $(n_r,\epsilon_r) \neq (1,1)$ for convergence. These objects were first studied systematically by Arakawa, Deligne, Goncharov, Kaneko and Racinet (see \cite{AK99, AK04, Del10, DG05, Gon95,Gon01, Gon05, Rac02}). Since then they have been also studied by a lot of mathematicians and physicists, including Broadhurst, Hoffman, Kreimer, Kaneko, Tsumura \cite{Bro96a,BK97,Hof19,KT13} in various fields of mathematics and physics. For $N=1$ we recover MZV's. For $N=2$ the cyclotomic multiple zeta values are also known as alternating multiple zeta values or Euler sums.

Significant progress has been made, but fundamental questions remain and cyclotomic multiple zeta values are still an active fast-moving field (see for example \cite{BPP20,CDPP24,Li24,Mur22}).

As explained in \cite{BGF,Del13,Zag94}, the main goal of the theory of cyclotomic multiple zeta values is to find all linear relations among cyclotomic MZV's over $\mathbb Q$. It led to several important conjectures formulated by Ihara-Kaneko-Zagier \cite{IKZ06}, Zagier \cite{Zag94} and Hoffman \cite{Hof97}. In this paper we work only with Zagier and Hoffman's conjectures which predict the dimension and a basis for the span of cyclotomic MZV's of fixed weight. We refer the reader to the excellent articles \cite{Bro14,BGF,Del13,Ter06} for more details on Ihara-Kaneko-Zagier's, Zagier-Hoffman's conjectures and related topics. We also refer the reader to the articles of Broadhurst and his colleagues \cite{Bro96a,Bro14b,BBV10} for a data mine of cyclotomic MZV's and further references.

\subsubsection{$N=1$ (the case of MZV's)} \ppar

We work with the case $N=1$, or equivalently with MZV's. Letting $\mathcal Z_k$ be the $\mathbb Q$-vector space spanned by MZV's of weight $k$, we consider a Fibonacci-like sequence of integers $d_k$ as follows. Letting $d_0=1, d_1=0$ and $d_2=1$ we define $d_k=d_{k-2}+d_{k-3}$ for $k \geq 3$. Using numerical evidence Zagier \cite{Zag94} formulated the following conjecture:
\begin{conjecture}[Zagier's conjecture]
For $k \in \N$ we have $\dim_{\mathbb Q} \mathcal Z_k = d_k$.
\end{conjecture}

Hoffman \cite{Hof97} went further and suggested a refinement of Zagier's conjecture.
\begin{conjecture}[Hoffman's conjecture]
The $\mathbb Q$-vector space $\mathcal Z_k$ is spanned by the basis consisting of MZV's of weight $k$ of the form $\zeta(n_1,\dots,n_r)$ with $n_i \in \{2,3\}$.
\end{conjecture}

In the last two decades one ``half'' of these conjectures (called the algebraic part in \cite{ND21}) has been completely solved by the seminal works of Brown \cite{Bro12}, Deligne-Goncharov \cite{DG05} and Terasoma \cite{Ter02}. Although Zagier-Hoffman's conjectures are easily stated, the proofs of Brown, Deligne-Goncharov and Terasoma use the theory of mixed Tate motives and also an intriguing identity due to Zagier \cite{Zag12}.

\begin{theorem}[Deligne-Goncharov, Terasoma]
For $k \in \N$ we have $\dim_{\mathbb Q} \mathcal Z_k \leq d_k$.
\end{theorem}

\begin{theorem}[Brown]
For all $k \in \N$, the $\mathbb Q$-vector space $\mathcal Z_k$ is spanned by MZV's of weight $k$ of the form $\zeta(n_1,\dots,n_r)$ with $n_i \in \{2,3\}$. 
\end{theorem}

We mention that the second half of these conjectures called the transcendental part in \cite{ND21} is completely out of reach: we do not know any $k \in \mathbb N$ such that $\dim_{\mathbb Q} \mathcal Z_k >1$.

\subsubsection{$N=2$ (the case of alternating MZV's) and $N=3,4,8$} \label{sec: Deligne} \ppar

In the seminal paper \cite{Del10} Deligne succeeded in proving the same half of Zagier-Hoffman's conjectures for cyclotomic MZV's for $N=2,3,4,8$ (see \cite[Théorèmes 6.1 and 7.2]{Del10}). For these few exceptional values, he was able to use the theory of mixed Tate motives and obtained a generating set for the $\mathbb Q$-span of $N$th cyclotomic MZV's of fixed weight. Thus the upper bound obtained in \cite{DG05} for $N=2,3,4,8$ is sharp.  As mentioned in {\it loc. cit.} Deligne's result was inspired by numerical computations due to Broadhurst \cite{Bro96a} in 1996 related to computations of Feynman integrals. Later Glanois \cite{Gla16} gave another proof of Deligne's result.

A technical but crucial point that we want to emphasize is that Deligne used the depth filtration for $N=2,3,4,8$ while Brown used the level filtration for $N=1$. It has already been noted that the depth is pathological for $N=1$ (see \cite[page 102, lines 23-25]{Del10}).

\subsubsection{Other values of $N$} \ppar

For other values of $N$ Deligne and Goncharov \cite{DG05} proved an upper bound on the dimension of the span of $N$th cyclotomic MZV's. Unfortunately, for many $N$ Goncharov \cite{Gon01b} showed that these bounds are not sharp. To our knowledge, no sharp upper bound on the dimensions is known. We mention that for $N=6$ partial results have been obtained by Deligne \cite[Théorème 8.9]{Del10} and also by Glanois \cite{Gla16}.

\subsection{Multiple zeta values in positive characteristic} \ppar

By a well-known analogy between number fields and function fields (see for example \cite{Iwa69,MW83,Wei39}), we switch to the setting of function fields in positive characteristic. Let $A=\Fq[\theta]$ be the polynomial ring in the variable $\theta$ over a finite field $k:=\Fq$ of $q$ elements of characteristic $p>0$. We denote by $A_+$ the set of monic polynomials in $A$.  Let $K=\Fq(\theta)$ be the fraction field of $A$ equipped with the rational point $\infty$. Let $K_\infty=k(\!(1/\theta)\!)$ be the completion of $K$ at $\infty$ and $\C_\infty$ be the completion of a fixed algebraic closure $\overline K$ of $K$ at $\infty$.

We fix $N \in \mathbb N$ and denote by $k_N \subset \bar k$ the cyclotomic field over $k$ generated by a primitive $N$th root $\zeta_N$ of unity. The group of $N$th roots of unity is denoted by $\Gamma_N$ whose cardinality is denoted by $\gamma_N$. We put $A_N=k_N[\theta]$, $K_N = k_N(\theta)$ and $K_{N,\infty}=k_N(\!(1/\theta)\!)$.

Based on the pioneering work of Carlitz \cite{Car35} and Thakur \cite{Tha04}, Harada \cite{Har21,Har23} introduced the notion of $N$th cyclotomic multiple zeta values in positive characteristic. Letting $\fs=(s_1,\dots,s_r) \in \N^r$ and $\fe=(\epsilon_1,\dots,\epsilon_r) \in (\Gamma_N)^r$ we introduce
\begin{equation*}
    \zeta_A \begin{pmatrix}
 \fe  \\
\fs  \end{pmatrix}  = \sum \dfrac{\epsilon_1^{\deg a_1} \dots \epsilon_r^{\deg a_r }}{a_1^{s_1} \dots a_r^{s_r}}  \in K_{N,\infty},
\end{equation*}
where the sum is over all tuples
$(a_1,\ldots,a_r) \in A_+^r$ with $\deg (a_1)>\cdots>\deg (a_r)$. We define the depth and weight of $ \zeta_A \begin{pmatrix}
 \fe  \\
\fs  \end{pmatrix}$ in a similar way, that is, $r$ is the depth and $w(\fs) = s_1 + \dots + s_r$ is the weight. When $N=1$ (resp. $N=1$ and $r=1$) we recover the notion of MZV's defined by Thakur (resp. zeta values introduced by Carlitz). For further references on the MZV's in positive characteristic, see \cite{AT90,AT09,Gos96,Pel12,Tha04,Tha09,Tha10,Tha17,Yu91}.

In \textit{loc. cit.}, Harada also showed some fundamental properties of cyclotomic MZV's including non-vanishing, sum-shuffle relations, period interpretation and linear independence.

As in the classical setting the main goal of the theory of cyclotomic MZV's in positive characteristic is to comprehend all of their $\overline{K}$-linear relations. Harada \cite{Har23} showed that it is equivalent to understand all $K_N$-linear relations among cyclotomic MZV's of the same weight. Letting $w \in \mathbb N$ we denote by $\mathcal{CZ}_{N,w}$ the $K_N$-span of $N$th cyclotomic MZV's of weight $w$. Then the positive characteristic version of Zagier-Hoffman's conjectures for cyclotomic MZV's consist in determining the dimension and finding an explicit basis of $\mathcal{CZ}_{N,w}$.

\subsection{Statement of the main theorems} \ppar

\subsubsection{Main results} \ppar

In this paper we prove, for all $N$, Zagier-Hoffman's conjectures in positive characteristic for $N$th cyclotomic MZV's. This could be seen as the final part of the series of work previously done by the fourth author in \cite{ND21} and the authors in \cite{IKLNDP22}.

The main result reads as follows:

\begin{theoremx}[Hoffman's conjecture in positive characteristic] \label{thm: Hoffman cyclotomic MZV}
Let $N \in \mathbb N$. For $w \in \mathbb N$ we recall that $\mathcal{CZ}_{N,w}$ denotes the $K_N$-span of $N$th cyclotomic MZV's of weight $w$. Let $\mathcal{CS}_{N,w}$ be the set of Carlitz multiple polylogarithms at $N$th roots of unity $\Li \begin{pmatrix}
 \epsilon_1 & \dotsb & \epsilon_n \\
s_1 & \dotsb & s_n \end{pmatrix}$ such that $q \nmid s_i$ and $\epsilon_i \in \Gamma_N$ for all $i$ (see \S \ref{sec: CMPL} for more details). Then the set $\mathcal{CS}_{N,w}$ forms a basis for $\mathcal{CZ}_{N,w}$.
\end{theoremx}

As a direct consequence, we obtain
\begin{theoremx}[Zagier's conjecture in positive characteristic] \label{thm: Zagier cyclotomic MZV}
We keep the above notation. We recall that $\gamma_N=|\Gamma_N|$ and define a Fibonacci-like sequence $d_N(w)$ as follows. We put
\begin{equation*}
    d_N(w) = \begin{cases}
    		1 & \text{if } w=0, \\
		\gamma_N (\gamma_N+1)^{w-1}& \text{if } 1 \leq w < q, \\
         \gamma_N ((\gamma_N+1)^{w-1} - 1) &\text{if } w = q,
		 \end{cases}
\end{equation*}
and for $w>q$, $d_N(w)=\gamma_N \sum \limits_{i = 1}^{q-1} d_N(w-i) + d_N(w - q)$. Then 
	\[ \dim_{K_N} \mathcal{CZ}_{N,w} = d_N(w). \]
\end{theoremx}

Surprisingly, our proof of Theorem \ref{thm: Hoffman cyclotomic MZV} is uniform on $N$ in contrast to the classical setting, where the proofs of known results differ on $N$ as we have explained before (see for example \S \ref{sec: Deligne} and \cite{Bro12,Bro14,Del10} for more details). To do this we have to work with the whole tower of cyclotomic extensions of the function field $K$.

Perhaps more importantly, our proof is constructive. The proof gives an effective algorithm for expressing any cyclotomic MZV in the basis $\mathcal{CS}_{N,w}$ as opposed to the classical setting (see \cite{Bro12b,Del10,Del13} for more explanations).

Our main result generalizes several previous works of the authors in \cite{IKLNDP22,ND21}. However, compared to \cite{ND21} we still have to overcome several notable challenges, which we will highlight. First, the naive extension of \cite[\S 2--3]{ND21} does not give a sharp upper bound due to the presence of non-trivial characters. Thus, we must develop a refinement of \cite[\S 2--3]{ND21} that takes the characters into account. Second, the transcendental techniques developed in \cite[\S 4--6]{ND21} deal only with Anderson $t$-motives of level $1$ and do not work for arbitrary $N$. We use Harada's construction of higher level Anderson $t$-motives connected to cyclotomic MZV's, where the level is equal to the degree of the $N$th cyclotomic extension of $\Fq$. Thus we have to develop a higher level transcendental theory that generalizes \cite[\S 4--6]{ND21}. Finally, another original ingredient is that we use a descent argument on the tower of all cyclotomic extensions of $K$ (compared to \cite{Del10,Gla16} in the classical setting).

\subsubsection{Known results} \ppar
Below we present the list of all known cases where Zagier-Hoffman's conjecture
in positive characteristic for $N$th cyclotomic MZV's. As noted above, it suffices to prove Theorem \ref{thm: Hoffman cyclotomic MZV} as Theorem \ref{thm: Hoffman cyclotomic MZV} implies Theorem \ref{thm: Zagier cyclotomic MZV}.

\subsection*{$N=1$ (the case of MZV's)} \ppar
\begin{itemize}
\item Theorem \ref{thm: Zagier cyclotomic MZV} (resp. a variant of Theorem \ref{thm: Hoffman cyclotomic MZV}) was formulated by Todd \cite{Tod18} (resp. Thakur \cite{Tha17}). 

\item In \cite[Theorem A]{ND21} the fourth author proved an analogue of Brown's theorem and solved one half (the algebraic part) of these conjectures. He also proved the second half of Theorem \ref{thm: Hoffman cyclotomic MZV} for small weights $w \leq 2q-2$ (see \cite[Theorem D]{ND21}). Consequently, it remains to prove the second half of Theorem \ref{thm: Hoffman cyclotomic MZV} to get the complete theorem \ref{thm: Hoffman cyclotomic MZV} as well as Thakur's variant.

\item In \cite{IKLNDP22} the authors proved the other half (the transcendental part) of Theorem \ref{thm: Hoffman cyclotomic MZV} for all weights $w$ (see \cite[Theorem B]{IKLNDP22}). In \cite{CCM22} the other half of Theorem \ref{thm: Hoffman cyclotomic MZV} was proved using the same method, objects and proof strategy but their presentation is different. 
\end{itemize}

\subsection*{$N=q-1$ (the case of alternating MZV's)} \ppar
\begin{itemize}
\item In \cite{IKLNDP22} the authors proved Theorem \ref{thm: Hoffman cyclotomic MZV} for $N=q-1$ and all weights $w$ (see \cite[Theorem A]{IKLNDP22}). 
\end{itemize}

\subsection{Outline of the proof and plan of the paper} \ppar

Similar to the classical setting, the proof is divided into two parts with completely different flavours.

\subsubsection{Part A: an analogue of Brown-Deligne's theorems} \ppar \label{sec: Part A}

We want to prove an analogue of Brown-Deligne's theorems (Brown \cite{Bro12} for $N=1$ and Deligne \cite{Del10} for $N=2,3,4,8$) in our context. It says that the set $\mathcal{CS}_{N,w}$ spans the vector space $\mathcal{CZ}_{N,w}$. The techniques are a refinement of those developed in \cite[\S 2--3]{ND21}. The proof is divided into three steps:
\begin{itemize}
\item[{\bf (A1)}] First we briefly review the techniques introduced in \cite[\S 2 and \S 3]{ND21} and show that the vector space $\mathcal{CZ}_{N,w}$ is spanned by another set of $\mathcal{CT}_{N,w}$ of larger cardinality consisting of cyclotomic MZV's $\zeta_A \begin{pmatrix}
 \epsilon_1 & \dotsb & \epsilon_n \\
s_1 & \dotsb & s_n \end{pmatrix}$ such that $s_i \leq q$, $\epsilon_i \in \Gamma_N$ for all $i$ and $s_n<q$. Note, however, that $|\mathcal{CT}_{N,w}|$ is generally larger than $|\mathcal{CS}_{N,w}|$. So we have to do extra work to get the sharp upper bound.

\item[{\bf (A2)}] Second we again extend the techniques mentioned above and prove a similar result for the vector space $\mathcal{CL}_{N,w}$ spanned by Carlitz multiple polylogarithms (CMPL's for short) at roots of unity. Importantly, we use the key fact that the MZV's in  $\mathcal{CT}_{N,w}$ given in Step (A1) are CMPL's at roots of unity. Thus we discover a connection between cyclotomic MZV's and CMPL's at roots of unity in Theorem \ref{thm:bridge} which extends \cite[Proposition 1.9]{IKLNDP22}. 

\item[{\bf (A3)}] Last we exploit the fact that the product for CMPL's at roots of unity is ``simple'' and similar to the stuffle product for classical cyclotomic MZV's. Thus we are able to ``reduce'' the generating set from $\mathcal{CT}_{N,w}$ to the desired set $\mathcal{CS}_{N,w}$ and prove the desired analogue of Brown-Deligne's theorems in Theorem \ref{thm: strong BD MZV}. 
\end{itemize}

\subsubsection{Part B: dual $t$-motives connected to CMPL's at roots of unity} \ppar  \label{sec: Part B}

We want to prove (see Theorem \ref{thm: transcendence}) that the elements in $\mathcal{CS}_{N,w}$ are linearly independent over $K_N$. We follow the strategy developed in \cite[\S 4--6]{ND21} and \cite[\S 2--3]{IKLNDP22} which are based on the theory of Anderson $t$-motives introduced by Anderson \cite{And86} and the fruitful transcendental tool called the Anderson-Brownawell-Papanikolas criterion devised in \cite{ABP04}. However, carrying out this strategy for arbitrary $N$ turns out to be highly non-trivial due to the fact that the extension of $k_N/k$ is non-trivial in general. 

Below we outline the steps of the proof, explain the problems that we have to overcome and present our solutions. Recall that $R$ denotes the degree of the cyclotomic extension $k_N$ over $k$.
\begin{itemize}
\item[{\bf (B1)}] First we construct $t$-motives connected to CMPL's at $N$th roots of unity (resp. cyclotomic MZV's). 
\begin{itemize}
\item When $N=1$ in \cite{ND21} and $N=q-1$ in \cite{IKLNDP22}, the cyclotomic field $k_N$ is $k$, hence $R=1$. The construction of dual $t$-motives connected to CMPL's at $N$th roots of unity (or $N$th cyclotomic MZV's) is due to Anderson-Thakur in \cite{AT09}, Chang \cite{Cha14} and Harada \cite{Har21} using the Anderson-Thakur polynomials in \cite{AT90} and twisted $L$-series in \cite{AT09,Cha14}. 
\end{itemize}
For arbitrary $N$, the previous construction does not work. Harada \cite{Har23} solved this problem by considering an untwisted $L$-series (see \S \ref{sec:CMPL motives}). In {\it loc. cit.} he then constructed dual $t$-motives of level $R$ connected to CMPL's at $N$th roots of unity (resp. $N$th cyclotomic MZV's). The trade-off is that we have to deal with higher level and that the associated matrices are more complicated.

\item[{\bf (B2)}] Second, we must develop a higher transcendental theory that generalizes that of level $1$ in \cite[\S 4--6]{ND21} and \cite[\S 2--3]{IKLNDP22}. Then using Harada's idea, we succeed in constructing dual $t$ motives of level $R$ which lift linear relations among CMPL's at roots of unity in $\mathcal{CS}_{N,w}$ and the appropriate of Carlitz's period.

\item[{\bf (B3)}] Last we apply the Anderson-Brownawell-Papanikolas criterion to reduce the proof to the task of finding solutions to a system of $\sigma$-linear equations of order $R$ (higher order), which turns out to be the hardest task. Our solution is more complicated than the one in \cite{ND21,IKLNDP22}, which explains the long section \ref{sec:transcendental part}.  

Roughly speaking, one original ingredient is that we use the descent on the tower of all cyclotomic extensions of $K$ (compared to \cite{Del10,Gla16} in the classical setting). Another key point (see Proposition \ref{prop: solving first equations}) is to show that the system of $\sigma$-linear equations of order $R$ for which  we want to find solutions in $k_N[t]$ has at most one solution up to a constant in $k_N(t)$. Furthermore, if they admit one solution, then there exists a solution belonging to $k[t]$ and satisfying a system of $\sigma$-linear equations of order $1$. Finally, the latter system can be solved as in our previous work \cite{IKLNDP22} and we are done. We note that in the last step we have to use the analogue of Brown-Deligne's theorems for CMPL's proved in \S \ref{sec: Part A}.
\end{itemize}

\subsubsection{Plan of the paper} 
We will briefly explain the organization of the manuscript.
\begin{itemize}
\item In \S \ref{sec: MZV CMPL} we introduce the notation, recall the definition and basic properties of cyclotomic multiple zeta values and Carlitz multiple polylogarithms at roots of unity. We then extend the algebraic tools given in \cite[\S 2--3]{ND21} for these objects and obtain an analogue of Brown-Deligne's theorems in Theorem \ref{thm: strong BD MZV}. Steps {\bf (A1)} and {\bf (A2--A3)} are carried out in \S \ref{sec: cyclotomic MZV} and \S \ref{sec:BD CMPL}, respectively.

\item In \S \ref{sec: dual motives} and \S \ref{sec:transcendental part} we recall Harada's construction of dual $t$-motives of higher level connected to CMPL's at roots of unity. We then construct dual $t$-motives (also of higher level) related to linear combinations of such objects and Carlitz's period. We then state the linear independence of the chosen set in Theorem \ref{thm: transcendence}, the proof of which is given in \S \ref{sec:transcendental part}. Steps {\bf (B1), (B2)} and {\bf (B3)} are carried out in \S \ref{sec:CMPL motives}, \S \ref{sec: linear combination CMPL motives} and \S \ref{sec:transcendental part}, respectively.

\item Finally, in \S \ref{sec:applications} we prove the main results, i.e., Zagier-Hoffman's conjectures in positive characteristic for $N$th cyclotomic multiple zeta values for all $N$.
\end{itemize}

\subsection*{Acknowledgments} 
The first named author (B.-H. Im) was supported by the Basic Science Research Program through the National Research Foundation of Korea (NRF) grant funded by the Korea government (MSIT) NRF-2023R1A2C1002385. Two of the authors (KN. L. and T. ND.) were partially supported by the Excellence Research Chair ``$L$-functions in positive characteristic and applications'' funded by the Normandy Region. The fourth author (T. ND.) was partially supported by the ANR grant COLOSS ANR-19-CE40-0015-02. 


\section{Cyclotomic MZV's and CMPL's at roots of unity} \label{sec: MZV CMPL}

In this section we introduce the notion of MZV's and that of CMPL's at roots of unity. We then develop a refinement of \cite[\S 2--3]{ND21} that takes into account the $N$th cyclotomic character. The presentation closely follows that in \cite{IKLNDP22} where the case $N=q-1$ was treated.

\subsection{Notation} \label{sec: notation}

We recall some notation used in \cite{IKLNDP22,ND21}. 

\subsubsection{}

We recall the notation given in the Introduction. Let $q$ be a power of prime $p$ and $k=\Fq$ be a finite field of order $q$. We denote by $\bar k:=\overline{\mathbb F}_q$ an algebraic closure of $\Fq$. We recall that $A=\Fq[\theta]$ denotes the polynomial ring in $\theta$ over $\Fq$ and $A_+$ denotes the set of monic polynomials in $A$. Then $K=\Fq(\theta)$ is the fraction field of $A$ equipped with the rational point $\infty$. Let $K_\infty=k(\!(1/\theta)\!)$ be the completion of $K$ at $\infty$ and $\C_\infty$ be the completion of a fixed algebraic closure $\overline K$ of $K$ at $\infty$. Further, let $v_\infty$ be the discrete valuation on $K$ associated to the place $\infty$ normalized as $v_\infty(\theta) = -1$ and $|\cdot|_\infty = q^{-v_\infty(\cdot)}$ be the corresponding norm on $K$.

The following constant will play an important role in this paper:
\begin{equation} \label{eq: D_1}
D_1:= \theta^q - \theta \in A.
\end{equation}

\subsubsection{}

Letting $t$ be another independent variable, we denote by $\bT$ the Tate algebra in the variable $t$ with coefficients in $\C_\infty$ equipped with the Gauss norm $\lVert. \rVert_\infty$, and by $\bL$ the fraction field of $\bT$.

We denote $\mathcal E$ the ring of series $\sum_{n \geq 0} a_nt^n \in \overline K[[t]]$ such that $\lim_{n \to +\infty} \sqrt[n]{|a_n|_\infty}=0$ and $[K_\infty(a_0,a_1,\ldots):K_\infty]<\infty$. Then any $f \in \mathcal E$ is an entire function.

We denote by $\widetilde \pi$ the Carlitz period  which is a fundamental period of the Carlitz module. We fix a choice of $(q-1)$st root of $(-\theta)$ and set
	\[ \Omega(t):=(-\theta)^{-q/(q-1)} \prod_{i \geq 1} \left(1-\frac{t}{\theta^{q^i}} \right) \in \bT^\times \]
so that
	\[ \Omega^{(-1)}=(t-\theta)\Omega \quad \text{ and } \quad \frac{1}{\Omega(\theta)}=\widetilde \pi. \]

\subsubsection{}

Throughout this paper, we fix $N \in \mathbb N$. We denote by $k_N \subset \bar k$ the cyclotomic extension of $k=\Fq$ generated by a primitive $N$th root $\zeta_N$ of unity and put $R=[k_N:k]$. The group of $N$th roots of unity is denoted by $\Gamma_N$. We set $\gamma_N:=|\Gamma_N|=N'$ where $N=p^k N'$ with $k \geq 0$ and $N'$ prime to $p$. We put $A_N=k_N[\theta]$, $K_N = k_N(\theta)$ and $K_{N,\infty}:=k_N(\!(1/\theta)\!)$.

\subsubsection{}

Let $\fs=(s_1,\dots,s_n)$ be a tuple of positive integers. We put $w(\fs)=s_1+\dots+s_n$ and call it the weight of $\fs$. The number $n$ is called the depth of $\fs$.

Let $\fs=(s_1,\dots,s_r)$ and $\mathfrak t=(t_1,\dots,t_k)$ be tuples of positive integers. We say that $\fs \leq \mathfrak t$ if the following assertions hold:
\begin{itemize}
\item For all $i \in \N$ we have $s_1+\dots+s_i \leq t_1+\dots+t_i$ where we recall $s_i=0$ (resp. $t_i=0$) for $i$ bigger than the depth of $\fs$ (resp. $\mathfrak t$).

\item $\fs$ and $\mathfrak t$ have the same weight.
\end{itemize}

Letting $\fs=(s_1,s_2,\dots,s_r)  \in \N^r$ we set $\fs_-:=(s_2,\dots,s_r)$. For $i \in \N$ we define $T_i(\fs)$ to be the tuple $(s_1+\dots+s_i,s_{i+1},\dots,s_r)$. Note that $T_1(\fs)=\fs$. Further, for tuples of positive integers $\fs, \mathfrak t$ and for $i \in \mathbb N$, if $T_i(\fs) \leq T_i(\mathfrak t)$, then $T_k(\fs) \leq T_k(\mathfrak t)$ for all $k \geq i$.

Let $\fs=(s_1,\dots,s_n)$ be a tuple of positive integers. We denote by $0 \leq i \leq n$ the biggest integer such that $s_j \leq q$ for all $1 \leq j \leq i$ and define the initial tuple $\Init(\fs)$ of $\fs$ to be the tuple
	\[ \Init(\fs):=(s_1,\dots,s_i). \]
In particular, if $s_1>q$, then $i=0$ and $\Init(\fs)$ is the empty tuple.

For two different tuples $\fs$ and $\mathfrak t$, we consider the lexicographical order for initial tuples and write $\Init(\mathfrak t) \preceq \Init(\fs)$ (resp. $\Init(\mathfrak t) \prec \Init(\fs)$, $\Init(\mathfrak t) \succeq \Init(\fs)$ and $\Init(\mathfrak t) \succ \Init(\fs)$).

\subsubsection{}

Letting $\fs = (s_1, \dotsc, s_n) \in \mathbb{N}^{n}$ and $\fe = (\epsilon_1, \dotsc, \epsilon_n) \in (\Gamma_N)^{n}$, we set $\fs_- := (s_2, \dotsc, s_n)$ and $\fe_- := (\epsilon_2, \dotsc, \epsilon_n) $. A positive array $\begin{pmatrix}
 \fe  \\
\fs  \end{pmatrix}$ is an array of the form 
\[
\begin{pmatrix}
 \fe  \\
\fs  \end{pmatrix}  = \begin{pmatrix}
 \epsilon_1 & \dotsb & \epsilon_n \\
s_1 & \dotsb & s_n \end{pmatrix}.
\]
We set $s_i = 0$ and $\epsilon_i = 1$ for all $i > \depth(\fs)$. Further we recall $w(\fs)=s_1+\dots+s_n$ and put $\chi(\fe) = \epsilon_1  \dots  \epsilon_n$.

Let $\begin{pmatrix}
 \fve  \\
\fs  \end{pmatrix}$, $\begin{pmatrix}
 \fe  \\
\mathfrak{t}  \end{pmatrix}$ be two positive arrays. We define 
$\begin{pmatrix}
 \fve  \\
\fs  \end{pmatrix} + \begin{pmatrix}
 \fe  \\
\mathfrak{t}  \end{pmatrix} := \begin{pmatrix}
 \fve \fe  \\
\fs + \mathfrak{t}  \end{pmatrix}$. We say that $\begin{pmatrix}
 \fve  \\
\fs  \end{pmatrix} \leq \begin{pmatrix}
 \fe  \\
\mathfrak{t}  \end{pmatrix}$
 if the following conditions are satisfied:
\begin{itemize}
    \item $w(\fs) = w(\mathfrak{t})$,
    \item $s_1 + \dotsb + s_i \leq t_1 + \dotsb + t_i$ for all $i \in \N$,
    \item $\chi(\fve) = \chi(\fe)$.
\end{itemize}

\subsection{Multiple zeta values in positive characteristic} \ppar \label{sec: cyclotomic MZV}

For $d \in \mathbb{Z}$ and for $\fs=(s_1,\dots,s_n) \in \N^n$, we recall that $S_d(\fs)$ and $S_{<d}(\fs)$ are given by
\begin{align*}
	S_d(\fs) &= \sum\limits_{\substack{a_1, \dots, a_n \in A_{+} \\ d = \deg a_1> \dots > \deg a_n\geq 0}} \dfrac{1}{a_1^{s_1} \dots a_n^{s_n}} \in K,\\
    S_{<d} (\fs) &= \sum\limits_{\substack{a_1, \dots, a_n \in A_{+} \\ d > \deg a_1> \dots > \deg a_n\geq 0}} \dfrac{1}{a_1^{s_1} \dots a_n^{s_n}} \in K.
\end{align*}

Further letting $\begin{pmatrix}
 \fe  \\
\fs  \end{pmatrix}  =  \begin{pmatrix}
 \epsilon_1 & \dots & \epsilon_n \\
s_1 & \dots & s_n \end{pmatrix}$ be a positive array, we put
\begin{align*}
	S_d \begin{pmatrix}
\fe \\ \fs
\end{pmatrix}  &= \sum\limits_{\substack{a_1, \dots, a_n \in A_{+} \\ d = \deg a_1> \dots > \deg a_n\geq 0}} \dfrac{\epsilon_1^{\deg a_1} \dots \epsilon_n^{\deg a_n }}{a_1^{s_1} \dots a_n^{s_n}} \in K_N, \\
    S_{<d} \begin{pmatrix}
 \fe  \\
\fs  \end{pmatrix}  &= \sum\limits_{\substack{a_1, \dots, a_n \in A_{+} \\ d > \deg a_1> \dots > \deg a_n\geq 0}} \dfrac{\epsilon_1^{\deg a_1} \dots \epsilon_n^{\deg a_n }}{a_1^{s_1} \dots a_n^{s_n}} \in K_N.
\end{align*}

Following Harada \cite{Har23}, we define the cyclotomic MZV by
\begin{equation*}
    \zeta_A \begin{pmatrix}
 \fe  \\
\fs  \end{pmatrix}  = \sum \limits_{d \geq 0} S_d \begin{pmatrix}
 \fe  \\
\fs  \end{pmatrix}  = \sum\limits_{\substack{a_1, \dots, a_n \in A_{+} \\ \deg a_1> \dots > \deg a_n\geq 0}} \dfrac{\epsilon_1^{\deg a_1} \dots \epsilon_n^{\deg a_n }}{a_1^{s_1} \dots a_n^{s_n}}  \in K_{N,\infty}.
\end{equation*}

\begin{remark}
In \cite{Har23} these objects are called colored multiple zeta values. We note that in the classical setting similar objects have several names in the literature: cyclotomic multiple zeta values by Brown in \cite{Bro14}, multiple zeta values relative to $\mu_N$ by Deligne in \cite{Del10}, or colored multiple zeta values by Zhao in \cite{Zha16}.
\end{remark}

Using Chen's formula in \cite{Che15}, Harada \cite{Har23} proved that for $s, t \in \mathbb{N}$ and $\varepsilon, \epsilon \in \Gamma_N$, 
\begin{equation} \label{eq: Chen}
S_d \begin{pmatrix}
 \varepsilon \\
s  \end{pmatrix} S_d \begin{pmatrix}
 \epsilon  \\
t  \end{pmatrix}  = S_d \begin{pmatrix}
 \varepsilon\epsilon  \\
s+t  \end{pmatrix}  + \sum \limits_i \Delta^i_{s,t} S_d \begin{pmatrix}
 \varepsilon\epsilon  & 1 \\
s+t-i & i \end{pmatrix},
\end{equation}
where
\begin{equation*} 
    \Delta^i_{s,t} = \begin{cases}
			(-1)^{s-1} {i - 1  \choose s - 1} +  (-1)^{t-1} {i-1 \choose t-1} & \quad  \text{if } q - 1 \mid i \text{ and } 0 < i < s + t, \\
            0 & \quad \text{otherwise.}
		 \end{cases}
\end{equation*}
We mention that when $s+t \leq q$, all the coefficients $\Delta^i_{s,t}$ are zero. 

Harada then proved similar results for products of cyclotomic MZV's (see \cite{Har23}):
\begin{proposition} \label{sums}
Let $ \begin{pmatrix}
 \fve  \\
\fs  \end{pmatrix}$, $ \begin{pmatrix}
 \fe  \\
\mathfrak{t}  \end{pmatrix}$ be two positive arrays. Then

1) There exist $f_i \in \mathbb{F}_q$ and positive arrays $ \begin{pmatrix}
 \fm_i  \\
\mathfrak{u}_i  \end{pmatrix}$ with $ \begin{pmatrix}
 \fm_i  \\
\mathfrak{u}_i  \end{pmatrix}  \leq  \begin{pmatrix}
 \fve  \\
\fs  \end{pmatrix}  +  \begin{pmatrix}
 \fe  \\
\mathfrak{t}  \end{pmatrix}  $ and $\depth(\mathfrak{u}_i) \leq \depth(\fs) + \depth(\mathfrak{t})$ for all $i$  such that
    \begin{equation*}
        S_d \begin{pmatrix}
 \fve  \\
\fs  \end{pmatrix} S_d \begin{pmatrix}
 \fe  \\
\mathfrak{t}  \end{pmatrix}  = \sum \limits_i f_i S_d \begin{pmatrix}
 \fm_i  \\
\mathfrak{u}_i  \end{pmatrix}  \quad \text{for all } d \in \mathbb{Z}.
    \end{equation*}

2) There exist $f'_i \in \mathbb{F}_q$ and positive arrays $ \begin{pmatrix}
 \fm'_i  \\
\mathfrak{u}'_i  \end{pmatrix}$ with $ \begin{pmatrix}
 \fm'_i  \\
\mathfrak{u}'_i  \end{pmatrix}  \leq  \begin{pmatrix}
 \fve  \\
\fs  \end{pmatrix}  +  \begin{pmatrix}
 \fe  \\
\mathfrak{t}  \end{pmatrix}  $ and $\depth(\mathfrak{u}'_i) \leq \depth(\fs) + \depth(\mathfrak{t})$ for all $i$  such that
    \begin{equation*}
        S_{<d} \begin{pmatrix}
 \fve  \\
\fs  \end{pmatrix} S_{<d} \begin{pmatrix}
 \fe  \\
\mathfrak{t}  \end{pmatrix}  = \sum \limits_i f'_i S_{<d} \begin{pmatrix}
 \fm'_i  \\
\mathfrak{u}'_i  \end{pmatrix}  \quad \text{for all } d \in \mathbb{Z}.
    \end{equation*}

3) There exist $f''_i \in \mathbb{F}_q$ and positive arrays $ \begin{pmatrix}
 \fm''_i  \\
\mathfrak{u}''_i  \end{pmatrix}$ with $ \begin{pmatrix}
 \fm''_i  \\
\mathfrak{u}''_i  \end{pmatrix}  \leq  \begin{pmatrix}
 \fve  \\
\fs  \end{pmatrix}  +  \begin{pmatrix}
 \fe  \\
\mathfrak{t}  \end{pmatrix}  $ and $\depth(\mathfrak{u}''_i) \leq \depth(\fs) + \depth(\mathfrak{t})$ for all $i$  such that
    \begin{equation*}
        S_d \begin{pmatrix}
 \fve  \\
\fs  \end{pmatrix} S_{<d} \begin{pmatrix}
 \fe  \\
\mathfrak{t}  \end{pmatrix}  = \sum \limits_i f''_i S_d \begin{pmatrix}
 \fm''_i  \\
\mathfrak{u}''_i  \end{pmatrix}  \quad \text{for all } d \in \mathbb{Z}.
    \end{equation*}
\end{proposition}

We denote by $\mathcal{CZ}_N$ the $K_N$-vector space spanned by the cyclotomic MZV's and $\mathcal{CZ}_{N,w}$ the $K_N$-vector space spanned by the cyclotomic MZV's of weight $w$. When $N=1$ we use $\mathcal Z$ (resp.  $\mathcal Z_w$) instead of $\mathcal{CZ}_1$ (resp. $\mathcal{CZ}_{1,w}$). 

Proposition~\ref{sums} easily implies that $\mathcal{CZ}_N$ is a commutative algebra over $K_N$. However, we mention that the associativity of $\mathcal Z$ was not known in \cite{ND21} (see Remark 2.2 of {\it loc. cit.}). It was recently proved in \cite[Theorem A]{IKLNDP23}. As a direct consequence, we see that $\mathcal{CZ}_N$ is also associative for all $N$. 

We extend without difficulty the algebraic theory developed in \cite[\S 2 and \S 3]{ND21} to the setting of cyclotomic MZV's to obtain a weak analogue of Brown-Deligne's theorems for cyclotomic MZV's (see Theorem \ref{theorem: weak BD cyclotomic MZV}). We refer the reader to \S \ref{sec: CMPL} where we give more details of this theory in a similar setting.

\begin{theorem} \label{theorem: weak BD cyclotomic MZV}
We denote by $\mathcal{CT}_{N,w}$ the set of all cyclotomic MZV's $\zeta_A \begin{pmatrix}
 \epsilon_1 & \dots & \epsilon_n \\
s_1 & \dots & s_n \end{pmatrix}$ of weight $w$ such that $s_1, \dots, s_{n-1} \leq q$ and $s_n < q$. 
Then 

1) For all positive arrays $ \begin{pmatrix}
 \fve  \\
\frak t  \end{pmatrix}$ of weight $w$, $\zeta_A \begin{pmatrix}
 \fve  \\
\frak t  \end{pmatrix}$ can be expressed as an $A_N$-linear combination of $\zeta_A \begin{pmatrix}
 \fe  \\
\fs  \end{pmatrix}$'s in $\mathcal{CT}_{N,w}$. 

2) The set of all elements $\zeta_A \begin{pmatrix}
 \fe  \\
\fs  \end{pmatrix}$ such that $\zeta_A \begin{pmatrix}
 \fe  \\
\fs  \end{pmatrix}  \in \mathcal{CT}_{N,w}$ forms a set of generators for $\mathcal{CZ}_{N,w}$.
\end{theorem}

\begin{remark} \label{rem: Brown cyclotomic MZV}
1) Theorem \ref{theorem: weak BD cyclotomic MZV}, Part 2 can be seen as an analogue (weak version) of a theorem for cyclotomic MZV's proved by Brown \cite{Bro12} for $N=1$ and Deligne \cite{Del10} for $N=2,3,4,8$ (see also \cite{Gla16}). More specifically,
\begin{itemize}
\item Brown \cite{Bro12} proved that the $\mathbb Q$-vector space spanned by classical MZV's of fixed weight is spanned by Hoffman's basis consisting of $\zeta(n_1,...,n_r)$ of the same weight with $n_i \in \{2,3\}$. As a corollary, he recovered the sharp upper bound on the dimension of this vector space proved by Terasoma \cite{Ter02} and Deligne-Goncharov \cite{DG05}. Brown used the theory of mixed Tate motives and the motivic version of MZV's. The proof is by induction on the level, i.e., the number of $3$'s in Hoffman's basis, and based on an identity among MZV's due to Zagier \cite{Zag12}. 

\item Deligne \cite{Del10} gave a generating set for the $\mathbb Q$-vector space spanned by cyclotomic MZV's for $N=2,3,4,8$ (Deligne used the terminology MZV's relative to $\mu_N$). He also used the theory of mixed Tate motives and the motivic version of MZV's. But the proof is by induction on the depth. It is already noted that the depth is pathological for $N=1$ (see \cite[page 102, lines 23-25]{Del10}).

\item We refer the reader to the articles \cite{Bro14,Del13} for more details.
\end{itemize}

2) Theorem \ref{theorem: weak BD cyclotomic MZV}, Part 1 states that $\mathcal{CT}_{N,w}$ is an integral generating set of $\mathcal{CZ}_{N,w}$. Furthermore, it gives an effective algorithm to express any cyclotomic MZV's as an $A_N$-linear combination of cyclotomic MZV's in $\mathcal{CT}_{N,w}$. In the classical setting, these results are not known as discussed in \cite[page 2, the discussion after Theorem 0.2]{Del13} and \cite[page 103, lines 8--10]{Del10} (see also \cite{Bro12b}).
\end{remark}


\subsection{Multiple polylogarithms in positive characteristic} \label{sec: CMPL} \ppar 

We now review the notion of multiple polylogarithms (or Carlitz multiple polylogarithms) in positive characteristic. We put $\ell_0 := 1$ and $\ell_d := \prod^d_{i=1}(\theta - \theta^{q^i})$ for all $d \in \mathbb{N}$. Letting $\mathfrak s = (s_1 , \dots, s_n) \in \mathbb{N}^n$, for $d \in \mathbb{Z}$, we define 
\begin{align*}
		\Si_d(\mathfrak s) &= \sum\limits_{d=d_1> \dots > d_n\geq 0} \dfrac{1 }{\ell_{d_1}^{s_1} \dots \ell_{d_n}^{s_n}} \in K, \\
	\Si_{<d}(\mathfrak s) &= \sum\limits_{d>d_1> \dots > d_n\geq 0} \dfrac{1 }{\ell_{d_1}^{s_1} \dots \ell_{d_n}^{s_n}} \in K.
\end{align*}

Let $\begin{pmatrix}
 \fe  \\
\fs  \end{pmatrix}  =  \begin{pmatrix}
 \epsilon_1 & \dots & \epsilon_n \\
s_1 & \dots & s_n \end{pmatrix}$ be a positive array. For $d \in \mathbb{Z}$, we put
\begin{align*}
        \Si_d \begin{pmatrix}
 \fe  \\
\fs  \end{pmatrix}  = \sum\limits_{d=d_1> \dots > d_n\geq 0} \dfrac{\epsilon_1^{d_1} \dots \epsilon_n^{d_n} }{\ell_{d_1}^{s_1} \dots \ell_{d_n}^{s_n}} \in K_N, \\
        \Si_{<d} \begin{pmatrix}
 \fe  \\
\fs  \end{pmatrix}  = \sum\limits_{d>d_1> \dots > d_n\geq 0} \dfrac{\epsilon_1^{d_1} \dots \epsilon_n^{d_n} }{\ell_{d_1}^{s_1} \dots \ell_{d_n}^{s_n}} \in K_N.
\end{align*}

Then we introduce the Carlitz multiple polylogarithm at roots of unity (CMPL at roots of unity for short) as follows:
\begin{equation*}
    \Li \begin{pmatrix}
 \fe  \\
\fs  \end{pmatrix}  = \sum \limits_{d \geq 0} \Si_d \begin{pmatrix}
 \fe  \\
\fs  \end{pmatrix}  = \sum\limits_{d_1> \dots > d_n\geq 0} \dfrac{\epsilon_1^{d_1} \dots \epsilon_n^{d_n} }{\ell_{d_1}^{s_1} \dots \ell_{d_n}^{s_n}}   \in K_{N,\infty}.
\end{equation*}
 
The connection between cyclotomic MZV's and CMPL's at roots of unity is encoded in the following lemma which turns out to be crucial in the sequel (see Theorem \ref{thm:bridge}).
\begin{lemma} \label{agree}
For all $ \begin{pmatrix}
 \fe  \\
\fs  \end{pmatrix}$ as above such that $s_i \leq q$ for all $i$, we have
\begin{equation*}
   S_d \begin{pmatrix}
 \fe  \\
\fs  \end{pmatrix}  =  \Si_d \begin{pmatrix}
 \fe  \\
\fs  \end{pmatrix}  \quad \text{for all } d \in \mathbb{Z}.
\end{equation*}
Therefore,
\begin{equation*}
    \zeta_A  \begin{pmatrix}
 \fe  \\
\fs  \end{pmatrix}  = \Li \begin{pmatrix}
 \fe  \\
\fs  \end{pmatrix} .
\end{equation*}
\end{lemma}

\begin{proof}
See for example \cite{Tha09}.
\end{proof}

We now equip an algebra structure for CMPL's at roots of unity. We observe 
\begin{equation} \label{eq: product CMPL}
  \Si_d \begin{pmatrix}
 \varepsilon \\
s  \end{pmatrix}   \Si_d \begin{pmatrix}
 \epsilon  \\
t  \end{pmatrix}   = \Si_d \begin{pmatrix}
 \varepsilon \epsilon \\
s + t  \end{pmatrix},
\end{equation}
hence, for  $\mathfrak{t} = (t_1, \dots, t_n)$,
\begin{equation}\label{eq: redsum}
  \Si_d \begin{pmatrix}
 \varepsilon \\
s  \end{pmatrix}   \Si_d \begin{pmatrix}
 \fe  \\
\mathfrak{t}  \end{pmatrix}   = \Si_d \begin{pmatrix}
 \varepsilon \epsilon_1 & \fe_{-}  \\
s + t_1  & \mathfrak{t}_{-} \end{pmatrix}.
\end{equation}

We emphasize that one advantage of working with CMPL's at roots of unity is that the product of these objects as in Eq. \eqref{eq: product CMPL} is much simpler than that of cyclotomic MZV's in Eq. \eqref{eq: Chen}. We will use this fact to obtain an analogue of Brown-Deligne's theorems for CMPL's at roots of unity (see Theorem \ref{thm: strong BD CMPL}).

Using Eqs. \eqref{eq: product CMPL} and \eqref{eq: redsum} we obtain
\begin{proposition} \label{polysums}
Let $ \begin{pmatrix}
 \fve  \\
\fs  \end{pmatrix}$, $ \begin{pmatrix}
 \fe  \\
\mathfrak{t}  \end{pmatrix}$ be two positive arrays. Then

1) There exist $f_i \in \mathbb{F}_q$ and positive arrays $ \begin{pmatrix}
 \fm_i  \\
\mathfrak{u}_i  \end{pmatrix}$ with $ \begin{pmatrix}
 \fm_i  \\
\mathfrak{u}_i  \end{pmatrix}  \leq  \begin{pmatrix}
 \fve  \\
\fs  \end{pmatrix}  +  \begin{pmatrix}
 \fe  \\
\mathfrak{t}  \end{pmatrix}  $ and $\depth(\mathfrak{u}_i) \leq \depth(\fs) + \depth(\mathfrak{t})$ for all $i$  such that
    \begin{equation*}
        \Si_d \begin{pmatrix}
 \fve  \\
\fs  \end{pmatrix} \Si_d \begin{pmatrix}
 \fe  \\
\mathfrak{t}  \end{pmatrix}  = \sum \limits_i f_i \Si_d \begin{pmatrix}
 \fm_i  \\
\mathfrak{u}_i  \end{pmatrix}  \quad \text{for all } d \in \mathbb{Z}.
    \end{equation*}

2) There exist $f'_i \in \mathbb{F}_q$ and positive arrays $ \begin{pmatrix}
 \fm'_i  \\
\mathfrak{u}'_i  \end{pmatrix}$ with $ \begin{pmatrix}
 \fm'_i  \\
\mathfrak{u}'_i  \end{pmatrix}  \leq  \begin{pmatrix}
 \fve  \\
\fs  \end{pmatrix}  +  \begin{pmatrix}
 \fe  \\
\mathfrak{t}  \end{pmatrix}  $ and $\depth(\mathfrak{u}'_i) \leq \depth(\fs) + \depth(\mathfrak{t})$ for all $i$  such that
    \begin{equation*}
        \Si_{<d} \begin{pmatrix}
 \fve  \\
\fs  \end{pmatrix} \Si_{<d} \begin{pmatrix}
 \fe  \\
\mathfrak{t}  \end{pmatrix}  = \sum \limits_i f'_i \Si_{<d} \begin{pmatrix}
 \fm'_i  \\
\mathfrak{u}'_i  \end{pmatrix}  \quad \text{for all } d \in \mathbb{Z}.
    \end{equation*}

3) There exist $f''_i \in \mathbb{F}_q$ and positive arrays $ \begin{pmatrix}
 \fm''_i  \\
\mathfrak{u}''_i  \end{pmatrix}$ with $ \begin{pmatrix}
 \fm''_i  \\
\mathfrak{u}''_i  \end{pmatrix}  \leq  \begin{pmatrix}
 \fve  \\
\fs  \end{pmatrix}  +  \begin{pmatrix}
 \fe  \\
\mathfrak{t}  \end{pmatrix}  $ and $\depth(\mathfrak{u}''_i) \leq \depth(\fs) + \depth(\mathfrak{t})$ for all $i$  such that
    \begin{equation*}
        \Si_d \begin{pmatrix}
 \fve  \\
\fs  \end{pmatrix} \Si_{<d} \begin{pmatrix}
 \fe  \\
\mathfrak{t}  \end{pmatrix}  = \sum \limits_i f''_i \Si_d \begin{pmatrix}
 \fm''_i  \\
\mathfrak{u}''_i  \end{pmatrix}  \quad \text{for all } d \in \mathbb{Z}.
    \end{equation*}

\end{proposition}
\begin{proof}
The proof follows the same line as in \cite[Proposition 2.1]{ND21}. We omit the details.
\end{proof}

We denote by $\mathcal{CL}_N$ the $K_N$-vector space spanned by the CMPL's at roots of unity and by $\mathcal{CL}_{N,w}$ the $K_N$-vector space spanned by the CMPL's at roots of unity of weight $w$. Proposition~\ref{polysums} implies that $\mathcal{CL}_N$ is a $K_N$-algebra which is  commutative and associative.

\subsection{Brown-Deligne's theorems for CMPL's at roots of unity} \label{sec:BD CMPL} \ppar

This section is devoted to the development of the so-called algebraic theory as in \cite[\S 2 and \S 3]{ND21} to the setting of CMPL's at roots of unity. We adopt the operators $\mathcal B^*$ and $\mathcal C$ of Todd \cite{Tod18} and the operator $\mathcal{BC}$ of Ngo Dac  \cite{ND21} to this setting. We then derive an analogue of Brown-Deligne's theorems for CMPL's at roots of unity (see Theorem \ref{theorem: weak BD CMPL}). As a consequence, we prove that for a fixed weight the vector space spanned by cyclotomic MZV's and that spanned by CMPL's at roots of unity are the same (see Theorem \ref{thm:bridge}). Finally, we exploit the fact that the product formula for CMPL's at roots of unity is ``simple'' and prove a strong version of Brown-Deligne's theorems for CMPL's at roots of unity (see Theorem \ref{thm: strong BD CMPL}).

When $N=q-1$ (i.e., the case of alternating CMPL's), the theory that we develop in this section was presented in detail in our previous paper \cite{IKLNDP22}. Our presentation follows the same line as in {\it loc. cit.} We write down some details for the convenience of the reader.

\subsubsection{Binary relations} \label{sec:Todd}

A binary relation is a $K_N$-linear combination of the form
\begin{equation*}
    \sum \limits_i a_i \Si_d \begin{pmatrix}
 \fve_i  \\
\fs_i  \end{pmatrix}  + \sum \limits_i b_i \Si_{d+1} \begin{pmatrix}
 \fe_i  \\
\mathfrak{t}_i  \end{pmatrix}  =0 \quad \text{for all } d \in \mathbb{Z},
\end{equation*}
where $a_i,b_i \in K_N$ and $ \begin{pmatrix}
 \fve_i  \\
\fs_i  \end{pmatrix},  \begin{pmatrix}
 \fe_i  \\
\mathfrak{t}_i  \end{pmatrix}$ are positive arrays of the same weight.

We denote by $\mathfrak{BR}_{w}$ the set of all binary relations of weight $w$. Then $\mathfrak{BR}_{w}$ is a $K_N$-vector space. From the fundamental relation in \cite[\S 3.4.6]{Tha09} follows an important family of binary relations
\begin{equation} \label{eq: fundamental relation epsilon}
 R_{\epsilon} \colon \quad  \Si_d \begin{pmatrix}
 \epsilon\\
q  \end{pmatrix}  + \epsilon^{-1}D_1 \Si_{d+1} \begin{pmatrix}
 \epsilon& 1 \\
1 & q-1  \end{pmatrix}  =0.
\end{equation}
Here $D_1 = \theta^q - \theta$ given in Eq. \eqref{eq: D_1} and $\epsilon \in \Gamma_N$. 

For the rest of this section, let $R \in \mathfrak{BR}_w$ be a binary relation of the form
\begin{equation} \label{eq: binary relation 2}
    R(d) \colon \quad \sum \limits_i a_i \Si_d \begin{pmatrix}
 \fve_i  \\
\fs_i  \end{pmatrix}  + \sum \limits_i b_i \Si_{d+1} \begin{pmatrix}
 \fe_i  \\
\mathfrak{t}_i  \end{pmatrix}  =0,
\end{equation}
where $a_i,b_i \in K_N$ and $ \begin{pmatrix}
 \fve_i  \\
\fs_i  \end{pmatrix},  \begin{pmatrix}
 \fe_i  \\
\mathfrak{t}_i  \end{pmatrix}$ are positive arrays of the same weight. We now define some operators on $K_N$-vector spaces of binary relations.

\subsubsection{Operators $\mathcal B^*$} 

Let $ \begin{pmatrix}
 \sigma  \\
v  \end{pmatrix}$ be a positive array of depth $1$. We define an operator
$\mathcal B^*_{\sigma,v} \colon \mathfrak{BR}_{w} \longrightarrow \mathfrak{BR}_{w+v}$
as follows: for each $R \in \mathfrak{BR}_{w}$,
the image $\mathcal B^*_{\sigma,v}(R) = \Si_d \begin{pmatrix}
 \sigma  \\
v  \end{pmatrix} \sum_{j < d} R(j)$ is a binary relation of the form
\begin{align*}
    0 &= \Si_d \begin{pmatrix}
 \sigma  \\
v  \end{pmatrix}   \left(\sum \limits_ia_i \Si_{<d} \begin{pmatrix}
 \fve_i  \\
\fs_i  \end{pmatrix}  + \sum \limits_i  b_i \Si_{<d+1} \begin{pmatrix}
 \fe_i  \\
\mathfrak{t}_i  \end{pmatrix} \right)  \\
    &= \sum \limits_i a_i \Si_d \begin{pmatrix}
 \sigma  \\
v  \end{pmatrix} \Si_{<d} \begin{pmatrix}
 \fve_i  \\
\fs_i  \end{pmatrix}  + \sum \limits_i  b_i \Si_d \begin{pmatrix}
 \sigma  \\
v  \end{pmatrix}  \Si_{<d} \begin{pmatrix}
 \fe_i  \\
\mathfrak{t}_i  \end{pmatrix}  + \sum \limits_i  b_i \Si_d \begin{pmatrix}
 \sigma  \\
v  \end{pmatrix}  \Si_{d} \begin{pmatrix}
 \fe_i  \\
\mathfrak{t}_i  \end{pmatrix} \\
    &= \sum \limits_i a_i \Si_d \begin{pmatrix}
\sigma & \fve_i  \\
v& \fs_i  \end{pmatrix}  + \sum \limits_i  b_i \Si_d \begin{pmatrix}
\sigma & \fe_i  \\
v& \mathfrak{t}_i  \end{pmatrix}  + \sum \limits_i  b_i  \Si_d \begin{pmatrix}
 \sigma \epsilon_{i1}  & \fe_{i-} \\
v + t_{i1} & \mathfrak{t}_{i-} \end{pmatrix} .
\end{align*}
The last equality follows from the explicit formula \eqref{eq: redsum}.

Let $ \begin{pmatrix}
 \Sigma  \\
V  \end{pmatrix}  =  \begin{pmatrix}
 \sigma_1 & \dots & \sigma_n \\
v_1 & \dots & v_n \end{pmatrix}$ be a positive array. We define $\mathcal{B}^*_{\Sigma,V}(R) $ by
$\mathcal B^*_{\Sigma,V}(R) = \mathcal B^*_{\sigma_1,v_1} \circ \dots \circ \mathcal B^*_{\sigma_n,v_n}(R)$.

\subsubsection{Operators $\mathcal C$}  

Let $ \begin{pmatrix}
 \Sigma  \\
V  \end{pmatrix}$ be a positive array of weight $v$. We define an operator $\mathcal C_{\Sigma,V}(R) \colon \mathfrak{BR}_{w} \longrightarrow \mathfrak{BR}_{w+v}$
as follows: for each $R \in \mathfrak{BR}_{w}$,
the image $\mathcal C_{\Sigma,V}(R) = R(d) \Si_{<d+1} \begin{pmatrix}
 \Sigma  \\
V  \end{pmatrix}$ is a binary relation of the form
\begin{align*}
    0 &= \left( \sum \limits_i a_i \Si_d \begin{pmatrix}
 \fve_i  \\
\fs_i  \end{pmatrix}  + \sum \limits_i b_i \Si_{d+1} \begin{pmatrix}
 \fe_i  \\
\mathfrak{t}_i  \end{pmatrix} \right) \Si_{<d+1} \begin{pmatrix}
 \Sigma  \\
V  \end{pmatrix}   \\
    &= \sum \limits_i a_i \Si_d \begin{pmatrix}
 \fve_i  \\
\fs_i  \end{pmatrix} \Si_{d} \begin{pmatrix}
 \Sigma  \\
V  \end{pmatrix}  + \sum \limits_i a_i \Si_d \begin{pmatrix}
 \fve_i  \\
\fs_i  \end{pmatrix} \Si_{<d} \begin{pmatrix}
 \Sigma  \\
V  \end{pmatrix}  + \sum \limits_i b_i \Si_{d+1} \begin{pmatrix}
 \fe_i  \\
\mathfrak{t}_i  \end{pmatrix} \Si_{<d+1} \begin{pmatrix}
 \Sigma  \\
V  \end{pmatrix} \\
    &= \sum \limits_i f_i \Si_d \begin{pmatrix}
 \fm_i  \\
\mathfrak{u}_i  \end{pmatrix}  + \sum \limits_i f'_i \Si_{d+1} \begin{pmatrix}
 \fm'_i  \\
\mathfrak{u}'_i  \end{pmatrix} .
\end{align*}
The last equality follows from Proposition \ref{polysums}.

\subsubsection{Operators $\mathcal{BC}$}  

Let $\epsilon \in \Gamma_N$. We define an operator $\mathcal{BC}_{\epsilon,q} \colon \mathfrak{BR}_{w} \longrightarrow \mathfrak{BR}_{w+q}$ as follows: for $R \in \mathfrak{BR}_{w}$  as in Eq. \eqref{eq: binary relation 2},
the image $\mathcal{BC}_{\epsilon,q}(R)$ is a binary relation given by
\begin{align*}
    \mathcal{BC}_{\epsilon,q}(R) = \mathcal B^*_{\epsilon,q}(R) - \sum\limits_i b_i \mathcal C_{\fe_i,\mathfrak{t}_i} (R_{\epsilon}).
\end{align*}

By direct calculations, we see that $\mathcal{BC}_{\epsilon,q}(R)$ is of the form
\begin{equation*}
   \sum \limits_i a_i \Si_d \begin{pmatrix}
\epsilon& \fve_i  \\
q& \fs_i  \end{pmatrix}  + \sum \limits_{i,j} b_{ij} \Si_{d+1} \begin{pmatrix}
\epsilon& \fe_{ij}  \\
1& \mathfrak{t}_{ij}  \end{pmatrix} =0,
\end{equation*}
where $b_{ij} \in K_N$ and $ \begin{pmatrix}
\fe_{ij}  \\
\mathfrak{t}_{ij}  \end{pmatrix}$ are positive arrays satisfying $ \begin{pmatrix}
\fe_{ij}  \\
\mathfrak{t}_{ij}  \end{pmatrix}  \leq  \begin{pmatrix}
1  \\
q-1  \end{pmatrix}  +  \begin{pmatrix}
\fe_{i}  \\
\mathfrak{t}_{i}  \end{pmatrix}$ for all $j$.

\subsubsection{A key expression}

We recall the following result which is crucial to obtain the strong version of Brown-Deligne's theorems in our setting (see \cite[Proposition 1.6]{IKLNDP22}). The key fact is that we exploit the simple product formula for CMPL's at roots of unity given by Eq. \eqref{eq: Chen} to obtain a simple and explicit expression of terms of type 1.

\begin{proposition}\label{polydecom}
We recall that $A_N=k_N[\theta]$.

1) Let $ \begin{pmatrix}
 \fve  \\
\fs  \end{pmatrix}  =  \begin{pmatrix}
 \varepsilon_1 & \dots & \varepsilon_n \\
s_1 & \dots & s_n \end{pmatrix}$ be a positive array such that $\Init(\fs) = (s_1, \dots, s_{k-1})$ (see \S \ref{sec: notation}) for some $1 \leq k \leq n$, and let $\varepsilon$ be an element in $\Gamma_N$.  Then $\Li \begin{pmatrix}
 \fve  \\
\fs  \end{pmatrix}$ can be decomposed as follows
    \begin{equation*}
    \Li \begin{pmatrix}
 \fve  \\
\fs  \end{pmatrix}
    = \underbrace{ - \Li \begin{pmatrix}
 \fve'  \\
\fs'  \end{pmatrix} }_\text{type 1} + \underbrace{\sum\limits_i b_i\Li \begin{pmatrix}
 \fe_i'  \\
\mathfrak{t}'_i  \end{pmatrix} }_\text{type 2} + \underbrace{\sum\limits_i c_i\Li \begin{pmatrix}
 \fm_i  \\
\mathfrak{u}_i  \end{pmatrix} }_\text{type 3}  ,
    \end{equation*}
    where $ b_i, c_i \in A_N$ are divisible by $D_1$ given as in Eq. \eqref{eq: D_1} such that for all $i$, the following properties are satisfied:
    \begin{itemize}
        \item For all positive arrays $ \begin{pmatrix}
 \fe  \\
\mathfrak{t}  \end{pmatrix}$ appearing on the right-hand side, $\depth(\mathfrak{t}) \geq \depth(\fs)$ and $T_k(\mathfrak{t}) \leq T_k(\fs)$.

     \item For the positive array $ \begin{pmatrix}
 \fve'  \\
\fs'  \end{pmatrix}$ of type $1$ with respect to $ \begin{pmatrix}
 \fve  \\
\fs  \end{pmatrix}$, we have
\begin{align*}
\begin{pmatrix}
 \fve'  \\
\fs'  \end{pmatrix} =
\begin{pmatrix}
\varepsilon_1 & \dots & \varepsilon_{k-1} & \varepsilon & \varepsilon^{-1}\varepsilon_{k} & \varepsilon_{k+1} & \dots & \varepsilon_n \\
s_1 & \dots & s_{k-1} & q &  s_k- q & s_{k+1} & \dots & s_n \end{pmatrix}.
\end{align*}
Moreover, for all $k \leq \ell \leq n$,
        \begin{equation*}
s'_{1} +  \dots + s'_\ell < s_1 +  \dots + s_\ell.
     \end{equation*}

\item For the positive array $ \begin{pmatrix}
 \fe'  \\
\mathfrak{t}'  \end{pmatrix}$ of type $2$ with respect to $ \begin{pmatrix}
 \fve  \\
\fs  \end{pmatrix}$, for all $k \leq \ell \leq n$,
        \begin{equation*}
t'_{1} +  \dots + t'_\ell < s_1 +  \dots + s_\ell.
     \end{equation*}

        \item For the positive array $ \begin{pmatrix}
 \fm  \\
\mathfrak{u}  \end{pmatrix}$ of type $3$ with respect to $ \begin{pmatrix}
 \fve  \\
\fs  \end{pmatrix}$, we have $\Init(\fs) \prec\Init(\mathfrak{u})$.
    \end{itemize}

\noindent 2) Let $ \begin{pmatrix}
 \fve  \\
\fs  \end{pmatrix}  =  \begin{pmatrix}
 \varepsilon_1 & \dots & \varepsilon_k \\
s_1 & \dots & s_k \end{pmatrix}$ be a positive array such that $\Init(\fs) = \fs$ and $s_k = q$. Then $\Li \begin{pmatrix}
 \fve  \\
\fs  \end{pmatrix}$ can be decomposed as follows:
    \begin{equation*}
    \Li \begin{pmatrix}
 \fve  \\
\fs  \end{pmatrix}
    =   \underbrace{\sum\limits_i b_i\Li \begin{pmatrix}
 \fe'_i  \\
\mathfrak{t}'_i  \end{pmatrix} }_\text{type 2} + \underbrace{\sum\limits_i c_i\Li \begin{pmatrix}
 \fm_i  \\
\mathfrak{u}_i  \end{pmatrix} }_\text{type 3}  ,
    \end{equation*}
    where $ b_i, c_i \in A_N$ divisible by $D_1$ such that for all $i$, the following properties are satisfied:
    \begin{itemize}
        \item For all positive arrays $ \begin{pmatrix}
 \fe  \\
\mathfrak{t}  \end{pmatrix}$ appearing on the right-hand side, $\depth(\mathfrak{t}) \geq \depth(\fs)$ and $T_k(\mathfrak{t}) \leq T_k(\fs)$.

\item For the positive array $ \begin{pmatrix}
 \fe'  \\
\mathfrak{t}'  \end{pmatrix}$ of type $2$ with respect to $ \begin{pmatrix}
 \fve  \\
\fs  \end{pmatrix}$,
        \begin{equation*}
t'_{1} +  \dots + t'_k < s_1 +  \dots + s_k.
     \end{equation*}

        \item For the positive array $ \begin{pmatrix}
 \fm  \\
\mathfrak{u}  \end{pmatrix}$ of type $3$ with respect to $ \begin{pmatrix}
 \fve  \\
\fs  \end{pmatrix}$, we have $\Init(\fs) \prec\Init(\mathfrak{u})$.
\end{itemize}
\end{proposition}

\begin{proof}
For the convenience of the reader we outline how to obtain the decompositions. For $m \in \mathbb{N}$, we denote by $q^{\{m\}}$ the sequence of length $m$ with all terms equal to $q$. We agree by convention that $q^{\{0\}}$ is the empty sequence. Setting $s_0 = 0$, we can assume that there exists a maximal index $j$ with $0 \leq j \leq k-1$ such that $s_j < q$, hence $\Init(\fs) = (s_1, \dots, s_j, q^{\{k-j-1\}})$. We set $ \begin{pmatrix}
\Sigma'  \\
V'  \end{pmatrix}  =  \begin{pmatrix}
 \varepsilon_{1} &\dots & \varepsilon_j \\
 s_1 &\dots & s_j \end{pmatrix} $

For Part 1, since $\Init(\fs) = (s_1, \dots, s_{k-1})$, we get $s_k > q$. We put $\begin{pmatrix}
\Sigma  \\
V  \end{pmatrix}  =  \begin{pmatrix}
\varepsilon^{-1} \varepsilon_{k} & \varepsilon_{k+1} &\dots & \varepsilon_n \\
s_k - q & s_{k+1} &\dots & s_n \end{pmatrix} $. Then the desired decomposition is obtained from the binary relation
\begin{equation*}
\mathcal{B}^*_{\Sigma',V'} \circ \mathcal{BC}_{\varepsilon_{j+1},q} \circ \dots \circ \mathcal{BC}_{\varepsilon_{k-1},q} \circ \mathcal{C}_{\Sigma,V}(R_{\varepsilon}).
\end{equation*} 
Here we recall that $R_{\varepsilon}$ is the binary relation given in Eq. \eqref{eq: fundamental relation epsilon}.
 
For Part 2, the desired decomposition is obtained from the binary relation
\begin{equation*}
\mathcal{B}^*_{\Sigma',V'} \circ \mathcal{BC}_{\varepsilon_{j+1},q} \circ \dots \circ \mathcal{BC}_{\varepsilon_{k-1},q} (R_1).
\end{equation*} 
Again $R_1$ is the binary relation given in Eq. \eqref{eq: fundamental relation epsilon} for $\varepsilon=1$ (known as the fundamental relation in \cite{ND21}).
\end{proof}

\subsubsection{A weak version of Brown-Deligne's theorems for CMPL's at roots of unity} \label{sec:weak BD CMPL} 

We recall (see Theorem \ref{theorem: weak BD cyclotomic MZV}) that $\mathcal{CT}_{N,w}$ denotes the set of all cyclotomic MZV's $\zeta_A \begin{pmatrix}
 \fve  \\
\fs  \end{pmatrix}$  of weight $w$ such that $s_1, \dots, s_{n-1} \leq q$ and $s_n < q$. By Lemma \ref{agree}, we know that $\zeta_A \begin{pmatrix}
 \fve  \\
\fs  \end{pmatrix}=  \Li \begin{pmatrix}
 \fve  \\
\fs  \end{pmatrix}$. Thus $\mathcal{CT}_{N,w}$ equals the set of all CMPL's at roots of unity $\Li \begin{pmatrix}
 \fve  \\
\fs  \end{pmatrix}$  of weight $w$ such that $s_1, \dots, s_{n-1} \leq q$ and $s_n < q$. 

Using Proposition \ref{polydecom} we obtain a weak version of Brown-Deligne's theorems for CMPL's at roots of unity whose proof follows the same line as that in \cite[\S 3]{ND21} (see also \cite[\S 1]{IKLNDP22}):

\begin{theorem} \label{theorem: weak BD CMPL}
1) For all positive arrays $ \begin{pmatrix}
 \fve  \\
\frak t  \end{pmatrix}$ of weight $w$, $\Li \begin{pmatrix}
 \fve  \\
\frak t  \end{pmatrix}$ can be expressed as an $A_N$-linear combination of $\Li \begin{pmatrix}
 \fe  \\
\fs  \end{pmatrix}$'s in $\mathcal{CT}_{N,w}$. 

2) The $K_N$-vector space $\mathcal{CL}_{N,w}$ spanned by all CMPL's at roots of unity of weight $w$ is spanned by the set $\mathcal{CT}_{N,w}$.
\end{theorem}

Combining Theorems \ref{theorem: weak BD cyclotomic MZV} and \ref{theorem: weak BD CMPL} allows us to identify the two vector spaces of cyclotomic MZV's and CMPL's at roots of unity.
\begin{theorem} \label{thm:bridge}
The $K_N$-vector space $\mathcal{CZ}_{N,w}$ of cyclotomic MZV's of weight $w$ and the $K_N$-vector space $\mathcal{CL}_{N,w}$ of CMPL's at roots of unity of weight $w$ are the same.
\end{theorem}

\subsubsection{An analogue of Brown-Deligne's theorems for CMPL's at roots of unity} \label{sec:strong Brown CMPL}

Following \cite{IKLNDP22} we consider the set $\mathcal{T}_w$ consisting of positive tuples $\fs = (s_1, \dots, s_n)$ of weight $w$ such that $s_1, \dots, s_{n-1} \leq q$ and $s_n <q$, together with the set  $\mathcal{S}_w$ consisting of positive tuples $\fs = (s_1, \dots, s_n)$ of weight $w$ such that $ q \nmid s_i$ for all $i$. We define a map
\begin{equation*}
    \iota \colon \mathcal{S}_w \longrightarrow \mathcal{T}_w
\end{equation*}
as follows: letting $\fs = (s_1, \dots, s_n) \in \mathcal{S}_w$, we express $s_i = h_i q + r_i $ where $0 < r_i < q$ and $h_i \in \mathbb{Z}^{\ge0}$ and put
\begin{equation*}
    \iota(\mathfrak s) = (\underbrace{q, \dots, q}_{\text{$h_1$ times}}, r_1 , \dots, \underbrace{q, \dots, q}_{\text{$h_n$ times}}, r_n).
\end{equation*}
It is clear that this map is a bijection. 

Let $\mathcal{CS}_{N,w}$ denote the set of CMPL's at roots of unity $\Li \begin{pmatrix}
 \fe  \\
\fs  \end{pmatrix}$ such that $\mathfrak s= (s_1, \dots, s_n) \in \mathcal{S}_w$ and $\fe=(\epsilon_1,\dots,\epsilon_n) \in (\Gamma_N)^n$. We note that $\mathcal{CS}_{N,w}$ is smaller than $\mathcal{CT}_{N,w}$. 

\begin{theorem} \label{thm: strong BD CMPL}
The set $\mathcal{CS}_{N,w}$ forms a set of generators for $\mathcal{CL}_{N,w}$.
\end{theorem}

\begin{proof}
Letting $\Li \begin{pmatrix}
 \fe  \\
\fs  \end{pmatrix}  \in \mathcal{CT}_{N,w}$, we know that $\mathfrak s \in \mathcal{T}_w$. Thus there exists a unique tuple $\fs' \in \mathcal{S}_w$ such that $\fs = \iota(\fs')$ as $\iota$ is a bijection. By Proposition \ref{polydecom}  and Theorem \ref{theorem: weak BD CMPL}, it follows that there exists a tuple $\fe' \in (\Gamma_N)^{\depth(\fs')} $ such that  $\Li \begin{pmatrix}
 \fe'  \\
\fs'  \end{pmatrix}$ can be expressed as follows:
\begin{equation*}
    \Li \begin{pmatrix}
 \fe'  \\
\fs'  \end{pmatrix}  = \sum  a_{(\fe',\fs'),(\fe,\mathfrak t)} \Li \begin{pmatrix}
 \fe  \\
\mathfrak{t}  \end{pmatrix},
\end{equation*}
where $ \begin{pmatrix}
 \fe  \\
\mathfrak{t}  \end{pmatrix}$ ranges over all elements of $\mathcal{CT}_{N,w}$ and $a_{(\fe',\fs'),(\fe,\mathfrak t)} \in A_N$ satisfying
\begin{equation*}
 a_{(\fe',\fs'),(\fe,\mathfrak t)} \equiv \begin{cases}
			\pm 1  \ (\text{mod } D_1) & \text{if }  \begin{pmatrix}
 \fe  \\
\mathfrak{t}  \end{pmatrix}  =  \begin{pmatrix}
 \fe  \\
\fs  \end{pmatrix},\\
            0 \ \ (\text{mod } D_1) & \text{otherwise}.
		 \end{cases}
\end{equation*}
We observe that $\Li \begin{pmatrix}
 \fve'  \\
\fs'  \end{pmatrix}  \in \mathcal{CS}_{N,w}$. 

It implies immediately that the transition matrix from the set consisting of such $\Li \begin{pmatrix}
 \fe'  \\
\fs'  \end{pmatrix}$ as above to the set consisting of $\Li \begin{pmatrix}
 \fe  \\
\fs  \end{pmatrix}$ with $ \begin{pmatrix}
 \fe  \\
\fs  \end{pmatrix}  \in \mathcal{CT}_{N,w}$ is invertible. The theorem follows.
\end{proof}

Combining Theorems \ref{thm:bridge} and \ref{theorem: weak BD CMPL} yields

\begin{theorem} \label{thm: strong BD MZV}
Let $\mathcal{CS}_{N,w}$ denote the set of CMPL's at roots of unity $\Li \begin{pmatrix}
 \fe  \\
\fs  \end{pmatrix}$ such that $\mathfrak s \in \mathcal{S}_w$. Then the set $\mathcal{CS}_{N,w}$ forms a set of generators for $\mathcal{CZ}_{N,w}$.
\end{theorem}

We easily see that one has an effective algorithm for expressing any cyclotomic multiple zeta value as a linear combination of $\mathcal{CZ}_{N,w}$ as opposed to the classical setting (see \cite{Bro12b,Del10,Del13} for more explanations).


\section{Dual $t$-motives connected to CMPL's at roots of unity} \label{sec: dual motives}

\subsection{Dual $t$-motives} \ppar

We recall the notion of dual $t$-motives due to Anderson and refer the reader to \cite{And86} for the related notion of $t$-motives. For $i \in \mathbb Z$ we consider the $i$-fold twisting of $\C_\infty((t))$ defined by
\begin{align*}
\C_\infty((t)) & \rightarrow \C_\infty((t)) \\
f=\sum_j a_j t^j & \mapsto f^{(i)}:=\sum_j a_j^{q^i} t^j.
\end{align*}
We extend $i$-fold twisting to matrices with entries in $\C_\infty((t))$ by twisting entry-wise.

Let $\overline K[t,\sigma]$ be the non-commutative $\overline K[t]$-algebra generated by
the new variable $\sigma$ subject to the relation $\sigma f=f^{(-1)}\sigma$ for all $f \in \overline K[t]$.

\begin{definition}
An effective dual $t$-motive is a $\overline K[t,\sigma]$-module $\mathcal M'$ which is free and finitely generated over $\overline K[t]$ such that for $\ell\gg 0$ we have
	\[(t-\theta)^\ell(\mathcal M'/\sigma \mathcal M') = \{0\}.\]
\end{definition}

We mention that effective dual $t$-motives are called Frobenius modules in \cite{CPY19,Har21,KL16} (see also \cite[\S 4]{HJ20} for a more general notion of dual $t$-motives). Throughout this paper we will always work with effective dual $t$-motives. Therefore, we will sometimes drop the word ``effective" where there is no confusion.

Letting $\mathcal M$ be a dual $t$-motive, we say that it is represented by a matrix $\Phi \in \Mat_r(\overline K[t])$ if $\mathcal M$ is a $\overline K[t]$-module free of rank $r$ and the action of $\sigma$ is represented by the matrix $\Phi$ on a given  $\overline K[t]$-basis for $\mathcal M$. We say that $\mathcal M$ is uniformizable or rigid analytically trivial if there exists a matrix $\Psi \in \text{GL}_r(\bT)$ satisfying $\Psi^{(-1)}=\Phi \Psi$. The matrix $\Psi$ is called a rigid analytic trivialization of $\mathcal M$.

We now recall the Anderson-Brownawell-Papanikolas criterion which is crucial in the sequel (see \cite[Theorem 3.1.1]{ABP04}). The curious reader might compare it with Beukers' result \cite{Beu06} in the classical setting.

\begin{theorem}[Anderson-Brownawell-Papanikolas] \label{thm:ABP}
Let $\Phi \in \Mat_\ell(\overline K[t])$ be a matrix such that $\det \Phi=c(t-\theta)^s$ for some $c \in \overline K$ and $s \in \mathbb Z^{\geq 0}$. Let $\psi \in \Mat_{\ell \times 1}(\mathcal E)$ (see \S \ref{sec: notation} for $\mathcal E$) be a vector  satisfying $\psi^{(-1)}=\Phi \psi$ and $\rho \in \Mat_{1 \times \ell}(\overline K)$ such that $\rho \psi(\theta)=0$. Then there exists a vector $P \in \Mat_{1 \times \ell}(\overline K[t])$ such that
	\[ P \psi=0 \quad \text{and} \quad P(\theta)=\rho. \]
\end{theorem}

To end this section we recall $R=[k_N:k]$ and mention that by replacing $\Fq$ (resp. $\sigma$) by $\mathbb F_{q^R}$ (resp. $\sigma^R$) one gets a theory of dual $t$-motives of level $R$ (see \cite{CPY19,Har23} for more details).

\subsection{Dual $t$-motives connected to CMPL's at roots of unity} \label{sec:CMPL motives}  \ppar

We recall that as in \S \ref{sec: notation}, $\widetilde \pi$ denotes the Carlitz period  which is a fundamental period of the Carlitz module (see \cite{Gos96, Tha04}) and $\Omega$ denotes the Anderson-Thakur function introduced in \cite{AT90}. 

Let $\fs=(s_1,\ldots,s_r) \in \mathbb N^r$ be a tuple and $\fe=(\epsilon_1,\ldots,\epsilon_r) \in (\Gamma_N)^r$.  For all $1 \leq i \leq r$ we fix a $(q^R-1)$th root $\mu_i$ of $\epsilon_i^R \in k_N^\times$. Following Harada \cite{Har23} we consider the untwisted $L$-series
\begin{align} \label{eq: series L Harada}
\frakL^{\text{un}}(\fs;\fe):=\sum_{i_1 > \dots > i_r \geq 0} \epsilon_1^{i_1} (\Omega^{s_1})^{(i_1)} \dots \epsilon_r^{i_r} (\Omega^{s_r})^{(i_r)}.
\end{align}

\begin{remark}
We mention that these series are different from those that appeared in our previous works on MZV's and AMZV's \cite{ND21,IKLNDP22} which are twisted series defined by Chang \cite{Cha14}.
\end{remark}

 It can be proved that $\frakL^{\text{un}}(\fs,\fe) \in \mathcal E$ (the proof follows the same line as in \cite[Lemma 4.2]{Har23}).  In the sequel, we will use the following property of this series (see \cite[Lemma 4.4]{Har23}): for all $j \in \mathbb Z^{\geq 0}$ divisible by $R$, we have
\begin{align} 
\frakL^{\text{un}}(\fs;\fe)(\theta) &=\frac{\Li \begin{pmatrix}
\fe \\ \fs
\end{pmatrix}}{\widetilde \pi^{w(\fs)}}, \label{eq: CMPL at roots of unity v0} \\
\frakL^{\text{un}}(\fs;\fe)(\theta^{q^j}) &=(\epsilon_1 \dots \epsilon_r)^j  (\frakL^{\text{un}}(\fs;\fe)(\theta))^{q^j}. \label{eq: CMPL at roots of unity 2 v0}
\end{align} 
We note that in the proof of \cite[Lemma 4.4]{Har23}, we need the equalities $\epsilon_i^{(j)}=\epsilon_i$ which follow from the fact that $j$ divisible by $R$.

Further,
\begin{align*}
& \frakL^{\text{un}}(s_1,\ldots,s_r;\epsilon_1,\ldots,\epsilon_r)^{(-1)} =(\epsilon_1 \dots \epsilon_r)^{(-1)} \frakL^{\text{un}}(s_1,\ldots,s_r;\epsilon_1^{(-1)},\ldots,\epsilon_r^{(-1)}) \\
& \hspace{2cm} +(\epsilon_1 \dots \epsilon_{r-1})^{(-1)} (t-\theta)^{s_r} \Omega^{s_r} \frakL^{\text{un}}(s_1,\ldots,s_{r-1};\epsilon_1^{(-1)},\ldots,\epsilon_{r-1}^{(-1)}).
\end{align*}
As in \cite[Lemma 3.3]{Har23}, it follows that for all $j \in \mathbb N$,
\begin{align*}
& \frakL^{\text{un}}(s_1,\ldots,s_r;\epsilon_1,\ldots,\epsilon_r)^{(-j)} =((\epsilon_1 \dots \epsilon_r)^j)^{(-j)} \sum_{i=0}^r \Omega^{s_{i+1}+\dots+s_r} \\
& \quad \times \sum_{0<j_{i+1}<\dots<j_r \leq j} T_{s_{i+1},j_{i+1}}((\epsilon_{i+1})^{(-j)}) \dots T_{s_r,j_r}((\epsilon_r)^{(-j)}) \frakL^{\text{un}}(s_1,\ldots,s_i;\epsilon_1^{(-j)},\ldots,\epsilon_i^{(-j)}).
\end{align*}
Here for all $s \in \mathbb N$, $k \in \mathbb N$ and $\epsilon \in \Gamma_N$, we put
\begin{align} \label{eq: Tsk}
T_{s,k} &:= (t-\theta)^s ((t-\theta)^s)^{(-1)} \dots ((t-\theta)^s)^{(-(k-1))}, \\
T_{s,k}(\epsilon) &:=\epsilon^{-k} T_{s,k} =\epsilon^{-k} (t-\theta)^s ((t-\theta)^s)^{(-1)} \dots ((t-\theta)^s)^{(-(k-1))}. \notag
\end{align}
In particular, when $j=R$, since $\epsilon_i^{(-R)}=\epsilon_i$, we obtain
\begin{align*}
& \frakL^{\text{un}}(s_1,\ldots,s_r;\epsilon_1,\ldots,\epsilon_r)^{(-R)} =(\epsilon_1 \dots \epsilon_r)^R \sum_{i=0}^r \Omega^{s_{i+1}+\dots+s_r} \\
& \quad \times \sum_{0<j_{i+1}<\dots<j_r \leq R} T_{s_{i+1},j_{i+1}}(\epsilon_{i+1}) \dots T_{s_r,j_r}(\epsilon_r) \frakL^{\text{un}}(s_1,\ldots,s_i;\epsilon_1,\ldots,\epsilon_i).
\end{align*}

We consider the dual $t$-motives $\mathcal M_{\fs,\fe}$ and $\mathcal M_{\fs,\fe}'$ attached to $(\fs,\fe)$ given by the matrices
\begin{align*}
&\Phi_{\fs,\fe} \\
&=
\begin{pmatrix} 
T_{s_1+\dots+s_r,R} & 0 & 0 & \dots & 0 \\
\mu_1 T_{s_2+\dots+s_r,R} \sum_{j_1=1}^R T_{s_1,j_1}(\epsilon_1) & T_{s_2+\dots+s_r,R} & 0 & \dots & 0 \\
\vdots & \mu_2 T_{s_3+\dots+s_r,R} \sum_{j_2=1}^R T_{s_2,j_2}(\epsilon_2) & \ddots & & \vdots \\
\vdots & & \ddots &T_{s_r,R} & 0 \\
{\displaystyle \mu_1 \dots \mu_r \hspace{-1em}\sum_{0<j_1<\dots<j_r \leq R}\hspace{-1em} T_{s_1,j_1}(\epsilon_1) \dots T_{s_r,j_r}(\epsilon_r)} & \dots & \dots & \mu_r \sum_{j_r=1}^R T_{s_r,j_r}(\epsilon_r) & 1
\end{pmatrix} \\
&\in \Mat_{r+1}(\overline K[t]),
\end{align*} 
and $\Phi'_{\fs,\fe} \in \Mat_r(\overline K[t])$ which is the upper left $r \times r$ sub-matrix of $\Phi_{\fs,\fe}$. More precisely, the $i$th row of $\Phi_{\fs,\fe}$ is $(\Phi_{i,1},\dots,\Phi_{i,i},0,\dots,0)$ where for $1 \leq k \leq i$,
\begin{align*}
\Phi_{i,k}=\mu_k \dots \mu_{i-1} T_{s_i+\dots+s_r,R} \sum_{0<j_k<\dots<j_{i-1} \leq R} T_{s_k,j_k}(\epsilon_k) \dots T_{s_{i-1},j_{i-1}}(\epsilon_{i-1}). 
\end{align*}

From now on, we write 
\begin{equation} \label{eq: twisted cyclotomic MZV}
L(\fs;\fe):=\mu(\fs;\fe) \frakL^{\text{un}}(\fs;\fe)
\end{equation}
instead of $\mu_1 \dots \mu_r \frakL^{\text{un}}(\fs;\fe)$. It follows from Eqs. \eqref{eq: CMPL at roots of unity v0} and \eqref{eq: CMPL at roots of unity 2 v0} that for all $j \in \mathbb Z^{\geq 0}$ divisible by $R$, we have
\begin{align}
L(\fs;\fe)(\theta) &=\mu_1 \dots \mu_r \frac{\Li \begin{pmatrix}
\fe \\ \fs
\end{pmatrix}}{\widetilde \pi^{w(\fs)}}, \label{eq: CMPL at roots of unity} \\
L(\fs;\fe)(\theta^{q^j}) &= (L(\fs;\fe)(\theta))^{q^j}. \label{eq: CMPL at roots of unity 2}
\end{align} 
Then the matrix given by	
\begin{align*}
& \Psi_{\fs,\fe} \\
&=
\begin{pmatrix}
\Omega^{s_1+\dots+s_r} & 0 & 0 & \dots & 0 \\
\mu_1 \frakL^{\text{un}}(s_1;\epsilon_1) \Omega^{s_2+\dots+s_r} & \Omega^{s_2+\dots+s_r} & 0 & \dots & 0 \\
\vdots & \mu_2 \frakL^{\text{un}}(s_2;\epsilon_2) \Omega^{s_3+\dots+s_r} & \ddots & & \vdots \\
\vdots & & \ddots & \ddots & \vdots \\
\mu_1 \dots \mu_{r-1} \frakL^{\text{un}}(s_1,\dots,s_{r-1};\epsilon_1,\dots,\epsilon_{r-1}) \Omega^{s_r}  & \mu_2 \dots \mu_{r-1} \frakL^{\text{un}}(s_2,\dots,s_{r-1};\epsilon_2,\dots,\epsilon_{r-1}) \Omega^{s_r} & \dots & \Omega^{s_r}& 0 \\
\mu_1 \dots \mu_r \frakL^{\text{un}}(s_1,\dots,s_r;\epsilon_1,\dots,\epsilon_r)  & \mu_2 \dots \mu_r \frakL^{\text{un}}(s_2,\dots,s_r;\epsilon_2,\dots,\epsilon_r) & \dots & \mu_r \frakL^{\text{un}}(s_r;\epsilon_r) & 1
\end{pmatrix} \\
&=
\begin{pmatrix}
\Omega^{s_1+\dots+s_r} & 0 & 0 & \dots & 0 \\
L(s_1;\epsilon_1) \Omega^{s_2+\dots+s_r} & \Omega^{s_2+\dots+s_r} & 0 & \dots & 0 \\
\vdots & L(s_2;\epsilon_2) \Omega^{s_3+\dots+s_r} & \ddots & & \vdots \\
\vdots & & \ddots & \ddots & \vdots \\
L(s_1,\dots,s_{r-1};\epsilon_1,\dots,\epsilon_{r-1}) \Omega^{s_r}  & L(s_2,\dots,s_{r-1};\epsilon_2,\dots,\epsilon_{r-1}) \Omega^{s_r} & \dots & \Omega^{s_r}& 0 \\
L(s_1,\dots,s_r;\epsilon_1,\dots,\epsilon_r)  & L(s_2,\dots,s_r;\epsilon_2,\dots,\epsilon_r) & \dots & L(s_r;\epsilon_r) & 1
\end{pmatrix} \\
&\in \text{GL}_{r+1}(\bT)
\end{align*}
satisfies
	\[ \Psi_{\fs,\fe}^{(-R)}=\Phi_{\fs,\fe} \Psi_{\fs,\fe}. \]
Thus $\Psi_{\fs,\fe}$ is a rigid analytic trivialization associated to the dual $t$-motive $\mathcal M_{\fs,\fe}$ of level $R$. We also denote by $\Psi_{\fs,\fe}'$ the upper $r \times r$ sub-matrix of $\Psi_{\fs,\fe}$. It is clear that $\Psi_\fs'$ is a rigid analytic trivialization associated to the dual $t$-motive $\mathcal M_{\fs,\fe}'$.

Further, combined with Eqs. \eqref{eq: CMPL at roots of unity} and \eqref{eq: CMPL at roots of unity 2} (see \cite[Lemma 4.4]{Har23}), the above construction of dual $t$-motives implies that $L(\fs;\fe)(\theta)$ has the MZ (multizeta) property in the sense of \cite[Definition 3.4.1]{Cha14}. By \cite[Theorem 4.7 and Lemma 4.11]{Har23} (see \cite[Proposition 4.3.1]{Cha14} for the MZV's case), we get

\begin{proposition} \label{prop: MZ property}
Let $(\fs_i;\fe_i)$ be as before for $1 \leq i \leq m$. We suppose that all the tuples of positive integers $\fs_i$ have the same weight, say $w$. Then the following assertions are equivalent:
\begin{itemize}
\item[i)] $L(\fs_1;\fe_1)(\theta),\dots,L(\fs_m;\fe_m)(\theta)$ are $K_N$-linearly independent.

\item[ii)] $L(\fs_1;\fe_1)(\theta),\dots,L(\fs_m;\fe_m)(\theta)$ are $\overline K$-linearly independent.
\end{itemize}
\end{proposition}

We end this section by mentioning that Harada \cite{Har23} also proved analogue of Goncharov's conjecture for cyclotomic MZV's, generalizing \cite{Cha14} for the MZV's case.

\subsection{Dual $t$-motives and linear combinations of CMPL's at roots of unity} \ppar \label{sec: linear combination CMPL motives}

Let $w \in \N$ be a positive integer. Let $\{(\fs_i;\fe_i)\}_{1 \leq i \leq n}$ be a collection of pairs as above such that $\fs_i$ has weight $w$. Let $\{(a_i)\}_{1 \leq i \leq n}$ be elements in $K_N$ and we set  $a_i(t):=a_i \rvert_{\theta=t} \in k_N[t]$. We write $\fs_i=(s_{i1},\dots,s_{i \ell_i}) \in \mathbb N^{\ell_i}$ and $\fe_i=(\epsilon_{i1},\dots,\epsilon_{i\ell_i}) \in (\Gamma_N)^{\ell_i}$ so that $s_{i1}+\dots+s_{i \ell_i}=w$. We introduce the set of tuples
	\[ I(\fs_i;\fe_i):=\{\emptyset,(s_{i1};\epsilon_{i1}),\dots,(s_{i1},\dots,s_{i (\ell_i-1)};\epsilon_{i1},\dots,\epsilon_{i(\ell_i-1)})\},\]
and set
	\[ I:=\cup_i I(\fs_i;\fe_i). \]

For all $(\frak t;\fe) \in I$, we set
\begin{equation} \label{eq:f}
f_{\mathfrak t,\fe}:= \sum_i a_i(t) L(s_{i(k+1)},\dots,s_{i \ell_i};\epsilon_{i(k+1)},\dots,\epsilon_{i \ell_i}),
\end{equation}
where the sum runs through the set of indices $i$ such that $(\mathfrak t;\fe)=(s_{i1},\dots,s_{i k};\epsilon_{i1},\dots,\epsilon_{i k})$ for some $0 \leq k \leq \ell_i-1$. In particular, $f_{\emptyset}= \sum_i a_i(t) L(\fs_i;\fe_i)$.

We now generalize some results of \cite[\S 4 and \S 5]{ND21} and also \cite[Theorem 2.4]{IKLNDP22} to study the linear dependence of CMPL's at roots of unity. The key point is a construction of dual $t$-motives connected to linear combinations of CMPL's at roots of unity. The proof follows the same line as the above. We write it down for completeness.

\begin{theorem} \label{theorem: linear independence}
We keep the above notation. We suppose further that $\{(\fs_i;\fe_i)\}_{1 \leq i \leq n}$ satisfies the following conditions:
\begin{itemize}
\item[(LW)] (for Lower Weights) For any weight $w'<w$, the values $L(\frak t;\fe)(\theta)$ with $(\frak t;\fe) \in I$ and $w(\frak t)=w'$ are all $K_N$-linearly independent. In particular, $L(\frak t;\fe)(\theta)$ is always nonzero.

\item[(LD)] (for Linear Dependence) There exist $a \in A_N$ and $a_i \in A_N$ for $1 \leq i \leq n$ which are not all zero such that
\begin{equation*}
a+\sum_{i=1}^n a_i L(\fs_i;\fe_i)(\theta)=0.
\end{equation*}
\end{itemize}

Then for all $(\mathfrak t;\fe) \in I$, $f_{\mathfrak t,\fe}(\theta)$ belongs to $K_N$ where $f_{\mathfrak t,\fe}$ is given as in Eq. \eqref{eq:f}.
\end{theorem}

\begin{proof}
The proof will be divided into two steps.

\medskip
\noindent {\bf Step 1.} We first construct a dual $t$-motive connected to linear combinations of CMPL's at roots of unity.

For each pair $(\fs_i;\fe_i)$ we have attached to it a matrix $\Phi_{\fs_i,\fe_i}$. For $\fs_i=(s_{i1},\dots,s_{i \ell_i}) \in \mathbb N^{\ell_i}$ and $\fe_i=(\epsilon_{i1},\dots,\epsilon_{i\ell_i}) \in (\Gamma_N)^{\ell_i}$ we recall
	\[ I(\fs_i;\fe_i)=\{\emptyset,(s_{i1};\epsilon_{i1}),\dots,(s_{i1},\dots,s_{i (\ell_i-1)};\epsilon_{i1},\dots,\epsilon_{(\ell_i-1)})\}, \]
and $I:=\cup_i I(\fs_i;\fe_i)$.

We now construct a new matrix $\Phi'$ by merging the same rows of $\Phi_{\fs_1,\fe_1}',\ldots,\Phi_{\fs_n,\fe_n}'$ as follows. The matrix $\Phi'$ will be a matrix indexed by elements of $I$, say $\Phi'=\left(\Phi'_{(\mathfrak t;\fe),(\mathfrak t';\fe')}\right)_{(\mathfrak t;\fe),(\mathfrak t';\fe') \in I} \in \Mat_{|I|}(\overline K[t])$. For the row which corresponds to the empty pair $\emptyset$ we put
\begin{align*}
\Phi'_{\emptyset,(\mathfrak t';\fe')}=
\begin{cases}
T_{w,R} & \text{if } (\mathfrak t';\fe')=\emptyset, \\
0 & \text{otherwise}.
\end{cases}
\end{align*}
For the row indexed by $(\mathfrak t;\fe)=(s_{i1},\dots,s_{i j};\epsilon_{i1},\dots,\epsilon_{ij})$ for some $i$ and $1 \leq j \leq \ell_i-1$ we put
\begin{align*}
& \Phi'_{(\mathfrak t;\fe),(\mathfrak t';\fe')} \\
&=
\begin{cases}
T_{w-w(\mathfrak t'),R} & \text{if } (\mathfrak t';\fe')=(\mathfrak t;\fe), \\
\mu_{i(k+1)} \dots \mu_{ij} T_{w-w(\mathfrak t),R} \sum^{*} T_{s_{i(k+1)},m_{i(k+1)}}(\epsilon_{i(k+1)}) \dots T_{s_{ij},m_{ij}}(\epsilon_{ij}) & \text{if } (\mathfrak t';\fe')=(s_{i1},\dots,s_{ik};\epsilon_{i1},\dots,\epsilon_{ik}), \\
0 & \text{otherwise}.
\end{cases}
\end{align*}
Here $\sum^* = \sum_{0<m_{i(k+1)}<\dots<m_{ij} \leq R}$.  Note that $\Phi_{\fs_i,\fe_i}'=\left(\Phi'_{(\mathfrak t;\fe),(\mathfrak t';\fe')}\right)_{(\mathfrak t;\fe),(\mathfrak t';\fe') \in I(\fs_i;\fe_i)}$ for all $i$.

We define $\Phi \in \Mat_{|I|+1}(\overline K[t])$ by
\begin{align*}
\Phi=\begin{pmatrix}
\Phi' & 0  \\
\bv & 1
\end{pmatrix} \in \Mat_{|I|+1}(\overline K[t]), \quad \bv=(v_{\mathfrak t,\fe})_{(\mathfrak t;\fe) \in I} \in \Mat_{1 \times|I|}(\overline K[t]),
\end{align*}
where
\begin{align*}
v_{\mathfrak t,\fe} = \sum_i a_i(t) \mu_{i(k+1)} \dots \mu_{i \ell_i} \sum_{0<m_{i(k+1)}<\dots<m_{i \ell_i} \leq R} T_{s_{i(k+1)},m_{i(k+1)}}(\epsilon_{i(k+1)}) \dots T_{s_{i \ell_i},m_{i \ell_i}}(\epsilon_{i \ell_i}).
\end{align*} 
Here the first sum runs through the set of indices $i$ such that $(\mathfrak t;\fe)=(s_{i1},\dots,s_{i k};\epsilon_{i1},\dots,\epsilon_{i k})$ for some $0 \leq k \leq \ell_i-1$.

We now introduce a rigid analytic trivialization matrix $\Psi$ for $\Phi$. We define $\Psi'=\left(\Psi'_{(\mathfrak t;\fe),(\mathfrak t';\fe')}\right)_{(\mathfrak t;\fe),(\mathfrak t';\fe') \in I} \in  \text{GL}_{|I|}(\bT)$ as follows. For the row which corresponds to the empty pair $\emptyset$ we define
\begin{align*}
\Psi'_{\emptyset,(\mathfrak t';\fe')}=
\begin{cases}
\Omega^w & \text{if } (\mathfrak t';\fe')=\emptyset, \\
0 & \text{otherwise}.
\end{cases}
\end{align*}
For the row indexed by $(\mathfrak t;\fe)=(s_{i1},\dots,s_{i j};\epsilon_{i1},\dots,\epsilon_{i j})$ for some $i$ and $1 \leq j \leq \ell_i-1$ we put
\begin{align*}
&\Psi'_{(\mathfrak t;\fe),(\mathfrak t';\fe')}= \\
&
\begin{cases}
L(\mathfrak t;\fe) \Omega^{w-w(\mathfrak t)} & \text{if } (\mathfrak t';\fe')=\emptyset, \\
L(s_{i(k+1)},\dots,s_{ij};\epsilon_{i(k+1)},\dots,\epsilon_{ij}) \Omega^{w-w(\mathfrak t)} & \text{if } (\mathfrak t';\fe')=(s_{i1},\dots,s_{i k};\epsilon_{i1},\dots,\epsilon_{i k}) \text{ for some } 1 \leq k \leq j, \\
0 & \text{otherwise}.
\end{cases}
\end{align*}
Note that $\Psi_{\fs_i,\fe_i}'=\left(\Psi'_{(\mathfrak t;\fe),(\mathfrak t';\fe')}\right)_{(\mathfrak t;\fe),(\mathfrak t';\fe') \in I(\fs_i;\fe_i)}$ for all $i$.

We define $\Psi \in \text{GL}_{|I|+1}(\bT)$ by
\begin{align*}
\Psi=\begin{pmatrix}
\Psi' & 0  \\
\bff & 1
\end{pmatrix} \in \text{GL}_{|I|+1}(\bT), \quad \bff=(f_{\mathfrak t,\fe})_{\mathfrak t \in I} \in \Mat_{1 \times|I|}(\bT).
\end{align*}
Here we recall (see Eq. \eqref{eq:f})
\begin{equation*}
f_{\mathfrak t,\fe}= \sum_i a_i(t) L(s_{i(k+1)},\dots,s_{i \ell_i};\epsilon_{i(k+1)},\dots,\epsilon_{i \ell_i})
\end{equation*}
where the sum runs through the set of indices $i$ such that $(\mathfrak t;\fe)=(s_{i1},\dots,s_{i k};\epsilon_{i1},\dots,\epsilon_{i k})$ for some $0 \leq k \leq \ell_i-1$. In particular, $f_{\emptyset}= \sum_i a_i(t) L(\fs_i;\fe_i)$.

By construction and by \S \ref{sec:CMPL motives}, we get $\Psi^{(-R)}=\Phi \Psi$.

\medskip
\noindent {\bf Step 2.} We apply the Anderson-Brownawell-Papanikolas criterion (see Theorem \ref{thm:ABP}) to the previous construction. 

In fact, we define
\begin{align*}
\widetilde \Phi=\begin{pmatrix}
1 & 0  \\
0 & \Phi
\end{pmatrix} \in  \Mat_{|I|+2}(\overline K[t])
\end{align*}
and consider the vector constructed from the first column vector of $\Psi$
\begin{align*}
\widetilde \psi=\begin{pmatrix}
1 \\
\Psi_{(\mathfrak t;\fe),\emptyset}' \\
f_\emptyset
\end{pmatrix}_{(\mathfrak t;\fe) \in I}.
\end{align*}
Then we have $\widetilde \psi^{(-R)}=\widetilde \Phi \widetilde \psi$.

We also observe that for all $(\mathfrak t;\fe) \in I$ we have $\Psi_{(\mathfrak t;\fe),\emptyset}'=L(\mathfrak t;\fe) \Omega^{w-w(\mathfrak t)}$. Further,
\begin{align*}
a+f_\emptyset(\theta)=a+\sum_i a_i L(\fs_i;\fe_i)(\theta)=0.
\end{align*}
By Theorem \ref{thm:ABP} with $\rho=(a,0,\dots,0,1)$ we deduce that there exists $\bh=(g_0,g_{\mathfrak t,\fe},g) \in \Mat_{1 \times (|I|+2)}(\overline K[t])$ such that $\bh \psi=0$, and that
$g_{\mathfrak t,\fe}(\theta)=0$ for $(\mathfrak t,\fe) \in I$, $g_0(\theta)=a$ and $g(\theta)=1 \neq 0$. If we put $ \bg:=(1/g)\bh \in \Mat_{1 \times (|I|+2)}(\overline K(t))$, then all the entries of $ \bg$ are regular at $t=\theta$.

Now we have
\begin{align} \label{eq: reduction}
( \bg- \bg^{(-R)} \widetilde \Phi) \widetilde \psi= \bg \widetilde \psi-( \bg  \widetilde \psi)^{(-R)}=0.
\end{align}
We write $ \bg- \bg^{(-R)} \widetilde \Phi=(B_0,B_{\mathfrak t},0)_{\mathfrak t \in I}$. We claim that $B_0=0$ and $B_{\mathfrak t,\fe}=0$ for all $(\mathfrak t;\fe) \in I$. In fact, expanding \eqref{eq: reduction} we obtain
\begin{equation} \label{eq: B}
B_0+\sum_{\mathfrak t \in I} B_{\mathfrak t,\fe} L(\mathfrak t;\fe) \Omega^{w-w(\mathfrak t)}=0.
\end{equation}

By \eqref{eq: CMPL at roots of unity 2} we see that for $(\mathfrak t;\fe) \in I$ and $j \in \N$ divisible by $R$,
\begin{equation*}
L(\mathfrak t;\fe)(\theta^{q^j})=(L(\mathfrak t;\fe)(\theta))^{q^j}
\end{equation*}
which is nonzero by Condition (LW).

First, as the function $\Omega$ has a simple zero at $t=\theta^{q^k}$ for $k \in \N$, specializing \eqref{eq: B} at $t=\theta^{q^j}$ yields $B_0(\theta^{q^j})=0$ for $j \geq 1$ divisible by $R$. Since $B_0$ belongs to $\overline K(t)$, it follows that $B_0=0$.

Next, we put $w_0:=\max_{(\mathfrak t;\fe) \in I} w(\mathfrak t)$ and denote by $I(w_0)$ the set of $(\mathfrak t;\fe) \in I$ such that $w(\mathfrak t)=w_0$. Then dividing \eqref{eq: B} by $\Omega^{w-w_0}$ yields
\begin{equation} \label{eq: B1}
\sum_{(\mathfrak t;\fe) \in I} B_{\mathfrak t,\fe} L(\mathfrak t;\fe) \Omega^{w_0-w(\mathfrak t)}=\sum_{(\mathfrak t;\fe) \in I(w_0)} B_{\mathfrak t,\fe} L(\mathfrak t;\fe)+\sum_{(\mathfrak t;\fe) \in I \setminus I(w_0)} B_{\mathfrak t,\fe} L(\mathfrak t;\fe) \Omega^{w_0-w(\mathfrak t)}=0.
\end{equation}
Since each $B_{\mathfrak t,\fe}$ belongs to $\overline K(t)$, they are defined at $t=\theta^{q^j}$ for $j \gg 1$. Note that the function $\Omega$ has a simple zero at $t=\theta^{q^k}$ for $k \in \N$. Specializing \eqref{eq: B1} at $t=\theta^{q^j}$ and using \eqref{eq: CMPL at roots of unity 2} yields
	\[ \sum_{(\mathfrak t;\fe) \in I(w_0)} B_{\mathfrak t,\fe}(\theta^{q^j}) (L(\mathfrak t;\fe)(\theta))^{q^j}=0 \]
for $j \gg 1$ divisible by $R$.

We claim that $B_{\mathfrak t,\fe}(\theta^{q^j})=0$ for $j \gg 1$ divisible by $R$ and for all $(\mathfrak t;\fe) \in I(w_0)$. Otherwise, we get a non-trivial $\overline K$-linear relation among $L(\mathfrak t;\fe)(\theta)$ with $(\frak t;\fe) \in I$ of weight $w_0$. By Proposition \ref{prop: MZ property} we deduce a non-trivial $K_N$-linear relation among $L(\mathfrak t;\fe)(\theta)$ with $(\frak t;\fe) \in I(w_0)$, which contradicts with Condition $(LW)$.
Now we know that $B_{\mathfrak t,\fe}(\theta^{q^j})=0$ for $j \gg 1$  divisible by $R$ and for all $(\mathfrak t;\fe) \in I(w_0)$. Since each $B_{\mathfrak t,\fe}$ belongs to $\overline K(t)$, it follows that
$B_{\mathfrak t,\fe}=0$ for all $(\mathfrak t;\fe) \in I(w_0)$.

Next, we put $w_1:=\max_{(\mathfrak t;\fe) \in I \setminus I(w_0)} w(\mathfrak t)$ and denote by $I(w_1)$ the set of $(\mathfrak t;\fe) \in I$ such that $w(\mathfrak t)=w_1$. Dividing \eqref{eq: B} by $\Omega^{w-w_1}$ and specializing at $t=\theta^{q^j}$ yields
	\[ \sum_{(\mathfrak t;\fe) \in I(w_1)} B_{\mathfrak t,\fe}(\theta^{q^j}) (L(\mathfrak t;\fe)(\theta))^{q^j}=0 \]
for $j \gg 1$. Since $w_1<w$, by Proposition \ref{prop: MZ property} and Condition $(LW)$ again we deduce that $B_{\mathfrak t,\fe}(\theta^{q^j})=0$ for $j \gg 1$  divisible by $R$ and for all $(\mathfrak t;\fe) \in I(w_1)$. Since each $B_{\mathfrak t,\fe}$ belongs to $\overline K(t)$, it follows that
$B_{\mathfrak t,\fe}=0$ for all $(\mathfrak t;\fe) \in I(w_1)$. Repeating the previous arguments we deduce that $B_{\mathfrak t,\fe}=0$ for all $(\mathfrak t;\fe) \in I$ as required.

We have proved that $ \bg- \bg^{(-R)} \widetilde \Phi=0$. Thus
\begin{align*}
\begin{pmatrix}
1 & 0 & 0 \\
0 & \text{Id} & 0 \\
g_0/g & (g_{\mathfrak t,\fe}/g)_{(\mathfrak t;\fe) \in I}  & 1
\end{pmatrix}^{(-R)}
\begin{pmatrix}
1 & 0  \\
0 & \Phi
\end{pmatrix}
= \begin{pmatrix}
1 & 0 & 0 \\
0 & \Phi' & 0 \\
0 & 0 & 1
\end{pmatrix}
\begin{pmatrix}
1 & 0 & 0 \\
0 & \text{Id} & 0 \\
g_0/g & (g_{\mathfrak t,\fe}/g)_{(\mathfrak t;\fe) \in I}  & 1
\end{pmatrix}.
\end{align*}
By \cite[Prop. 2.2.1]{CPY19} we see that the common denominator $b$ of $g_0/g$ and $g_{\mathfrak t,\fe}/g$ for $(\mathfrak t,\fe) \in I$ belongs to $k_N[t] \setminus \{0\}$. If we put $\delta_0=bg_0/g$ and $\delta_{\mathfrak t,\fe}=b g_{\mathfrak t,\fe}/g$ for $(\mathfrak t,\fe) \in I$ which belong to $\overline K[t]$ and $\delta:=(\delta_{\mathfrak t,\fe})_{\mathfrak t \in I} \in \Mat_{1 \times |I|}(\overline K[t])$, then $\delta_0^{(-R)}=\delta_0$ and
\begin{align} \label{eq:equation for delta}
\begin{pmatrix}
\text{Id} & 0 \\
\delta & 1
\end{pmatrix}^{(-R)}
\begin{pmatrix}
\Phi' & 0  \\
b \bv & 1
\end{pmatrix}
= \begin{pmatrix}
 \Phi' & 0 \\
0 & 1
\end{pmatrix}
\begin{pmatrix}
\text{Id} & 0 \\
\delta & 1
\end{pmatrix}.
\end{align}

If we put $X:=\begin{pmatrix}
\text{Id} & 0 \\
\delta & 1
\end{pmatrix}
\begin{pmatrix}
\Psi' & 0  \\
b \bff & 1
\end{pmatrix}$, then $X^{(-R)}=\begin{pmatrix}
 \Phi' & 0 \\
0 & 1
\end{pmatrix} X$. By \cite[\S 4.1.6]{Pap08} there exist $\nu_{\mathfrak t,\fe} \in k_N(t)$ for $(\mathfrak t,\fe) \in I$ such that if we set $\nu=(\nu_{\mathfrak t,\fe})_{(\mathfrak t,\fe) \in I} \in \Mat_{1 \times |I|}(k_N(t))$,
\begin{align*}
X=\begin{pmatrix}
\Psi' & 0 \\
0 & 1
\end{pmatrix}
\begin{pmatrix}
\text{Id} & 0 \\
\nu & 1
\end{pmatrix}.
\end{align*}
Thus the equation $\begin{pmatrix}
\text{Id} & 0 \\
\delta & 1
\end{pmatrix}
\begin{pmatrix}
\Psi' & 0  \\
b \bff & 1
\end{pmatrix}=\begin{pmatrix}
\Psi' & 0 \\
0 & 1
\end{pmatrix}
\begin{pmatrix}
\text{Id} & 0 \\
\nu & 1
\end{pmatrix}$ implies
\begin{equation} \label{eq:nu}
\delta \Psi'+b \bff=\nu.
\end{equation}
The left-hand side belongs to $\bT$, so does the right-hand side. Thus $\nu=(\nu_{\mathfrak t,\fe})_{(\mathfrak t,\fe) \in I} \in \Mat_{1 \times |I|}(k_N[t])$. For any $j \in \N$, by specializing Eq. \eqref{eq:nu} at $t=\theta^{q^j}$ for $j \in \mathbb N$ divisible by $R$ and using Eq. \eqref{eq: CMPL at roots of unity 2} and the fact that $\Omega$ has a simple zero at $t=\theta^{q^j}$ we deduce that
	\[ \bff(\theta)=\nu(\theta)/b(\theta). \]
Thus for all $(\mathfrak t,\fe) \in I$, $f_{\mathfrak t,\fe}(\theta)$ given as in Eq. \eqref{eq:f} belongs to $K_N$. The proof is complete.
\end{proof}

\subsection{Linear relations among CMPL's at roots of unity and Carlitz's period} \ppar

The main result for linear relations among CMPL's roots of unity reads as follows:

\begin{theorem} \label{thm: trans CMPL at roots of unity}
Let $w \in \N$. We recall that $\mathcal{S}_w$ denotes the set of positive tuples $\fs = (s_1, \dots, s_n)$ of weight $w$ such that $ q \nmid s_i$ for all $i$ and $\mathcal{CS}_{N,w}$ denotes the set of CMPL's at roots of unity $\Li \begin{pmatrix}
 \fe  \\
\fs  \end{pmatrix}$ such that $\mathfrak s= (s_1, \dots, s_n) \in \mathcal{S}_w$  and $\fe=(\epsilon_1,\dots,\epsilon_n) \in (\Gamma_N)^n$ (see \S \ref{sec:strong Brown CMPL}). 

1) Suppose that we have a non-trivial relation denoted by $R(N,w)$:
\begin{equation*}
a+\sum_{\fs_i \in \mathcal{S}_w} a_i L(\fs_i;\fe_i)(\theta)=0, \quad \text{for some $a, a_i \in \overline{\mathbb F}_q(\theta)$.}
\end{equation*}
Then $q-1 \mid w$, $a \neq 0$ and $\fe_i=(1,\dots,1)$ for all $i$.

2) Further, if $q-1 \mid w$, then there is a unique relation
\begin{equation*}
1+\sum_{\fs_i \in \mathcal{S}_w} a_i L(\fs_i;1)(\theta)=0, \quad \text{for $a_i \in K$.}
\end{equation*}
\end{theorem}

We postpone the proof of Theorem \ref{thm: trans CMPL at roots of unity} to the next section. We derive an immediate consequence of Theorem \ref{thm: trans CMPL at roots of unity}:

\begin{theorem} \label{thm: transcendence}
The CMPL's at roots of unity in $\mathcal{CS}_{N,w}$ are linearly independent over $K_N$.
\end{theorem}


\section{CMPL's at roots of unity and Carlitz's period} \label{sec:transcendental part}

This section is devoted to proving Theorem \ref{thm: trans CMPL at roots of unity}. We keep the notation of the previous section and recall $R:=[k_N:k]<N$. 

\subsection{Auxiliary lemmas} \ppar

We begin this section by recalling several lemmas in \cite{IKLNDP22}.

\begin{lemma} \label{lem: gamma_i}
Let $\epsilon_i \in \Gamma_N$ be different elements such that the values $\epsilon_i^R$ are pairwise distinct. We denote by $\mu_i \in \overline{\F}_q$ a $(q^R-1)$ root of $\epsilon_i^R$. Then $\mu_i$ are all $k_N$-linearly independent.
\end{lemma}

\begin{proof}
See \cite[Lemma 3.1]{IKLNDP22}.
\end{proof}

\begin{lemma} \label{lem: same character}
Let $ \Li \begin{pmatrix}
 \fe_i  \\
\fs_i \end{pmatrix} \in \mathcal{CS}_{N,w}$ and $a_i \in K_N$ satisfying
\begin{equation*}
\sum_i a_i L(\fs_i;\fe_i)(\theta)=0.
\end{equation*}
For $\epsilon \in \Gamma_N$ we denote by $I(\epsilon^R)=\{i:\, \chi(\fe_i)^R=\epsilon^R\}$ the set of pairs such that the $R$th power of the corresponding character equals $\epsilon^R$. Then for all $\epsilon \in \Gamma_N$,
	\[ \sum_{i \in I(\epsilon^R)} a_i L(\fs_i;\fe_i)(\theta)=0. \]
\end{lemma}

\begin{proof}
See \cite[Lemma 3.2]{IKLNDP22}.
\end{proof}


\begin{lemma} \label{lem: KuanLin 2}
Let $m \in \mathbb N$, $\epsilon \in \Gamma_N$, $\delta \in \overline K[t]$ and $F(t,\theta) \in \overline{\F}_q[t,\theta]$ (resp. $F(t,\theta) \in k_N[t,\theta]$) satisfying
	\[ \epsilon\delta = \delta^{(-R)}T_{m,R}+F^{(-R)}(t,\theta). \]
Then $\delta \in \overline{\F}_q[t,\theta]$ (resp. $\delta \in k_N[t,\theta]$) and
	\[ \deg_\theta \delta \leq \max\left\{\frac{qm}{q-1},\frac{\deg_\theta F(t,\theta)}{q^R}\right\}. \]
\end{lemma}

\begin{proof}
This lemma is a slight generalization of \cite[Theorem 2]{KL16}, which proved the case $N=1$ (hence $R=1$ and $T_{m,R}=(t-\theta)^m$).

By twisting $R$ times the equality $\epsilon\delta = \delta^{(-R)}T_{m,R}+F^{(-R)}(t,\theta)$ and the fact that $\epsilon^{(R)}=\epsilon$, we get
	\[ \epsilon\delta^{(R)} = \delta (T_{m,R})^{(R)}+F(t,\theta). \]
We put $n=\deg_t \delta$ and express $\delta=a_n t^n + \dots +a_1t+a_0 \in \overline K[t]$
with $a_0,\dots,a_n \in \overline K$. For $i<0$ we put $a_i=0$.

Recall that $T_{m,R}=(t-\theta)^m ((t-\theta)^m)^{(-1)} \dots ((t-\theta)^m)^{(-(R-1))}$, thus 
\begin{align*}
T_{m,R}^{(R)}&=((t-\theta)^m)^{(R)} ((t-\theta)^m)^{(R-1)} \dots ((t-\theta)^m)^{(1)} \\
&=(t-\theta^{q^R})^m \dots (t-\theta^q)^m \\
&=t^{mR} + c_{mR-1} t^{mR-1} +\dots +c_1 t + c_0,
\end{align*}
where $c_i \in \Fq[\theta]$. We put $c_{mR}=1$ and $c_i=0$ for $i<0$.

Since 
\[ \deg_t \delta^{(R)}=\deg_t \delta=n <\deg_t(\delta (T_{m,R})^{(R)})=n+mR, \]
it follows that $\deg_t F(t,\theta)=n+mR$. Thus we write $F(t,\theta)=b_{n+mR} t^{n+mR}+\dots+b_1 t+b_0$ with $b_0,\dots,b_{n+mR} \in \overline{\F}_q[\theta]$. Plugging into the previous equation, we obtain
\begin{align*}
&\epsilon(a_n^{q^R} t^n + \dots +a_0^{q^R}) \\
&= (a_n t^n + \dots +a_0)(T_{m,R})^{(R)}+b_{n+mR} t^{n+mR}+\dots+b_0 \\
&= (a_n t^n + \dots +a_0)(t^{mR} + c_{mR-1} t^{mR-1} + \dots + c_0)+b_{n+mR} t^{n+mR}+\dots+b_0.
\end{align*}
Comparing the coefficients $t^j$ for $n+1 \leq j \leq n+mR$ yields
	\[ a_{j-mR}+\sum_{i=j-mR+1}^{n} a_i c_{j-i}+b_j=0. \]
Since $b_j \in \overline{\F}_q[\theta]$ for all $n+1 \leq j \leq n+mR$, we can show by descending induction that $a_j \in \overline{\F}_q[\theta]$ for all $n+1-mR \leq j \leq n$.

If $n+1-mR \leq 0$, then we are done. Otherwise, comparing the coefficients $t^j$ for $mR \leq j \leq n$ yields
	\[ a_{j-mR}+\sum_{i=j-mR+1}^{n} a_i c_{j-i}+b_j-\epsilon a_j^{q^R}=0. \]
Since $b_j \in \overline{\F}_q[\theta]$ for all $mR \leq j \leq n$ and $a_j \in \overline{\F}_q[\theta]$ for all $n+1-mR \leq j \leq n$, we can show by descending induction that $a_j \in \overline{\F}_q[\theta]$ for all $0 \leq j \leq n-mR$. We conclude that $\delta \in \overline{\F}_q[t,\theta]$.

We now show that $\deg_\theta \delta \leq \max\{\frac{qm}{q-1},\frac{\deg_\theta F(t,\theta)}{q^R}\}$. Otherwise, suppose that $\deg_\theta \delta > \max\{\frac{qm}{q-1},\frac{\deg_\theta F(t,\theta)}{q^R}\}$. Then $\deg_\theta \delta^{(R)}=q^R \deg_\theta \delta$. It implies that $\deg_\theta \delta^{(R)} > \deg_\theta (\delta(T_{m,R})^{(R)})=\deg_\theta \delta+m \frac{q^{R+1}-q}{q-1}$ and $\deg_\theta \delta^{(R)} >\deg_\theta F(t,\theta)$. Hence we get
	\[\deg_\theta (\epsilon \delta^{(R)})= \deg_\theta \delta^{(R)} > \deg_\theta(T_{m,R})^{(R)}+F(t,\theta), \]
which is a contradiction.
\end{proof}

\begin{lemma} \label{lem: relation Carlitz period}
Recall that the set  $\mathcal{S}_w$ consists of positive tuples $\fs = (s_1, \dots, s_n)$ of weight $w$ such that $ q \nmid s_i$ for all $i$ as in \S \ref{sec:strong Brown CMPL}. Let $w \in \mathbb N$ such that $q-1 \mid w$. Then there exists a linear relation
\begin{equation*}
1+\sum_{\fs_i \in \mathcal{S}_w} a_i L(\fs_i;1)(\theta)=0, \quad \text{with $a_i \in K$.}
\end{equation*}

\end{lemma}

\begin{proof}
Since $q-1 \mid w$, it follows that $\zeta_A(w)/\widetilde{\pi}^w$ belongs to $K^\times$. Theorems \ref{thm:bridge} and \ref{thm: strong BD CMPL} applied to $N=1$ imply immediately the desired assertion.
\end{proof}

The rest of this section presents a proof of Theorem \ref{thm: trans CMPL at roots of unity}.

\subsection{Setup} \ppar

The proof of Theorem \ref{thm: trans CMPL at roots of unity} is by  induction on $N+w$.

If $N+w=2$, then $N=w=1$. Thus we have only one tuple $(\fs_1,\fe_1)=(1,1)$. Suppose that we have a non-trivial relation
\begin{equation*}
a+a_1 L(1;1)=0, \quad \text{for $a, a_1 \in \overline{\mathbb F}_q(\theta)$.}
\end{equation*}
Then $a \neq 0$ and it follows that $q-1 \mid 1$. Hence $q=2$ and Theorem \ref{thm: trans CMPL at roots of unity} follows as $L(1;1)$ belongs to $K$.

Let $M$ be a positive integer with $M>2$. Suppose that 
\begin{center}
$H(N,w)$: \quad Theorem \ref{thm: trans CMPL at roots of unity} holds for all pairs $(N,w)$ such that $N+w<M$. 
\end{center}

We claim that Theorem \ref{thm: trans CMPL at roots of unity} holds for all pairs $(N,w)$ such that $N+w=M$. In what follows, we fix such a pair $(N,w)$. We recall $R=[k_N:k]<N$.

Suppose that we have a non-trivial relation
\begin{equation} \label{eq: non-trivial relation MZV}
R(N,w): \quad a+\sum_{\fs_i \in \mathcal{S}_w} a_i L(\fs_i;\fe_i)(\theta)=0, \quad \text{for $a, a_i \in \overline{\mathbb F}_q(\theta)$.}
\end{equation}

We claim that we can suppose $a, a_i \in K_N$. In fact, if $a=0$, then the claim follows from Proposition \ref{prop: MZ property}. If $q-1 \nmid w$, then by Eq. \eqref{eq: CMPL at roots of unity} any linear relation
	\[ a+\sum_{\fs_i \in \mathcal{S}_w} a_i L(\fs_i;\fe_i)(\theta)=0 \]
with $a,a_i \in \overline{\mathbb F}_q(\theta)$ implies that $a=0$. By the previous discussion, we are done. Otherwise, we have $q-1 \mid w$ and $a \neq 0$. By Lemma \ref{lem: relation Carlitz period}, there exists a linear relation
\[ 1+\sum_{\frak t_i \in \mathcal{S}_w} b_i L(\frak t_i;1)(\theta)=0 \]
with $b_i \in K$. Thus
	\[ -a\sum_{\frak t_i \in \mathcal{S}_w} b_i L(\frak t_i;1)(\theta)+\sum_{\fs_i \in \mathcal{S}_w} a_i L(\fs_i;\fe_i)(\theta)=0. \]
By Proposition \ref{prop: MZ property} again, we can suppose that $a, a_i \in K_N$. 

Further, by Lemma \ref{lem: same character} we can suppose further that $\fe_i^R$ has the same character, i.e., there exists $\epsilon \in \Gamma_N$ such that for all $i$,
\begin{equation} \label{eq: same character}
\chi(\fe_i)^R=(\epsilon_{i1} \dots \epsilon_{i\ell_i})^R=\epsilon^R.
\end{equation}

\subsection{Trivial characters} \ppar

We now apply Theorem \ref{theorem: linear independence} to our setting of CMPL's at roots of unity. We know that the hypotheses are verified:
\begin{itemize}
\item[(LW)] By the induction hypothesis, for any weight $w'<w$, the values $L(\frak t;\fe)(\theta)$ with $(\frak t;\fe) \in I$ and $w(\frak t)=w'$ are all $K_N$-linearly independent.

\item[(LD)] By Eq. \eqref{eq: non-trivial relation MZV}, there exist $a \in A_N$ and $a_i \in A_N$ for $1 \leq i \leq n$ which are not all zero such that
\begin{equation*}
a+\sum_{i=1}^n a_i L(\fs_i;\fe_i)(\theta)=0.
\end{equation*}
\end{itemize}
Thus Theorem \ref{theorem: linear independence} implies that for all $(\mathfrak t;\fe) \in I$, $f_{\mathfrak t,\fe}(\theta)$ belongs to $K_N$ where $f_{\mathfrak t,\fe}$ is given by
\begin{equation*}
f_{\mathfrak t;\fe}:= \sum_i a_i(t) L(s_{i(k+1)},\dots,s_{i \ell_i};\epsilon_{i(k+1)},\dots,\epsilon_{i \ell_i}).
\end{equation*}
Here, the sum runs through the set of $i$ such that $(\mathfrak t;\fe)=(s_{i1},\dots,s_{i k};\epsilon_{i1},\dots,\epsilon_{i k})$ for some $0 \leq k \leq \ell_i-1$.

We derive a direct consequence of this result. Let $(\mathfrak t;\fe) \in I$ and $\mathfrak t \neq \emptyset$. Then $(\mathfrak t;\fe)=(s_{i1},\dots,s_{i k};\epsilon_{i1},\dots,\epsilon_{i k})$ for some $i$ and $1 \leq k \leq \ell_i-1$. We denote by $J(\mathfrak t;\fe)$ the set of all such $i$. We know that there exists $b \in K_N$ such that
\begin{equation*}
b+f_{\mathfrak t;\fe}(\theta)=0,
\end{equation*}
or equivalently,
\begin{equation*}
b+\sum_{i \in J(\frak t;\fe)} a_i L(s_{i(k+1)},\dots,s_{i \ell_i};\epsilon_{i(k+1)},\dots,\epsilon_{i \ell_i})(\theta)=0.
\end{equation*}
The CMPL's at roots of unity appearing in the above equality belong to $\mathcal{CS}_{N,w-w(\mathfrak t)}$. By the induction hypothesis applied to the pair $(N,w-w(\frak t))$, we deduce that $\epsilon_{i(k+1)}=\dots=\epsilon_{i \ell_i}=1$. Further, if $q-1 \nmid w-w(\frak t)$, then $a_i=0$ for all $i \in J(\frak t;\fe)$. Therefore,  letting $(\fs_i,\fe_i)=(s_{i1},\dots,s_{i \ell_i};\epsilon_{i1},\dots,\epsilon_{i \ell_i})$ we can suppose that $s_{i2},\dots,s_{i \ell_i}$ are all divisible by $q-1$ and $\epsilon_{i2}=\dots=\epsilon_{i \ell_i}=1$. In particular, for all $i$, 
\begin{equation*}
\epsilon_{i1}^R=\chi(\fe_i)^R=\epsilon^R.
\end{equation*}

\subsection{$\sigma$-linear equations of higher order} \ppar

Now we want to solve Eq. \eqref{eq:equation for delta} and we can assume that $b=1$. We define
	\[ J:=I \cup \{(\fs_i;\fe_i)\}\]
For $(\mathfrak t;\fe) \in J$, 
\begin{itemize}
\item if $(\mathfrak t;\fe) \neq (\emptyset,\emptyset)$, then we put
	\[ m_{\frak t}:=\frac{w-w(\frak t)}{q-1} \in \mathbb Z^{\geq 0}. \]

\item we denote by $J_1(\frak t;\fe)$ consisting of $(\frak t';\fe') \in J$ such that there exist $i$ and $0 \leq j <\ell_i$ so that $(\frak t;\fe)=(s_{i1},s_{i2},\dots,s_{ij};\epsilon_{i1},1,\dots,1)$ and $(\frak t',\fe')=(s_{i1},s_{i2},\dots,s_{i(j+1)};\epsilon_{i1},1,\dots,1))$;

\item we denote by $J_\infty(\frak t;\fe)$ consisting of $(\frak t';\fe') \in J$ such that there exist $i$ and $0 \leq j<k \leq \ell_i$ so that $(\frak t;\fe)=(s_{i1},s_{i2},\dots,s_{ij};\epsilon_{i1},1,\dots,1)$ and $(\frak t',\fe')=(s_{i1},s_{i2},\dots,s_{ik};\epsilon_{i1},1,\dots,1)$. 
\end{itemize}
Thus $J_1(\frak t;\fe) \subset J_\infty(\frak t;\fe)$. Further, for $(\frak t;\fe)=(\fs_i;\fe_i)$, both $J_1(\frak t;\fe)$ and $J_\infty(\frak t;\fe)$ are the empty set. 

Here letting $(\frak t;\fe)=(s_{i1},s_{i2},\dots,s_{ij};\epsilon_{i1},1,\dots,1)$ for some $i$ and $0 \leq j <\ell_i$ and $(\mathfrak t';\fe')=(s_{i1},s_{i2},\dots,s_{i k};\epsilon_{i1},1,\dots,1) \in J_\infty(\frak t;\fe)$ with $j < k \leq \ell_i$, we recall (see the proof of Theorem \ref{theorem: linear independence})
\begin{itemize}
\item $\Phi'_{(\mathfrak t;\fe),(\mathfrak t;\fe)}=T_{w-w(\mathfrak t),R}$.

\item If $j=0$ that means $(\frak t;\fe)=\emptyset$, then
\begin{align*}
\Phi'_{(\mathfrak t';\fe'),\emptyset}=\mu_{i1} T_{w-w(\mathfrak t'),R} \sum_{0<m_{i1}<\dots<m_{ik} \leq R} \epsilon_{i1}^{-m_{i1}} T_{s_{i1},m_{i1}} \dots T_{s_{ik},m_{ik}}.
\end{align*}

\item If $j>0$ that means $(\frak t;\fe) \neq \emptyset$, then
\begin{align*}
\Phi'_{(\mathfrak t';\fe'),(\mathfrak t;\fe)}=T_{w-w(\mathfrak t'),R} \sum_{0<m_{i(j+1)}<\dots<m_{ik} \leq R} T_{s_{i(j+1)},m_{i(j+1)}} \dots T_{s_{ik},m_{ik}}.
\end{align*}
\end{itemize}

It is clear that \eqref{eq:equation for delta} is equivalent finding $(\delta_{\mathfrak t,\fe})_{(\mathfrak t;\fe) \in J} \in \Mat_{1 \times |J|}(\overline K[t])$ such that
\begin{equation} \label{eq: delta}
\delta_{\mathfrak t,\fe}=\delta_{\mathfrak t,\fe}^{(-R)} T_{w-w(\frak t),R}+\sum_{(\mathfrak t',\fe') \in J_\infty(\frak t,\fe)} \delta_{\mathfrak t',\fe'}^{(-R)}  \Phi'_{(\mathfrak t';\fe'),(\mathfrak t;\fe)}, \quad \text{for all } (\frak t,\fe) \in J \setminus \emptyset,
\end{equation}
and
\begin{equation} \label{eq: delta emptyset}
\delta_\emptyset=\delta_\emptyset^{(-R)} T_{w-w(\frak t),R}+\sum_{(\mathfrak t',\fe') \in J_\infty(\emptyset)} \delta_{\mathfrak t',\fe'}^{(-R)} \Phi'_{(\mathfrak t';\fe'),\emptyset}, \quad \text{for } (\frak t,\fe)=\emptyset.
\end{equation}
In fact, for $(\frak t,\fe)=(\fs_i,\fe_i)$, the corresponding equation becomes $\delta_{\fs_i,\fe_i}=\delta_{\fs_i,\fe_i}^{(-R)}$. Thus $\delta_{\fs_i,\fe_i}=a_i(t) \in k_N[t]$. 

\subsection{Solving Eqs. \eqref{eq: delta}: statement of the result} \ppar

Letting $y$ be a variable, we denote by $v_y$ the valuation associated to the place $y$ of the field $\Fq(y)$. We put
\begin{equation} \label{eq: T and X}
T:=t-t^q, \quad X:=t^q-\theta^q. 
\end{equation}

We are ready to solve Eqs. \eqref{eq: delta}:
\begin{proposition} \label{prop: solving first equations}
We keep the above notation. Then

\noindent 1) For all $(\mathfrak t,\fe) \in J \setminus \emptyset$, the polynomial $\delta_{\frak t,\fe}$ is of the form
	\[ \delta_{\frak t,\fe}=f_{\frak t,\fe} \left(X^{m_{\frak t}}+\sum_{i=0}^{m_{\frak t}-1} P_{(\frak t,\fe),i}(T) X^i \right)\]
where
\begin{itemize}
\item $f_{\frak t,\fe} \in k_N[t]$,
\item for all $0 \leq i \leq m_{\frak t}-1$, $P_{(\frak t,\fe),i}(y)$ belongs to $\Fq(y)$ with $v_y(P_{(\frak t,\fe),i}) \geq 1$.
\end{itemize}

\smallskip
\noindent 2) For all $(\mathfrak t,\fe) \in J \setminus \emptyset$ and all $(\frak t',\fe') \in J_1(\frak t,\fe)$, there exists $F_{(\frak t,\fe),(\frak t',\fe')} \in \Fq(y)$ such that
	\[ f_{\frak t',\fe'}=f_{\frak t,\fe} F_{(\frak t,\fe),(\frak t',\fe')}(T). \]
In particular, if $f_{\frak t,\fe}=0$, then $f_{\frak t',\fe'}=0$.

\smallskip
\noindent 3) For all $(\mathfrak t,\fe) \in J \setminus \emptyset$, if  for all $(\frak t',\fe') \in J_1(\frak t,\fe)$, we set 
\begin{align*}
\xi_{\mathfrak t,\fe} &=X^{m_{\frak t}}+\sum_{i=0}^{m_{\frak t}-1} P_{(\frak t,\fe),i}(T) X^i, \\
\xi_{\mathfrak t',\fe'} &=F_{(\frak t,\fe),(\frak t',\fe')}(T) \left(X^{m_{\frak t'}}+\sum_{i=0}^{m_{\frak t'}-1} P_{(\frak t',\fe'),i}(T) X^i \right).
\end{align*}
then
\begin{equation*} 
\xi_{\mathfrak t,\fe}=\xi_{\mathfrak t,\fe}^{(-1)} (t-\theta)^{w-w(\frak t)}+\sum_{(\mathfrak t',\fe') \in J_1(\frak t;\fe)} \xi_{\mathfrak t',\fe'}^{(-1)} (t-\theta)^{w-w(\frak t)}.
\end{equation*}
\end{proposition}

This proposition states that there is a unique solution $\delta_{\frak t,\fe} \in k_N[t,\theta]$ for $(\mathfrak t,\fe) \in J \setminus \emptyset$ of Eqs. \eqref{eq: delta} up to a constant in $k_N(t)$. Further they admit a solution all belonging to $k[t,\theta]$ which satisfies $\sigma$-linear equations of order 1.

\subsection{Solving Eqs. \eqref{eq: delta}: proof of the result} \ppar

In this section we present a proof of Proposition \ref{prop: solving first equations}.

\subsubsection{Induction} 

The proof of Proposition \ref{prop: solving first equations} is by induction on $m_{\frak t}$. We start with $m_{\frak t}=0$. Then $\frak t=\fs_i$ and $\fe=\fe_i$ for some $i$. Since $\delta_{\fs_i,\fe_i}=\delta_{\fs_i,\fe_i}^{(-R)}$, we have observed that $\delta_{\fs_i,\fe_i}=a_i(t) \in k_N[t]$. Thus $\xi_{\fs_i,\fe_i}=1$ and we are done.

Suppose that Proposition \ref{prop: solving first equations} holds for all $(\frak t,\fe) \in J \setminus \emptyset$ with $m_{\frak t}<m$ for some $m \in \mathbb N$. We now prove the claim for all $(\frak t,\fe) \in J \setminus \emptyset$ with $m_{\frak t}=m$. In fact, we fix a pair $(\frak t,\fe) \in J \setminus \emptyset$ with $m_{\frak t}=m$ and want to find $\delta_{\frak t,\fe} \in \overline K[t]$ such that Eq. \eqref{eq: delta} for $(\frak t,\fe)$ is satisfied:
\begin{equation} \label{eq: delta t,e 0}
\delta_{\mathfrak t,\fe}=\delta_{\mathfrak t,\fe}^{(-R)} T_{w-w(\frak t),R}+\sum_{(\mathfrak t',\fe') \in J_\infty(\frak t,\fe)} \delta_{\mathfrak t',\fe'}^{(-R)}  \Phi'_{(\mathfrak t';\fe'),(\mathfrak t;\fe)}.
\end{equation}

The induction hypothesis implies
\begin{itemize}
\item[1)] For all $(\frak t',\fe') \in J_\infty(\frak t;\fe)$, we know that
	\[ \delta_{\frak t',\fe'}=f_{\frak t',\fe'} \left(X^{m_{\frak t'}}+\sum_{i=0}^{m_{\frak t'}-1} P_{(\frak t',\fe'),i}(T) X^i \right)\]
where
\begin{itemize}
\item $f_{\frak t',\fe'} \in k_N[t]$,
\item for all $0 \leq i \leq m_{\frak t'}-1$, $P_{(\frak t',\fe'),i}(y) \in \Fq(y)$ with $v_y(P_{(\frak t,\fe),i}) \geq 1$.
\end{itemize}

\item[2)] For all $(\mathfrak t',\fe') \in J_\infty(\frak t,\fe)$ and all $(\frak t'',\fe'') \in J_1(\frak t',\fe')$, there exists $F_{(\frak t',\fe'),(\frak t'',\fe'')} \in \Fq(y)$ such that
	\[ f_{\frak t'',\fe''}=f_{\frak t',\fe'} F_{(\frak t',\fe'),(\frak t'',\fe'')}(T). \]

\item[3)] For all $(\mathfrak t',\fe') \in J_\infty(\frak t,\fe)$ and all $(\frak t'',\fe'') \in J_1(\frak t',\fe')$, if we put
\begin{align*}
\xi_{\mathfrak t',\fe'} &=X^{m_{\frak t'}}+\sum_{i=0}^{m_{\frak t'}-1} P_{(\frak t',\fe'),i}(T) X^i, \\
\xi_{\mathfrak t'',\fe''} &=F_{(\frak t',\fe'),(\frak t'',\fe'')}(T) \left(X^{m_{\frak t''}}+\sum_{i=0}^{m_{\frak t''}-1} P_{\frak t'',i}(T) X^i \right),
\end{align*}
then $\delta_{\mathfrak t',\fe'} =f_{\frak t',\fe'} \xi_{\mathfrak t',\fe'}$, $\delta_{\mathfrak t'',\fe''} =f_{\frak t',\fe'} \xi_{\mathfrak t'',\fe''}$ and
\begin{equation} \label{eq: first order 2}
\xi_{\mathfrak t',\fe'}=\xi_{\mathfrak t',\fe'}^{(-1)} (t-\theta)^{w-w(\frak t')}+\sum_{(\mathfrak t'',\fe'') \in J_1(\frak t';\fe')} \xi_{\mathfrak t'',\fe''}^{(-1)} (t-\theta)^{w-w(\frak t')}.
\end{equation}
\end{itemize}

Hence we are reduced to solve Eq. \eqref{eq: delta t,e 0}, i.e., to find $\delta_{\frak t,\fe} \in \overline K[t]$ with the unknown parameters $f_{\frak t',\fe'} \in k_N[t]$ where $(\frak t',\fe')$ runs through the set $J_1(\frak t,\fe)$.

\subsubsection{Proof of Proposition \ref{prop: solving first equations} for $(\frak t,\fe)$: first part} \label{sec: t epsilon Part 1}

In this section we derive several useful formulas, in particular a simpler expression for the right-hand side of Eq. \eqref{eq: delta t,e 0} as a consequence of the induction hypothesis. 

We write 
\[ 
(\frak t;\fe)=(s_{i1},\dots,s_{ij};\epsilon_{i1},1,\dots,1) 
\] 
for some $i$ and $1 \leq j <\ell_i$. If $(\mathfrak t';\fe')=(s_{i1},\dots,s_{i k};\epsilon_{i1},1,\dots,1) \in J_\infty(\frak t;\fe)$ with $j < k \leq \ell_i$, we recall 
\begin{align*}
\Phi'_{(\mathfrak t';\fe'),(\mathfrak t;\fe)}=T_{w-w(\mathfrak t'),R} \sum_{0<m_{i(j+1)}<\dots<m_{ik} \leq R} T_{s_{i(j+1)},m_{i(j+1)}} \dots T_{s_{ik},m_{ik}}.
\end{align*}
If $(\mathfrak t',\fe') \in J_\infty(\frak t,\fe) \setminus J_1(\frak t,\fe)$, then we put $(\mathfrak t_0;\fe_0)=(s_{i1},\dots,s_{i (j+1)};\epsilon_{i1},1,\dots,1)$. Thus $(\mathfrak t_0;\fe_0) \in J_1(\frak t;\fe)$ and $(\mathfrak t',\fe') \in J_\infty(\frak t_0,\fe_0)$. 

\begin{lemma} \label{lem: xi higher order}
Let $(\mathfrak t';\fe') \in J_\infty(\frak t;\fe)$. Then for all $n \in \mathbb N$, we have
\begin{equation} \label{eq: xi higher order}
\xi_{\mathfrak t',\fe'}=\xi_{\mathfrak t',\fe'}^{(-n)} T_{w-w(\frak t'),n}+\sum_{(\mathfrak t'',\fe'') \in J_1(\frak t';\fe')} \sum_{\ell=1}^n  \xi_{\mathfrak t'',\fe''}^{(-\ell)} T_{w-w(\frak t'),\ell}.
\end{equation}
\end{lemma}

\begin{proof}
The proof is by induction on $n$. For $n=1$, Eq. \eqref{eq: xi higher order} is exactly Eq. \eqref{eq: first order 2} and we are done. Suppose that Eq. \eqref{eq: xi higher order} holds for some $n \in \mathbb N$. We prove that it still holds for $n+1$. Twisting Eq. \eqref{eq: xi higher order} once and putting this equality into Eq. \eqref{eq: first order 2} gets
\begin{align*} 
\xi_{\mathfrak t',\fe'} &=\xi_{\mathfrak t',\fe'}^{(-1)} (t-\theta)^{w-w(\frak t')}+\sum_{(\mathfrak t'',\fe'') \in J_1(\frak t';\fe')} \xi_{\mathfrak t'',\fe''}^{(-1)} (t-\theta)^{w-w(\frak t')} \\
&=(t-\theta)^{w-w(\frak t')} \xi_{\mathfrak t',\fe'}^{(-n-1)} T_{w-w(\frak t'),n}^{(-1)}+(t-\theta)^{w-w(\frak t')} \sum_{(\mathfrak t'',\fe'') \in J_1(\frak t';\fe')} \sum_{\ell=1}^n  \xi_{\mathfrak t'',\fe''}^{(-\ell-1)} T_{w-w(\frak t'),\ell}^{(-1)} \\
& \quad +\sum_{(\mathfrak t'',\fe'') \in J_1(\frak t';\fe')} \xi_{\mathfrak t'',\fe''}^{(-1)} (t-\theta)^{w-w(\frak t')} \\
&=T_{w-w(\frak t'),n+1} \xi_{\mathfrak t',\fe'}^{(-n-1)} + \sum_{(\mathfrak t'',\fe'') \in J_1(\frak t';\fe')} \sum_{\ell=1}^n  \xi_{\mathfrak t'',\fe''}^{(-\ell-1)} T_{w-w(\frak t'),\ell+1} \\
& \quad +\sum_{(\mathfrak t'',\fe'') \in J_1(\frak t';\fe')} \xi_{\mathfrak t'',\fe''}^{(-1)} (t-\theta)^{w-w(\frak t')} \\
&=T_{w-w(\frak t'),n+1} \xi_{\mathfrak t',\fe'}^{(-n-1)} + \sum_{(\mathfrak t'',\fe'') \in J_1(\frak t';\fe')} \sum_{\ell=1}^{n+1}  \xi_{\mathfrak t'',\fe''}^{(-\ell-1)} T_{w-w(\frak t'),\ell+1}.
\end{align*}
The proof is complete.
\end{proof}

\begin{lemma} \label{lem: simplification}
We keep the above notation. Let $(\mathfrak t';\fe')=(s_{i1},\dots,s_{i k};\epsilon_{i1},1,\dots,1) \in J_\infty(\frak t;\fe)$ with $j < k \leq \ell_i$. Then for all $n \in \mathbb N$, we have
\begin{align*}
& \sum  \xi_{\mathfrak t'',\fe''}^{(-n)}  T_{w-w(\mathfrak t''),n} \sum_{0<m_{i(k+1)}<\dots<m_{i \ell} \leq n} T_{s_{i(k+1)},m_{i(k+1)}} \dots T_{s_{i \ell},m_{i \ell}} \\
& \quad = \xi_{\mathfrak t',\fe'}- \xi_{\mathfrak t',\fe'}^{(-n)}  T_{w-w(\frak t'),n}
\end{align*}
where the first sum runs through the set of all $(\mathfrak t'',\fe'')=(s_{i1},\dots,s_{i \ell};\epsilon_{i1},1,\dots,1) \in J_\infty(\frak t',\fe')$.
\end{lemma}

\begin{proof}
Letting $n \in \mathbb N$ and $(\mathfrak t';\fe')=(s_{i1},\dots,s_{i k};\epsilon_{i1},1,\dots,1) \in J_\infty(\frak t;\fe)$ with $j < k \leq \ell_i$, we put
\begin{align*}
\mathcal S(\frak t',\fe',n):=\sum  \xi_{\mathfrak t'',\fe''}^{(-n)}  T_{w-w(\mathfrak t''),n} \sum_{0<m_{i(k+1)}<\dots<m_{i \ell} \leq n} T_{s_{i(k+1)},m_{i(k+1)}} \dots T_{s_{i \ell},m_{i \ell}}
\end{align*}
where the first sum runs through the set of all $(\mathfrak t'',\fe'')=(s_{i1},\dots,s_{i \ell};\epsilon_{i1},1,\dots,1) \in J_\infty(\frak t',\fe')$. 

We want to show that for all $n \in \mathbb N$ and all $(\mathfrak t';\fe') \in J_\infty(\frak t;\fe)$,
\begin{equation} \label{eq: Hn}
\mathcal S(\frak t',\fe',n)= \xi_{\mathfrak t',\fe'}- \xi_{\mathfrak t',\fe'}^{(-n)}  T_{w-w(\frak t'),n}.
\end{equation}

The proof is by induction on $n \in \mathbb N$. For $n=1$ we are done by Eq. \eqref{eq: first order 2} applied to $(\mathfrak t';\fe') \in J_\infty(\frak t;\fe)$. Let $n \in \mathbb N$. Suppose that for all $r<n$ and all $(\mathfrak t';\fe') \in J_\infty(\frak t;\fe)$, Eq. \eqref{eq: Hn} holds. We claim that for all $(\mathfrak t';\fe') \in J_\infty(\frak t;\fe)$,
\begin{equation*}
\mathcal S(\frak t',\fe',n)= \xi_{\mathfrak t',\fe'}- \xi_{\mathfrak t',\fe'}^{(-n)}  T_{w-w(\frak t'),n}.
\end{equation*}

In fact, let $(\mathfrak t';\fe')=(s_{i1},\dots,s_{i k};\epsilon_{i1},1,\dots,1) \in J_\infty(\frak t;\fe)$ with $j < k \leq \ell_i$. We consider the following partition of $J_\infty(\frak t',\fe')$
\begin{equation*}
J_\infty(\frak t',\fe') = J_1(\frak t',\fe') \ \cup \ \bigcup_{ (\frak t_0,\fe_0) \in J_1(\frak t',\fe')} J_\infty(\frak t_0,\fe_0).
\end{equation*}
and express $\mathcal S(\frak t',\fe',n)$ as a sum of terms induced by this partition:
\begin{equation*}
\mathcal S(\frak t',\fe',n)= \sum_{(\mathfrak t_0,\fe_0) \in J_1(\frak t',\fe')} + \sum_{(\mathfrak t_0,\fe_0) \in J_1(\frak t',\fe')} \sum_{(\mathfrak t'',\fe'') \in J_\infty(\frak t_0,\fe_0)}.
\end{equation*}

In the following we fix $(\frak t_0,\fe_0)=(s_{i1},\dots,s_{i (k+1)};\epsilon_{i1},1,\dots,1) \in J_1(\frak t',\fe')$ and consider the sums running through the set of all $(\mathfrak t'',\fe'')=(s_{i1},\dots,s_{i \ell};\epsilon_{i1},1,\dots,1) \in J_\infty(\frak t_0,\fe_0)$:
\begin{align*}
&\sum_{(\mathfrak t'',\fe'') \in J_\infty(\frak t_0,\fe_0)} \xi_{\mathfrak t'',\fe''}^{(-n)} T_{w-w(\mathfrak t''),n} \sum_{0<m_{i(k+1)}<\dots<m_{i\ell} \leq n} T_{s_{i(k+1)},m_{i(k+1)}} \dots T_{s_{i\ell},m_{i\ell}} \\
&=\sum_{(\mathfrak t'',\fe'') \in J_\infty(\frak t_0,\fe_0)} \sum_{0<m_{i(k+1)} < n}  \xi_{\mathfrak t'',\fe''}^{(-n)} T_{w-w(\mathfrak t''),n} T_{s_{i(k+1)},m_{i(k+1)}} \sum_{m_{i(k+1)}<\dots<m_{i\ell} \leq n} T_{s_{i(k+2)},m_{i(k+2)}} \dots T_{s_{i\ell},m_{i\ell}} \\
&= \sum_{(\mathfrak t'',\fe'') \in J_\infty(\frak t_0,\fe_0)} \sum_{0<m_{i(k+1)} < n}  \xi_{\mathfrak t'',\fe''}^{(-n)} T_{s_{i(k+1)},m_{i(k+1)}} (T_{w-w(\mathfrak t''),m_{i(k+1)}} (T_{w-w(\mathfrak t''),n-m_{i(k+1)}})^{(-m_{i(k+1)})}) \\
& \quad \times \sum_{m_{i(k+1)}<\dots<m_{i\ell} \leq n} (T_{s_{i(k+2)},m_{i(k+1)}} (T_{s_{i(k+2)},m_{i(k+2)}-m_{i(k+1)}})^{(-m_{i(k+1)})}) \dots (T_{s_{i\ell},m_{i(k+1)}} (T_{s_{i\ell},m_{i\ell}-m_{i(k+1)}})^{(-m_{i(k+1)})}) \\
&= \sum_{(\mathfrak t'',\fe'') \in J_\infty(\frak t_0,\fe_0)} \sum_{0<m_{i(k+1)} < n}  \xi_{\mathfrak t'',\fe''}^{(-n)} T_{w-w(\frak t'),m_{i(k+1)}} (T_{w-w(\mathfrak t''),n-m_{i(k+1)}})^{(-m_{i(k+1)})} \\
& \quad \times \sum_{0<m_{i(k+2)}''<\dots<m_{i\ell}'' \leq n-m_{i(k+1)}} (T_{s_{i(k+2)},m_{i(k+2)}''})^{(-m_{i(k+1)})} \dots (T_{s_{i\ell},m_{i\ell}''})^{(-m_{i(k+1)})} \\
&=\sum_{(\mathfrak t'',\fe'') \in J_\infty(\frak t_0,\fe_0)} \sum_{0<m_{i(k+1)} < n} T_{w-w(\frak t'),m_{i(k+1)}} \\
& \quad \times \left( \xi_{\mathfrak t'',\fe''}^{(-(n-m_{i(k+1)}))}  T_{w-w(\mathfrak t''),n-m_{i(k+1)}} \sum_{0<m_{i(k+2)}''<\dots<m_{i\ell}'' \leq n-m_{i(k+1)}} T_{s_{i(k+2)},m_{i(k+2)}''} \dots T_{s_{i\ell},m_{i\ell}''} \right)^{(-m_{i(k+1)})} \\
&= \sum_{0<m_{i(k+1)} < n} T_{w-w(\frak t'),m_{i(k+1)}}  \mathcal S(\frak t_0,\fe_0,n-m_{i(k+1)})^{(-m_{i(k+1)})}.
\end{align*}
The induction hypothesis implies
\begin{align*}
&\sum_{(\mathfrak t'',\fe'') \in J_\infty(\frak t_0,\fe_0)} \xi_{\mathfrak t'',\fe''}^{(-n)} T_{w-w(\mathfrak t''),n} \sum_{0<m_{i(k+1)}<\dots<m_{i\ell} \leq n} T_{s_{i(k+1)},m_{i(k+1)}} \dots T_{s_{i\ell},m_{i\ell}} \\
&= \sum_{0<m < n} T_{w-w(\frak t'),m}  \mathcal S(\frak t_0,\fe_0,n-m)^{(-m)} \\
&= \sum_{0<m < n} T_{w-w(\frak t'),m}  (\xi_{\mathfrak t_0,\fe_0}- \xi_{\mathfrak t_0,\fe_0}^{(-m)}  T_{w-w(\frak t_0),n-m})^{(-m)} \\
&= \sum_{0<m < n} \left(T_{w-w(\frak t'),m}  \xi_{\mathfrak t_0,\fe_0}^{(-m)}- T_{s_{i (k+1)},m} T_{w-w(\frak t_0),n} \xi_{\mathfrak t_0,\fe_0}^{(-n)} \right).
\end{align*}

Putting altogether, we obtain
\begin{align*}
& \mathcal S(\frak t',\fe',n) \\
&= \sum_{(\mathfrak t_0,\fe_0) \in J_1(\frak t',\fe')} \xi_{\mathfrak t_0,\fe_0}^{(-n)}  T_{w-w(\mathfrak t_0),n} \sum_{0<m \leq n} T_{s_{i(k+1)},m} \\
&\quad + \sum_{(\mathfrak t_0,\fe_0) \in J_1(\frak t',\fe')} \sum_{0<m < n} \left(T_{w-w(\frak t'),m}  \xi_{\mathfrak t_0,\fe_0}^{(-m)}- T_{s_{i (k+1)},m} T_{w-w(\frak t_0),n} \xi_{\mathfrak t_0,\fe_0}^{(-n)} \right) \\
&= \sum_{(\mathfrak t_0,\fe_0) \in J_1(\frak t',\fe')} \xi_{\mathfrak t_0,\fe_0}^{(-n)}  T_{w-w(\mathfrak t_0),n} \sum_{0<m \leq n} T_{s_{i(k+1)},m} \\
&\quad + \sum_{(\mathfrak t_0,\fe_0) \in J_1(\frak t',\fe')} \sum_{0<m < n} \left(T_{w-w(\frak t'),m}  \xi_{\mathfrak t_0,\fe_0}^{(-m)}- T_{s_{i (k+1)},m} T_{w-w(\frak t_0),n} \xi_{\mathfrak t_0,\fe_0}^{(-n)} \right) \\
&= \sum_{(\mathfrak t_0,\fe_0) \in J_1(\frak t',\fe')} \sum_{m=1}^n \xi_{\mathfrak t_0,\fe_0}^{(-m)} T_{w-w(\frak t'),m}.
\end{align*}
By Lemma \ref{lem: xi higher order} we get the desired equality
\begin{align*}
\mathcal S(\frak t',\fe',n) = \sum_{(\mathfrak t_0,\fe_0) \in J_1(\frak t',\fe')} \sum_{m=1}^n \xi_{\mathfrak t_0,\fe_0}^{(-m)} T_{w-w(\frak t'),m} = \xi_{\mathfrak t',\fe'}- \xi_{\mathfrak t',\fe'}^{(-n)}  T_{w-w(\frak t'),n}.
\end{align*}
\end{proof}

\begin{lemma} \label{lem: first order}
We keep the above notation. Letting $(\mathfrak t_0;\fe_0) \in J_1(\frak t;\fe)$ we put
\begin{align*}
S(\frak t_0,\fe_0)&:=\sum_{(\mathfrak t',\fe') \in \{(\mathfrak t_0;\fe_0)\} \cup J_\infty(\frak t_0,\fe_0)} \xi_{\mathfrak t',\fe'}^{(-R)}  \Phi'_{(\mathfrak t';\fe'),(\mathfrak t;\fe)}.
\end{align*} 
Then
\begin{equation*} 
S(\frak t_0,\fe_0)=\sum_{m=1}^R \xi_{\mathfrak t_0,\fe_0}^{(-m)} T_{w-w(\frak t),m}.
\end{equation*} 
\end{lemma}

\begin{proof}
We have
\begin{align*}
S(\frak t_0,\fe_0) &= \sum_{(\mathfrak t',\fe') \in \{(\mathfrak t_0;\fe_0)\} \cup J_\infty(\frak t_0,\fe_0)} \xi_{\mathfrak t',\fe'}^{(-R)}  \Phi'_{(\mathfrak t';\fe'),(\mathfrak t;\fe)} \\
&=\sum_{(\mathfrak t',\fe') \in \{(\mathfrak t_0;\fe_0)\} \cup J_\infty(\frak t_0,\fe_0)} \xi_{\mathfrak t',\fe'}^{(-R)} T_{w-w(\mathfrak t'),R} \sum_{0<m_{i(j+1)}<\dots<m_{ik} \leq R} T_{s_{i(j+1)},m_{i(j+1)}} \dots T_{s_{ik},m_{ik}}.
\end{align*}
By direct calculations as in the proof of Lemma \ref{lem: simplification} we get
\begin{align*}
& S(\frak t_0,\fe_0)=\sum_{(\mathfrak t',\fe') \in \{(\mathfrak t_0;\fe_0)\} \cup J_\infty(\frak t_0,\fe_0)} \sum_{0<m_{i(j+1)} < R} T_{w-w(\frak t),m_{i(j+1)}} \\
& \quad \times \left( \xi_{\mathfrak t',\fe'}^{(-(R-m_{i(j+1)}))}  T_{w-w(\mathfrak t'),R-m_{i(j+1)}} \sum_{0<m_{i(j+2)}'<\dots<m_{ik}' \leq R-m_{i(j+1)}} T_{s_{i(j+2)},m_{i(j+2)}'} \dots T_{s_{ik},m_{ik}'} \right)^{(-m_{i(j+1)})}.
\end{align*}

We use the decomposition
\[ \sum_{(\mathfrak t',\fe') \in \{(\mathfrak t_0;\fe_0)\} \cup J_\infty(\frak t_0,\fe_0)}=\sum_{(\mathfrak t',\fe')=(\frak t_0,\fe_0)} +\sum_{(\mathfrak t',\fe') \in J_\infty(\frak t_0,\fe_0)} \]
to express 
\[
S(\frak t_0,\fe_0)=S_1(\frak t_0,\fe_0)+S_2(\frak t_0,\fe_0). 
\]
We analyze each term of the previous sum. For the first term, we get
\begin{align*}
S_1(\frak t_0,\fe_0) &=\xi_{\frak t_0, \fe_0}^{(-R)} T_{w-w(\mathfrak t_0),R} \sum_{0<m_{i(j+1)} \leq R} T_{s_{i(j+1)},m_{i(j+1)}} \\
&= \sum_{0<m_{i(j+1)} < R} \xi_{\frak t_0,\fe_0}^{(-R)} T_{w-w(\mathfrak t_0),R} T_{s_{i(j+1)},m_{i(j+1)}} +  \xi_{\frak t_0,\fe_0}^{(-R)} T_{w-w(\mathfrak t),R}.
\end{align*}
For the second term,
\begin{align*}
& S_2(\frak t_0,\fe_0)\\
&=\sum_{(\mathfrak t',\fe') \in J_\infty(\frak t_0,\fe_0)} \sum_{0<m_{i(j+1)} < R} T_{w-w(\frak t),m_{i(j+1)}} \\
& \quad \times \left( \xi_{\mathfrak t',\fe'}^{(-(R-m_{i(j+1)}))}  T_{w-w(\mathfrak t'),R-m_{i(j+1)}} \sum_{0<m_{i(j+2)}'<\dots<m_{ik}' \leq R-m_{i(j+1)}} T_{s_{i(j+2)},m_{i(j+2)}'} \dots T_{s_{ik},m_{ik}} \right)^{(-m_{i(j+1)})} \\
&=\sum_{0<m_{i(j+1)} < R} T_{w-w(\frak t),m_{i(j+1)}} \\
& \quad \times \left( \sum_{(\mathfrak t',\fe') \in J_\infty(\frak t_0,\fe_0)}  \xi_{\mathfrak t',\fe'}^{(-(R-m_{i(j+1)}))}  T_{w-w(\mathfrak t'),R-m_{i(j+1)}} 
\hspace{-2em} 
\sum_{0<m_{i(j+2)}'<\dots<m_{ik}' \leq R-m_{i(j+1)}} 
\hspace{-2em} 
T_{s_{i(j+2)},m_{i(j+2)}'} \dots T_{s_{ik},m_{ik}} \right)^{(-m_{i(j+1)})}.
\end{align*}
By Lemma \ref{lem: simplification} we get
\begin{align*}
S_2(\frak t_0,\fe_0)&= \sum_{0<m_{i(j+1)} < R} T_{w-w(\frak t),m_{i(j+1)}}  \left( \xi_{\mathfrak t_0,\fe_0}- \xi_{\mathfrak t_0,\fe_0}^{(-(R-m_{i(j+1)}))}  T_{w-w(\frak t_0),R-m_{i(j+1)}} \right)^{(-m_{i(j+1)})} \\
&= \sum_{0<m_{i(j+1)} < R} \left( T_{w-w(\frak t),m_{i(j+1)}} \xi_{\mathfrak t_0,\fe_0}^{(-m_{i(j+1)})} - T_{s_{i(j+1)},m_{i(j+1)}} T_{w-w(\frak t_0),R} \xi_{\mathfrak t_0,\fe_0}^{(-R)}   \right).
\end{align*}

Putting altogether, we obtain the desired equality:
\begin{align*}
S(\frak t_0,\fe_0)&=S_1(\frak t_0,\fe_0)+S_2(\frak t_0,\fe_0) \\
&=T_{w-w(\mathfrak t),R} \xi_{\mathfrak t_0,\fe_0}^{(-R)} + \sum_{0<m_{i(j+1)} < R} T_{w-w(\frak t),m_{i(j+1)}} \xi_{\mathfrak t_0,\fe_0}^{(-m_{i(j+1)})} \\
&=\sum_{m=1}^R T_{w-w(\frak t),m} \xi_{\mathfrak t_0,\fe_0}^{(-m)}.
\end{align*}
\end{proof}

\subsubsection{Proof of Proposition \ref{prop: solving first equations} for $(\frak t,\fe)$: second part} \label{sec: t epsilon Part 2}

We are ready to prove the following result:
\begin{proposition} \label{prop: solving t epsilon 1}
We keep the above notation. 

\noindent a) If $\delta_{\frak t,\fe} \in \overline K[t]$ is a solution of Eq. \eqref{eq: delta t,e 0}, then
\begin{itemize}
\item[1)] $\delta_{\frak t,\fe}$ is of the following form
	\[ \delta_{\frak t,\fe}=f_{\frak t,\fe} \left(X^{m_{\frak t}}+\sum_{i=0}^{m_{\frak t}-1} P_{(\frak t,\fe),i}(T) X^i \right)\]
where
\begin{itemize}
\item $f_{\frak t,\fe} \in k_N[t]$,
\item for all $0 \leq i \leq m_{\frak t}-1$, $P_{(\frak t,\fe),i}(y) \in \Fq(y)$ with $v_y(P_{(\frak t,\fe),i}) \geq 1$.
\end{itemize}

\item[2)] For all $(\mathfrak t',\fe') \in J_1(\frak t,\fe)$, there exists $F_{(\frak t',\fe'),(\frak t'',\fe'')} \in \Fq(y)$ such that
	\[ f_{\frak t',\fe'}=f_{\frak t,\fe} F_{(\frak t,\fe),(\frak t',\fe')}(T). \]
\end{itemize}

\noindent b) Further, if $R=1$, then there exists a solution $\delta_{\frak t,\fe} \in \overline K[t]$ of Eq. \eqref{eq: delta t,e 0}.
\end{proposition}

\begin{proof}
For Part a), we recall $T=t-t^q$ and  $X=t^q-\theta^q$ given as in Eq. \eqref{eq: T and X}. For $(\frak t',\fe') \in J_1(\frak t;\fe)$, we write $\frak t'=(\frak t,(m-k)(q-1))$ with $0 \leq k < m$ and $k \not\equiv m \pmod{q}$, in particular $m_{\frak t'}=k$. We note that $\fe'$ is uniquely determined and equals $(\fe,1)$. We put $f_k=f_{\frak t',\fe'} \in k_N[t]$ and $P_{(\frak t',\fe'),j}=P_{k,j} \in \Fq(t)$ so that
\begin{equation} \label{eq: delta t,e 1}
\delta_{\frak t',\fe'}=f_k \left(X^k+\sum_{j=0}^{k-1} P_{k,j}(T) X^j \right) \in k_N[t,\theta^q].
\end{equation}
We recall
\begin{equation*} 
\xi_{\frak t',\fe'}=X^k+\sum_{j=0}^{k-1} P_{k,j}(T) X^j \in \Fq[\theta^q](t).
\end{equation*}

By Lemma \ref{lem: first order} we are reduced to find $\delta_{\mathfrak t,\fe} \in \overline K[t]$ and $f_{\frak t',\fe'} \in k_N[t]$ for $(\mathfrak t',\fe') \in J_1(\frak t;\fe)$ such that
\begin{equation} \label{eq: delta t,e}
\delta_{\mathfrak t,\fe}=\delta_{\mathfrak t,\fe}^{(-R)} T_{w-w(\frak t),R}+\sum_{(\mathfrak t',\fe') \in J_1(\frak t;\fe)} f_{\frak t',\fe'} \sum_{\ell=1}^R  \xi_{\mathfrak t',\fe'}^{(-\ell)} T_{w-w(\frak t),\ell}.
\end{equation}

By Lemma \ref{lem: KuanLin 2}, $\delta_{\frak t,\fe}$ belongs to $K_N[t]$, and
$\deg_\theta \delta_{\frak t,\fe} \leq mq$. Further, since $\delta_{\frak t,\fe}$ is divisible by $(t-\theta)^{m(q-1)}$,  we write $\delta_{\frak t,\fe}=F (t-\theta)^{m(q-1)}$ with $F \in K_N[t]$ and $\deg_\theta F \leq m$. Dividing Eq. \eqref{eq: delta t,e} by $(t-\theta)^{m(q-1)}$ and twisting $R$ times yields
\begin{align} \label{eq:F}
F^{(R)}& =F (t-\theta)^{m(q-1)} ((t-\theta)^{m(q-1)})^{(R-1)} \dots ((t-\theta)^{m(q-1)})^{(1)} \\
& \quad +\sum_{(\mathfrak t',\fe') \in J_1(\frak t;\fe)} f_{\frak t',\fe'} \sum_{\ell=1}^R  \xi_{\mathfrak t',\fe'}^{(R-\ell)} ((t-\theta)^{m(q-1)})^{(R-1)} \dots ((t-\theta)^{m(q-1)})^{(R-\ell+1)}. \notag
\end{align}
As $\xi_{\frak t',\fe'} \in k_N[\theta^q](t)$ for all $(\frak t',\fe') \in J_1(\frak t;\fe)$, it follows that $F (t-\theta)^{m(q-1)} \in k_N[t,\theta^q]$. As $\deg_\theta F \leq m$, we get
	\[ F=\sum_{0 \leq i \leq m/q} f_{m-iq} (t-\theta)^{m-iq}, \quad \text{with $f_{m-iq} \in k_N[t]$}. \]
Thus
\begin{align*}
& F (t-\theta)^{m(q-1)} =\sum_{0 \leq i \leq m/q} f_{m-iq} (t-\theta)^{mq-iq}=\sum_{0 \leq i \leq m/q} f_{m-iq} X^{m-i}, \\
& F^{(R)} =\sum_{0 \leq i \leq m/q} f_{m-iq} (t-\theta^{q^R})^{m-iq}=\sum_{0 \leq i \leq m/q} f_{m-iq} (T+T^q+\dots+T^{q^{R-1}}+X^{q^{R-1}})^{m-iq}.
\end{align*}
For all $j \geq 1$, we write
\begin{align*}
(t-\theta)^{(j)}&=t-\theta^{q^j}=T+T^q+\dots+T^{q^{j-1}}+X^{q^{j-1}}, \\
X^{(j-1)}&=t^q-\theta^{q^j}=T^q+\dots+T^{q^{j-1}}+X^{q^{j-1}}.
\end{align*}

Putting these equalities and Eq. \eqref{eq: delta t,e 1} into Eq. \eqref{eq:F} gets
\begin{align} \label{eq: delta t,e 2}
& \sum_{0 \leq i \leq m/q} f_{m-iq} (T+T^q+\dots+T^{q^{R-1}}+X^{q^{R-1}})^{m-iq}\\
&=\sum_{0 \leq i \leq m/q} f_{m-iq} X^{m-i} \prod_{j=1}^{R-1} (T+T^q+\dots+T^{q^{j-1}}+X^{q^{j-1}})^{m(q-1)} \notag \\
& +\sum_{\substack{0 \leq k<m \\ k \not\equiv m \pmod{q}}} f_k \sum_{\ell=1}^R \left(X^k+\sum_{j=0}^{k-1} P_{k,j}(T) X^j \right)^{(R-\ell)} \prod_{j=R-\ell+1}^{R-1} (T+T^q+\dots+T^{q^{j-1}}+X^{q^{j-1}})^{m(q-1)}. \notag
\end{align}

Reducing both sides $\pmod{T}$ yields
\begin{align*}
& \sum_{0 < i \leq m/q} f_{m-iq} X^{(m-iq) q^{R-1}}=\sum_{0 < i \leq m/q} f_{m-iq} X^{m-i} \prod_{j=1}^{R-1} X^{m(q-1)q^{j-1}} \\
& \quad +\sum_{\substack{0 \leq k<m \\ k \not\equiv m \pmod{q}}} f_k \sum_{\ell=1}^R X^{k q^{R-\ell}} \prod_{j=R-\ell+1}^{R-1} X^{m(q-1)q^{j-1}} \pmod{T},
\end{align*}
or equivalently
\begin{align*}
& \sum_{0 < i \leq m/q} f_{m-iq} X^{(m-iq) q^{R-1}}=\sum_{0 < i \leq m/q} f_{m-iq} X^{mq^{R-1}-i}  \\
& \quad +\sum_{\substack{0 \leq k<m \\ k \not\equiv m \pmod{q}}} f_k \sum_{\ell=1}^R X^{k q^{R-\ell}+mq^{R-1}-m q^{R-\ell}} \pmod{T},
\end{align*}

Comparing the coefficients of powers of $X^{q^{R-1}}$ of Eq. \eqref{eq: delta t,e 2} yields the following linear system in the variables $f_0,\dots,f_{m-1}$:
\begin{equation} \label{eq: delta t,e 2b}
B_{\big| y=T} \begin{pmatrix}
f_{m-1} \\
\vdots \\
f_0
\end{pmatrix}=f_m \begin{pmatrix}
Q_{m-1} \\
\vdots \\
Q_0
\end{pmatrix}_{\big| y=T}.
\end{equation} 
Here for $0 \leq i \leq m-1$, $Q_i \in y\Fq[y]$ and $B=(B_{ij})_{0 \leq i,j \leq m-1} \in \Mat_m(\Fq(y))$ such that
\begin{itemize}
\item $v_y(B_{ij}) \geq 1$ if $i>j$,
\item $v_y(B_{ij}) \geq 0$ if $i<j$,
\item $v_y(B_{ii})=0$ as $B_{ii}=\pm 1$.
\end{itemize}
The above properties follow from the fact that $P_{k,i} \in \Fq(y)$ and $v_y(P_{k,i}) \geq 1$.
Thus $v_y(\det B)=0$ so that $\det B \neq 0$. It follows that for all $0 \leq i \leq m-1$, $f_i=f_m P_i(T)$ with $P_i \in \Fq(y)$ and $v_y(P_i) \geq 1$. Thus Part a) is proved.

To conclude, we prove Part b). We want to show that if $R=1$, then there exists a solution $\delta_{\frak t,\fe} \in \overline K[t]$ of Eq. \eqref{eq: delta t,e 0}. In fact, since $R=1$, we see that Eqs. \eqref{eq:F} and \eqref{eq: delta t,e 2b} are equivalent and we are done. 
\end{proof}

\subsubsection{Proof of Proposition \ref{prop: solving first equations} for $(\frak t,\fe)$: last part} \label{sec: t epsilon Part 3}

We continue the proof of Proposition \ref{prop: solving first equations} for $(\frak t,\fe)$. Recall that we want to solve Eq. \eqref{eq: delta t,e 0}, i.e., to find $\delta_{\frak t,\fe} \in \overline K[t]$ with the unknown parameters $f_{\frak t',\fe'} \in k_N[t]$ where $(\frak t',\fe')$ runs through the set $J_1(\frak t,\fe)$. By Proposition \ref{prop: solving t epsilon 1} there exists at most one solution up to a scalar in $k_N(t)$. Further, if $R=1$, then we get a stronger statement, i.e., there exists a unique solution up to a scalar in $k(t)$. 

To finish the proof of Proposition \ref{prop: solving first equations} for $(\frak t,\fe)$, it suffices to exhibit a solution of Eq. \eqref{eq: delta t,e 0}, or equivalently Eq. \eqref{eq: delta t,e}, such that 
\begin{itemize}
\item $\delta_{\frak t,\fe} \in K[t]$ and $f_{\frak t',\fe'} \in \Fq[t]$ where $(\frak t',\fe')$ runs through the set $J_1(\frak t,\fe)$,

\item The following $\sigma$-equation of order 1 holds
\begin{equation} \label{eq: delta t,e 3}
\delta_{\mathfrak t,\fe}=\delta_{\mathfrak t,\fe}^{(-1)} (t-\theta)^{w-w(\frak t)}+\sum_{(\mathfrak t',\fe') \in J_1(\frak t;\fe)} \delta_{\mathfrak t',\fe'}^{(-1)} (t-\theta)^{w-w(\frak t)}.
\end{equation}
\end{itemize}
In fact, by Proposition \ref{prop: solving t epsilon 1}, Part b), there exist $\delta_{\frak t,\fe} \in K[t]$ and $f_{\frak t',\fe'} \in \Fq[t]$ where $(\frak t',\fe')$ runs through the set $J_1(\frak t,\fe)$ such that Eq. \eqref{eq: delta t,e 3} holds. We claim that these elements $\delta_{\frak t,\fe} \in K[t]$ and $f_{\frak t',\fe'} \in \Fq[t]$ for $(\frak t',\fe') \in J_1(\frak t,\fe)$ verify Eq. \eqref{eq: delta t,e}. By similar arguments as those given in the proof of Lemma \ref{lem: xi higher order} and the fact that $f_{\frak t',\fe'}^{(-1)}=f_{\frak t',\fe'}$ as $f_{\frak t',\fe'} \in \Fq[t]$ for all $(\frak t',\fe') \in J_1(\frak t,\fe)$, we show by induction on $n \in \mathbb N$ that
\begin{equation} \label{eq: delta t,e 4}
\delta_{\mathfrak t,\fe}=\delta_{\mathfrak t,\fe}^{(-n)} T_{w-w(\frak t),n}+\sum_{(\mathfrak t',\fe') \in J_1(\frak t;\fe)} f_{\frak t',\fe'} \sum_{\ell=1}^n  \xi_{\mathfrak t',\fe'}^{(-\ell)} T_{w-w(\frak t),\ell}.
\end{equation}
We conclude by applying Eq. \eqref{eq: delta t,e 4} for $n=R$.

\subsection{Solving Eq. \eqref{eq: delta emptyset}} \ppar

Our last task is to solve Eq. \eqref{eq: delta emptyset}. We have some extra work as we may have some pairs $(\frak t',\fe') \in J_1(\emptyset)$ with the same tuple $\frak t'$. Further, there are a factor $\mu^{(-R)}$ and some powers of $\epsilon$ in the terms $\Phi'_{(\mathfrak t';\fe'),\emptyset}$ on the right-hand side of Eq. \eqref{eq: delta emptyset}. 

As there exists $\epsilon \in \Gamma_N$ such that $\epsilon_{i1}^R=\epsilon^R$ for all $i$, there exists $\mu \in \overline{\mathbb F}_q$ such that $\mu^{(R)}=\mu \epsilon^R$ and $\mu_{i1}=\mu$ for all $i$.
We use $\mu^{(-R)}=\mu/\epsilon^R$ and put $\delta:=\delta_{\emptyset,\emptyset}/\mu \in \overline K[t]$. Then we have to solve 
\begin{align*}
\epsilon^R \delta &=\delta^{(-R)} T_{w,R}+ \epsilon^R \sum_{(\mathfrak t',\fe')=(s_{i1},\dots,s_{i k};\epsilon_{i1},1,\dots,1)  \in J_\infty(\emptyset)} \delta_{\mathfrak t',\fe'}^{(-R)} T_{w-w(\mathfrak t'),R} \times \\
& \hspace{2cm} \times \sum_{0<m_{i1}<\dots<m_{ik} \leq R} \epsilon_{i1}^{-m_{i1}} T_{s_{i1},m_{i1}} \dots T_{s_{ik},m_{ik}}. \notag
\end{align*}
By Lemma \ref{lem: first order} and using the fact that $\epsilon_{i1}^R=\epsilon^R$ for all $i$ we are reduced to solve
\begin{equation} \label{eq: delta emptyset2}
\epsilon^R \delta =\delta^{(-R)} T_{w,R}+\sum_{(\mathfrak t',\fe')=(s_{i1},\epsilon_{i1}) \in J_1(\emptyset)} f_{\frak t',\fe'} \sum_{\ell=1}^R  \xi_{\mathfrak t',\fe'}^{(-\ell)} \epsilon_{i1}^{R-\ell} T_{w,\ell}.
\end{equation}

We distinguish two cases.
\subsubsection{Case 1: $q-1 \nmid w$ and we write $w=m(q-1)+r$ with $0<r<q-1$} \ppar

We denote by $I(\epsilon,R)$ the set of $\epsilon_i \in \Gamma_N$ such that $\epsilon_i^R=\epsilon^R$.

We know that for all $(\frak t',\fe') \in J_1(\emptyset)$, says $\frak t'=((m-k)(q-1)+r)$ with $0 \leq k \leq m$  and $k \not\equiv m-r \pmod{q}$, and $\fe'=\epsilon_i$ for some $\epsilon_i \in I(\epsilon,R)$. We write $(k,\epsilon_i)$ instead of $(\frak t',\fe')=((m-k)(q-1)+r,\epsilon_i)$. We know
\begin{equation} \label{eq: delta4}
\delta_{k,\epsilon_i}=\delta_{\frak t',\fe'}=f_{k,\epsilon_i} \left(X^k+\sum_{j=0}^{k-1} P_{k,j}(T) X^j \right) \in k_N[t,\theta^q]
\end{equation}
where
\begin{itemize}
\item $f_{k,\epsilon_i} \in k_N[t]$,
\item for all $0 \leq j \leq k-1$, $P_{k,j}(y)$ belongs to $\Fq(y)$ with $v_y(P_{k,j}) \geq 1$.
\end{itemize}
We put
\begin{equation*} 
\xi_{k,\epsilon_i}=X^k+\sum_{j=0}^{k-1} P_{k,j}(T) X^j \in \Fq[t,\theta^q].
\end{equation*}
We mention that the coefficients $P_{k,j}(T)$ do not depend on the character $\epsilon_i$.

By Lemma \ref{lem: KuanLin 2}, $\delta$ belongs to $K_N[t]$. We claim that $\deg_\theta \delta \leq mq$. Otherwise, we have $\deg_\theta \delta_\emptyset > mq$. Twisting Eq. \eqref{eq: delta emptyset2} $R$ times gets
\begin{equation*}
\epsilon^R \delta^{(R)}=\delta  T_{w,R}^{(R)}+\sum_{(k,\epsilon_i) \in J_1(\emptyset)} f_{k,\epsilon_i} \sum_{\ell=1}^R  \xi_{k,\epsilon_i}^{(R-\ell)} \epsilon_i^{-\ell} T_{w,\ell}^{(R)}.
\end{equation*}
As $\deg_\theta \delta > mq$, we compare the degrees of $\theta$ on both sides and obtain
	\[ q\deg_\theta \delta=\deg_\theta \delta+wq. \]
Thus $q-1 \mid w$, which is a contradiction. We conclude that $\deg_\theta \delta \leq mq$.

From Eq. \eqref{eq: delta emptyset2} we see that $\delta$ is divisible by $(t-\theta)^w$. Thus we write $\delta=F (t-\theta)^w$ with $F \in K_N[t]$ and $\deg_\theta F \leq mq-w=m-r$. Dividing Eq. \eqref{eq: delta emptyset2} by $(t-\theta)^w$ and twisting $R$ times yields
\begin{align} \label{eq:F1}
\epsilon^R F^{(R)}& =F (t-\theta)^{w} ((t-\theta)^{w})^{(R-1)} \dots ((t-\theta)^{w})^{(1)} \\
& \quad +\sum_{(k,\epsilon_i) \in J_1(\emptyset)} f_{k,\epsilon_i} \sum_{\ell=1}^R  \xi_{k,\epsilon_i}^{(R-\ell)} \epsilon_i^{R-\ell} ((t-\theta)^{w})^{(R-1)} \dots ((t-\theta)^{w})^{(R-\ell+1)}. \notag
\end{align}

Since $\xi_{k,\epsilon_i} \in \Fq[t,\theta^q]$  for all $(k,\epsilon_i) \in J_1(\emptyset)$, it follows that $F (t-\theta)^w \in k_N[t,\theta^q]$. As $\deg_\theta F \leq m-r$, we write
	\[ F=\sum_{0 \leq i \leq (m-r)/q} f_{m-r-iq} (t-\theta)^{m-r-iq}, \quad \text{for $f_{m-r-iq} \in k_N[t]$}. \]
It follows that
\begin{align*}
& F (t-\theta)^w =\sum_{0 \leq i \leq (m-r)/q} f_{m-r-iq} (t-\theta)^{mq-iq}=\sum_{0 \leq i \leq (m-r)/q} f_{m-r-iq} X^{m-i}, \\
& F^{(R)} =\sum_{0 \leq i \leq (m-r)/q} f_{m-r-iq} (t-\theta^{q^R})^{m-r-iq}=\sum_{0 \leq i \leq (m-r)/q} f_{m-r-iq} (T+T^q+\dots+T^{q^{R-1}}+X^{q^{R-1}})^{m-r-iq}.
\end{align*}
Combining the previous equations with Eq. \eqref{eq: delta4} into Eq. \eqref{eq:F1} implies
\begin{align*}
& \epsilon^R \sum_{0 \leq i \leq (m-r)/q} f_{m-r-iq} (T+T^q+\dots+T^{q^{R-1}}+X^{q^{R-1}})^{m-r-iq} \\
&=\sum_{0 \leq i \leq (m-r)/q} f_{m-r-iq} X^{m-i} \prod_{j=1}^{R-1} (T+T^q+\dots+T^{q^{j-1}}+X^{q^{j-1}})^{m(q-1)+r} \\
& \quad +\sum_{\substack{0 \leq k<m \\ k \not\equiv m-r \pmod{q}}} \sum_{\epsilon_i \in I(\epsilon,R)} f_{k,\epsilon_i} \sum_{\ell=1}^R \left(X^k+\sum_{j=0}^{k-1} P_{k,j}(T) X^j \right)^{(R-\ell)} \epsilon_i^{R-\ell}\times  \\
& \quad  \quad \times \prod_{j=R-\ell+1}^{R-1} (T+T^q+\dots+T^{q^{j-1}}+X^{q^{j-1}})^{m(q-1)+r}.
\end{align*}

Taking $\pmod{T}$ yields
\begin{align} \label{eq: delta 10}
& \epsilon^R \sum_{0 \leq i \leq (m-r)/q} f_{m-r-iq} X^{q^{R-1}(m-r-iq)} \\ \notag
&=\sum_{0 \leq i \leq (m-r)/q} f_{m-r-iq} X^{m-i} \prod_{j=1}^{R-1} X^{q^{j-1}(m(q-1)+r)} \\ \notag
& \quad +\sum_{\substack{0 \leq k<m \\ k \not\equiv m-r \pmod{q}}} \sum_{\epsilon_i \in I(\epsilon,R)} f_{k,\epsilon_i} \sum_{\ell=1}^R X^{kq^{R-\ell}} \epsilon_i^{R-\ell} \prod_{j=R-\ell+1}^{R-1} X^{q^{j-1}(m(q-1)+r)} \\ \notag
&=\sum_{0 \leq i \leq (m-r)/q} f_{m-r-iq} X^{mq^{R-1}-i+r \frac{q^{R-1}-1}{q-1}} \\ \notag
& \quad +\sum_{\substack{0 \leq k<m \\ k \not\equiv m-r \pmod{q}}} \sum_{\epsilon_i \in I(\epsilon,R)} f_{k,\epsilon_i} \sum_{\ell=1}^R \epsilon_i^{R-\ell} X^{mq^{R-1}-mq^{R-\ell}+kq^{R-\ell}+r \frac{q^{R-1}-q^{R-\ell}}{q-1}}.
\end{align}

We set $R_0:=|I(\epsilon,R)| \leq R$. We consider the following sets
\begin{align*}
\mathcal S_1 &=\{X^{q^{R-1}(m-r-iq)}: 0 \leq i \leq (m-r)/q \}, \\
\mathcal S_2 &=\{X^{mq^{R-1}-mq^{R-\ell}+kq^{R-\ell}+r \frac{q^{R-1}-q^{R-\ell}}{q-1}}: 0 \leq k<m, k \not\equiv m-r \pmod{q}; 1 \leq \ell \leq R\} \\
\mathcal S_{2,R_0} &=\{X^{mq^{R-1}-mq^{R-\ell}+kq^{R-\ell}+r \frac{q^{R-1}-q^{R-\ell}}{q-1}}: 0 \leq k<m, k \not\equiv m-r \pmod{q}; 1 \leq \ell \leq R_0\} \\
\mathcal S &=\mathcal S_1 \cup \mathcal S_{2,R_0}.
\end{align*}
Then the elements in $\mathcal S_1 \cup \mathcal S_2$ are pairwise distinct. Thus the cardinality of the set $\mathcal S$ equals the number of variables 
\begin{align*}
\mathcal V &=\{f_{m-r-iq}: 0 \leq i \leq (m-r)/q \} \\
& \quad \cup \{f_{k,\epsilon_i}: 0 \leq k<m, k \not\equiv m-r \pmod{q}; \epsilon_i \in I(\epsilon,R)\}.
\end{align*}

Comparing the coefficients of $X^k$ with $X^k \in \mathcal S$ yields the following linear system in the variables $f_* \in \mathcal V$:
\begin{equation*}
B_{\big| y=T} (f_*)_{|\mathcal V| \times 1}=0.
\end{equation*}
Here $B=(B_{ij})_{0 \leq i,j \leq m} \in \Mat_{|\mathcal V|}(k_N(y))$. 

We claim that $B \pmod{T}$ is invertible. In fact, it suffices to prove the following statement: suppose that 
\begin{align*}
& \epsilon^R \sum_{0 \leq i \leq (m-r)/q} f_{m-r-iq} X^{q^{R-1}(m-r-iq)} =\sum_{0 \leq i \leq (m-r)/q} f_{m-r-iq} X^{mq^{R-1}-i+r \frac{q^{R-1}-1}{q-1}} \\ 
& \hspace{2cm} +\sum_{\substack{0 \leq k<m \\ k \not\equiv m-r \pmod{q}}} \sum_{\epsilon_i \in I(\epsilon,R)} f_{k,\epsilon_i} \sum_{\ell=1}^{R_0} \epsilon_i^{R-\ell} X^{mq^{R-1}-mq^{R-\ell}+kq^{R-\ell}+r \frac{q^{R-1}-q^{R-\ell}}{q-1}}.
\end{align*}
Then we want to prove that $f_*=0$ for all $f_* \in \mathcal V$. In fact, since the elements in $\mathcal S_1 \cup \mathcal S_2$ are pairwise distinct, by descending induction on $i$ we show that $f_{m-r-iq} = 0$ for all $0 \leq i \leq (m-r)/q$. Therefore, 
\begin{align*}
0=\sum_{\substack{0 \leq k<m \\ k \not\equiv m-r \pmod{q}}} \sum_{\epsilon_i \in I(\epsilon,R)} f_{k,\epsilon_i} \sum_{\ell=1}^{R_0} \epsilon_i^{R-\ell} X^{mq^{R-1}-mq^{R-\ell}+kq^{R-\ell}+r \frac{q^{R-1}-q^{R-\ell}}{q-1}}.
\end{align*}
Again, by the fact that the elements in $\mathcal S_2$ are pairwise distinct and using a variant of Vandermonde's matrix, we conclude that $f_{k,\epsilon_i}=0$ for all pairs $(k,\epsilon_i)$.

As $B \pmod{T}$ is invertible, $B$ is invertible and $f_*=0$ for all $f_* \in \mathcal V$. It follows that $\delta_{\emptyset;\emptyset}=0$ and $\delta_{\frak t',\fe'}=0$ for all $(\frak t';\fe') \in J_1(\emptyset)$. We conclude that $\delta_{\frak t,\fe}=0$ for all $(\frak t,\fe) \in J$. In particular, for all $i$, $a_i(t)=\delta_{\fs_i,\fe_i}=0$, which is a contradiction. Thus this case can never happen.

\subsubsection{Case 2: $q-1 \mid w$, says $w=m(q-1)$} \ppar

By similar arguments as above, we show that $\delta=F (t-\theta)^{m(q-1)}$ with $F \in K_N[t]$ of the form
	\[ F=\sum_{0 \leq i \leq m/q} f_{m-iq} (t-\theta)^{m-iq}, \quad \text{for $f_{m-iq} \in k_N[t]$}. \]
Thus
\begin{align*}
& F (t-\theta)^{m(q-1)} =\sum_{0 \leq i \leq m/q} f_{m-iq} (t-\theta)^{mq-iq}=\sum_{0 \leq i \leq m/q} f_{m-iq} X^{m-i}, \\
& F^{(R)} =\sum_{0 \leq i \leq m/q} f_{m-iq} (t-\theta^{q^R})^{m-iq}=\sum_{0 \leq i \leq m/q} f_{m-iq} (T+T^q+\dots+T^{q^{R-1}}+X^{q^{R-1}})^{m-iq}.
\end{align*}

Putting these and Eq. \eqref{eq: delta t,e 1} into Eq. \eqref{eq: delta emptyset2} gets
\begin{align} \label{eq: delta 12}
& \epsilon^R \sum_{0 \leq i \leq m/q} f_{m-iq} (T+T^q+\dots+T^{q^{R-1}}+X^{q^{R-1}})^{m-iq}\\
&=\sum_{0 \leq i \leq m/q} f_{m-iq} X^{m-i} \prod_{j=1}^{R-1} (T+T^q+\dots+T^{q^{j-1}}+X^{q^{j-1}})^{m(q-1)} \notag \\
& \quad +\sum_{\substack{0 \leq k<m \\ k \not\equiv m \pmod{q}}} \sum_{\epsilon_i \in I(\epsilon,R)} f_{k,\epsilon_i} \sum_{\ell=1}^R \left(X^k+\sum_{j=0}^{k-1} P_{k,j}(T) X^j \right)^{(R-\ell)} \epsilon_i^{R-\ell} \notag\\
& \quad \quad \times  \prod_{j=R-\ell+1}^{R-1} (T+T^q+\dots+T^{q^{j-1}}+X^{q^{j-1}})^{m(q-1)}. \notag
\end{align}

Taking $\pmod{T}$ yields
\begin{align*}
& (\epsilon^R-1)f_m X^m+\sum_{0 < i \leq m/q} f_{m-iq} X^{(m-iq) q^{R-1}}=\sum_{0 < i \leq m/q} f_{m-iq} X^{m-i} \prod_{j=1}^{R-1} X^{m(q-1)q^{j-1}} \\
& \hspace{2cm} +\sum_{\substack{0 \leq k<m \\ k \not\equiv m \pmod{q}}} \sum_{\epsilon_i \in I(\epsilon,R)} f_{k,\epsilon_i} \sum_{\ell=1}^R X^{k q^{R-\ell}} \epsilon_i^{R-\ell} \prod_{j=R-\ell+1}^{R-1} X^{m(q-1)q^{j-1}} \pmod{T},
\end{align*}
or equivalently
\begin{align*}
& (\epsilon^R-1)f_m X^m+\sum_{0 < i \leq m/q} f_{m-iq} X^{(m-iq) q^{R-1}}=\sum_{0 < i \leq m/q} f_{m-iq} X^{mq^{R-1}-i}  \\
& \quad +\sum_{\substack{0 \leq k<m \\ k \not\equiv m \pmod{q}}} \sum_{\epsilon_i \in I(\epsilon,R)} f_{k,\epsilon_i} \sum_{\ell=1}^R \epsilon_i^{R-\ell} X^{k q^{R-\ell}+mq^{R-1}-m q^{R-\ell}} \pmod{T}.
\end{align*}

We set 
\[ \mathcal V=\{f_{m-iq}: 0 < i \leq m/q\} \cup \{f_{k,\epsilon_i}: 0 \leq k<m, k \not\equiv m \pmod{q}; \epsilon_i \in I(\epsilon,R)\}, \]
By similar arguments as in the case $q-1 \nmid w$, comparing the coefficients of powers of $X$ in Eq. \eqref{eq: delta 12} yields
	\[ \epsilon^R f_m=f_m \]
and the following linear system in the variables in $\mathcal V$:
\begin{equation*}
B_{\big| y=T} (f_*)_{|\mathcal V| \times 1}=f_m ((Q_*)_{|\mathcal V| \times 1})_{\big| y=T}.
\end{equation*}
Here all $Q_* \in y\Fq[y]$ and $B=(B_{ij}) \in \Mat_{|\mathcal V|}(k_N(y))$.  We can show that $B \pmod{T}$ is invertible, hence $B$ is invertible.

We distinguish two subcases.

\medskip
\noindent {\bf Subcase 1: $\epsilon^R \neq 1$.} \ppar
It follows that $f_m=0$. Then $f_*=0$ for all $f_* \in \mathcal V$. Thus $\delta_{\frak t,\fe}=0$ for all $(\frak t,\fe) \in J$. In particular, for all $i$, $a_i(t)=\delta_{\fs_i,\fe_i}=0$. This is a contradiction and we conclude that this case can never happen.

\medskip
\noindent {\bf Subcase 2: $\epsilon^R=1$.} \ppar

From the induction hypothesis for $(R,w)$ with $R+w<N+w$, it follows that $f_{k,\epsilon_i}=0$ whenever $\epsilon_i \neq 1$. We are reduced to solve a linear system in the variables 
\[ \mathcal V'=\{f_{m-iq}: 0 < i \leq m/q\} \cup \{f_{k,1}: 0 \leq k<m, k \not\equiv m \pmod{q} \} \]
given by
\begin{equation*}
B_{\big| y=T} (f_*)_{|\mathcal V'| \times 1}=f_m ((Q_*)_{|\mathcal V'| \times 1})_{\big| y=T}.
\end{equation*}
Here $f_m \in k_N[t]$, all $Q_* \in y\Fq[y]$ and $B=(B_{ij}) \in \Mat_{|\mathcal V'|}(\Fq(y))$. We note that the coefficients of $B$ belong to $\Fq(y)$ since the characters are all trivial. We have already know that $B$ is invertible. Hence we have shown that there exists at most one solution of Eqs. \eqref{eq: delta} and \eqref{eq: delta emptyset} up to a factor in $k_N(t)$. Recall that for all $i$, $a_i(t)=\delta_{\fs_i,\fe_i}$. Therefore, up to a scalar in $K_N^\times$, there exists at most one non-trivial relation
	\[ a \widetilde \pi^w+\sum_i a_i L(\fs_i;\fe_i)(\theta)=0 \]
with $a, a_i \in K_N$ and $\Li \begin{pmatrix}
 \fe  \\
\fs  \end{pmatrix} \in \mathcal{CS}_{N,w}$. Further, we must have $\fe_i=(1,\dots,1)$ for all $i$. Hence we have proved Theorem \ref{thm: trans CMPL at roots of unity}, Part 1.

By Lemma \ref{lem: relation Carlitz period}, there exists such a linear relation with $a \neq 0$. Combining this with Theorem \ref{thm: trans CMPL at roots of unity}, Part 1, we obtain Theorem \ref{thm: trans CMPL at roots of unity}, Part 2. Hence the proof of Theorem \ref{thm: trans CMPL at roots of unity} is finished.


\section{Proof of the main theorems} \label{sec:applications} 

We now prove Theorems \ref{thm: Hoffman cyclotomic MZV} and \ref{thm: Zagier cyclotomic MZV}.

\begin{proof}[Proof of Theorem \ref{thm: Hoffman cyclotomic MZV}]
By Theorem \ref{thm: transcendence} the CMPL's at roots of unity in $\mathcal{CS}_{N,w}$ are all linearly independent over $K_N$. Then by Theorem \ref{thm: strong BD MZV} we deduce that the set  $\mathcal{CS}_{N,w}$ form a basis for $\mathcal{CL}_{N,w}$. 
\end{proof}

\begin{proof}[Proof of Theorem \ref{thm: Zagier cyclotomic MZV}]
We recall
\begin{equation*}
    d_N(w) = \begin{cases}
    		1 & \text{if } w=0, \\
		\gamma_N (\gamma_N+1)^{w-1}& \text{if } 1 \leq w < q, \\
         \gamma_N ((\gamma_N+1)^{w-1} - 1) &\text{if } w = q,
		 \end{cases}
\end{equation*}
and for $w>q$, $d_N(w)=\gamma_N \sum \limits_{i = 1}^{q-1} d_N(w-i) + d_N(w - q)$. Then we want to show that
	\[ \dim_{K_N} \mathcal{CZ}_{N,w} = d_N(w). \]

We note that if we set $d_N(w) = 0$ for $w < 0$, then the equality $d_N(w)=\gamma_N\sum \limits_{i = 1}^{q-1} d_N(w-i) + d_N(w - q)$ holds for every integer $w \neq 0, q$.

For $w \in \mathbb N$ we put $d_N'(w) = |\mathcal{CS}_{N,w}|$. We also set $d_N'(w)=1$ for $w=0$ and $d_N'(w) = 0$ for $w < 0$. Then the equality $d_N'(w)=\gamma_N\sum \limits_{i = 1}^{q-1} d_N'(w-i) + d_N'(w - q)$ holds for all $w \in \mathbb N$ and $w \neq q$. Furthermore, for $w=q$ we have $d_N'(q)=\gamma_N\sum \limits_{i = 1}^{q-1} d_N'(q-i)$. Hence it follows that for all $w \in \mathbb N$, $d_N'(w)=d_N(w)$. We combine this equality with Theorem \ref{thm: Hoffman cyclotomic MZV} to obtain Theorem \ref{thm: Zagier cyclotomic MZV}.
\end{proof}


\end{document}